\documentclass[a4paper,11pt]{article}

\usepackage[utf8]{inputenc}
\usepackage[english]{babel}
\usepackage[T1]{fontenc}
\usepackage{amssymb, xr, multicol, textcomp, amsthm, amsmath, url, enumerate, hyperref, geometry, tikz, indentfirst, xfrac}
\usepackage[labelfont=bf,margin=2cm]{caption}
\usepackage{mathtools}
\usetikzlibrary{arrows}
\hypersetup{urlbordercolor={0 0 1},linkbordercolor={0.8 0 0},pdfborderstyle={/S/U/W 1}}
\definecolor{rouge1}{rgb}{0.6,0.2,0}
\definecolor{rouge2}{rgb}{0.8,0,0}
\definecolor{gris1}{rgb}{0.25,0.25,0.25}

\definecolor{wwwwww}{rgb}{0.4,0.4,0.4}
\definecolor{qqqqzz}{rgb}{0,0,0.6}
\definecolor{qqzzqq}{rgb}{0,0.6,0}
\definecolor{zzqqqq}{rgb}{0.6,0,0}

\definecolor{qqzzzz}{rgb}{0,0.6,0.6}

\usepackage[bottom]{footmisc}
\usepackage{perpage} %the perpage package
\MakePerPage{footnote} %the perpage package command

\newtheorem{theorem}{Theorem}[section]
\newtheorem{proposition}[theorem]{Proposition}
\newtheorem{lemma}[theorem]{Lemma}
\newtheorem{corollary}[theorem]{Corollary}

\theoremstyle{definition}
\newtheorem{definition}[theorem]{Definition}

\theoremstyle{remark}
\newtheorem{remark}[theorem]{Remark}

\newtheorem*{exampleemempty}{Example}

\makeatletter
\let\c@figure\c@theorem
\makeatother

\def\go#1{{\mathfrak{#1}}}
\def\l#1{{\mathcal{#1}}} 
\def\b#1{{\mathbb{#1}}}

\def\italicegal{\hspace{-1mm}=\hspace{-1mm}}
\def\minusens{\hspace{-0.5mm}\setminus\hspace{-0.5mm}}

\def\#{\ \sharp\ }
\def\ie{\textit{ie}\ }
\def\card(#1){\vert#1\vert}

\def\i={\leqslant} 
\def\s={\geqslant}

\def\card(#1){\vert #1 \vert}

\title{Variation formulas for an extended Gompf invariant}
\author{Jean-Mathieu Magot }
\date{}
\usepackage{tocloft}

\begin{document}

\maketitle

\begin{abstract}
 In 1998, R. Gompf defined a homotopy invariant $\theta_G$ of oriented 2-plane fields in 3-manifolds.  This invariant is defined for oriented 2-plane fields $\xi$ in a closed oriented 3-manifold $M$ when the first Chern class $c_1(\xi)$ is a torsion element of $H^2(M;\b Z)$. In this article, we define an extension of the Gompf invariant for all compact oriented 3-manifolds with boundary and we study its iterated variations under Lagrangian-preserving surgeries. It follows that the extended Gompf invariant is a degree two invariant with respect to a suitable finite type invariant theory. 
\end{abstract}

\section*{Introduction}

\renewcommand{\thetheorem}{\arabic{theorem}}

\subsection*{Context}
In \cite{gompf}, R. Gompf defined a homotopy invariant $\theta_G$ of oriented 2-plane fields in 3-manifolds. This invariant is defined for oriented 2-plane fields $\xi$ in a closed oriented 3-manifold $M$ when the first Chern class $c_1(\xi)$ is a torsion element of $H^2(M;\b Z)$. This invariant appears, for instance, in the construction of an absolute grading for the Heegaard-Floer homology groups, see \cite{GH}.
Since the positive unit normal of an oriented 2-plane field of a Riemannian 3-manifold $M$ is a section of its unit tangent bundle $UM$, homotopy classes of oriented 2-plane fields of $M$ are in one-to-one correspondence with homotopy classes of sections of $UM$. Thus, the invariant $\theta_G$ may be regarded as an invariant of homotopy classes of nowhere zero vector fields, also called \textit{combings}. In that setting, the Gompf invariant is defined for \textit{torsion combings} of closed oriented 3-manifolds $M$, \ie combings $X$ such that the Euler class $e_2(X^\perp)$ of the normal bundle $X^\perp$ is a torsion element of $H^2(M;\b Z)$.  \\

In \cite{lescopcombing}, C. Lescop proposed an alternative definition of $\theta_G$ using a Pontrjagin construction from the combing viewpoint. Here, we use a similar approach to show how to define Pontrjagin numbers for torsion combings by using pseudo-parallelizations, which are a generalization of parallelizations. This enables us to define a relative extension of the Gompf invariant for torsion combings in all compact oriented 3-manifolds with boundary. We also study the iterated variations under Lagrangian-preserving surgeries of this extended invariant and prove that it is a degree two invariant with respect to a suitable finite type invariant theory. In such a study, pseudo-parallelizations reveal decisive since they are, in some sense, compatible with Lagrangian-preserving surgeries while genuine parallelizations are not.

\subsection*{Conventions}
In this article, compact oriented 3-manifolds may have boundary unless otherwise mentioned. All manifolds are implicitly equipped with Riemannian structures. The statements and the proofs are independent of the chosen Riemannian structures. \\

If $M$ is an oriented manifold and if $A$ is a submanifold of $M$, let $TM$, resp. $TA$, denote the tangent bundles to $M$, resp. $A$, and let $NA$ refer to the orthogonal bundle to $A$ in $M$, which is canonically isomorphic to the normal bundle to $A$ in $M$. The fibers of $NA$ are oriented so that $NA \oplus TA = TM$ fiberwise and the boundaries of all compact manifolds are oriented using the outward normal first convention. \\

If $A$ and $B$ are transverse submanifolds of an oriented manifold $M$, their intersection is oriented so that $N(A\cap B) = NA \oplus NB$, fiberwise. Moreover, if $A$ and $B$ have comple\-mentary dimensions, \ie if $\mbox{dim}(A)+\mbox{dim}(B)=\mbox{dim}(M)$, let $\varepsilon_{A\cap B}(x) = 1$ if $x \in A\cap B$ is such that  $T_xA\oplus T_xB = T_xM$ and $\varepsilon_{A\cap B}(x) = -1$ otherwise. If $A$ and $B$ are compact transverse submani\-folds of an oriented manifold $M$ with complementary dimensions, the \textit{algebraic intersection of $A$ and $B$ in $M$} is
$$
\langle A, B \rangle_M = \sum_{x \in A\cap B} \varepsilon_{A\cap B}(x).
$$

Let $L_1$ and $L_2$ be two rational cycles of an oriented $n$-manifold $M$. Assume that $L_1$ and $L_2$ bound two rational chains $\Sigma_1$ and $\Sigma_2$, respectively. If $L_1$ is transverse to $\Sigma_2$, if $L_2$ is transverse to $\Sigma_1$ and if $\mbox{dim}(L_1)+\mbox{dim}(L_2) = n-1$, then the \textit{linking number of $L_1$ and $L_2$ in $M$} is
$$
lk_M(L_1,L_2) = \langle \Sigma_1 , L_2 \rangle_M = (-1)^{\tiny{n-\mbox{dim}(L_2)}}\langle L_1 , \Sigma_2 \rangle_M.
$$

\subsection*{Setting and statements}
A \textit{combing} $(X,\sigma)$ of a compact oriented 3-manifold $M$ is a section $X$ of the unit tangent bundle $UM$ together with a nonvanishing section $\sigma$ of the restriction $X^\perp_{|\partial M}$ of the normal bundle $X^\perp$ to $\partial M$. For simplicity's sake, the section $\sigma$ may be omitted in the notation of a combing. For any combing $(X,\sigma)$, note that $\rho(X)=(X_{|\partial M}, \sigma, X_{|\partial M} \wedge \sigma)$, where $\wedge$ denotes the cross product, is a trivialization of $TM_{|\partial M}$. So, a combing of a compact oriented 3-manifold $M$ may also be seen as a pair $(X,\rho)$ where $X$ is a section of $UM$ that is the first vector of a trivialization $\rho$ of $TM_{|\partial M}$ together with this trivialization. \\

Two combings $(X,\sigma_X)$ and $(Y,\sigma_Y)$ of a compact oriented 3-manifold $M$ are said to be \textit{transverse} when the graph $X(M)$ is transverse to $Y(M)$ and $-Y(M)$ in $UM$. The combings $(X,\sigma_X)$ and $(Y,\sigma_Y)$ are said to be \textit{$\partial$-compatible} when $X_{|\partial M}= Y_{|\partial M}$, $\sigma_X = \sigma_Y$, $X(\mathring M)$ is transverse to $Y(\mathring M)$ and $-Y(\mathring M)$ in $UM$, and 
$$
\overline{X(\mathring M)\cap Y(\mathring M)} \cap UM_{|\partial M} =\emptyset.
$$
\iffalse
If $X(\mathring M)$ is transverse to $Y(\mathring M)$ and $-Y(\mathring M)$ in $UM$ and if 
$$
\overline{X(\mathring M)\cap Y(\mathring M)} \cap UM_{|\partial M} =\emptyset
$$
then $(X,\sigma_X)$ and $(Y,\sigma_Y)$ are \textit{almost transverse}. Two combings $(X,\sigma_X)$ and $(Y,\sigma_Y)$ of $M$ are said to be \textit{coincide on $\partial M$} if $X_{|\partial M}= Y_{|\partial M}$ and $\sigma_X=\sigma_Y$. \fi
When $(X,\sigma_X)$ and $(Y,\sigma_Y)$ are $\partial$-compatible, define two links $L_{X=Y}$ and $L_{X=-Y}$ as follows. First, let $P_M$ denote the projection from $UM$ to $M$ and set
$$
L_{X=-Y} = P_M (X(M) \cap (-Y)(M)).
$$
Second, there exists a link $L_{X=Y}$ in $\mathring{M}$ such that
$$
P_M (X(M) \cap Y(M)) = \partial M \sqcup L_{X=Y}.
$$

If $(X,\sigma)$ is a combing of a compact oriented 3-manifold $M$, its \textit{relative Euler class} $e_2^M(X^{\perp}\hspace{-1mm}, \sigma)$ in $H^2(M,\partial M;\b Z)$ is an obstruction to extending the section $\sigma$ as a nonvanishing section of $X^\perp$. This obstruction is such that its Poincaré dual $P(e_2^M(X^\perp, \sigma))$ is represented by the zero set of a generic section of $X^\perp$ extending $\sigma$. This zero set is oriented by its coorientation induced by the orientation of $X^\perp$. When $M$ is closed, the \textit{Euler class} $e_2(X^\perp)$ of $X$ is just this obstruction to finding a nonvanishing section of $X^\perp$. \\

A combing $(X,\sigma)$ of a compact oriented 3-manifold $M$ is a \textit{torsion combing} if $e_2^M(X^\perp,\sigma)$ is a torsion element of $H^2(M,\partial M ; \b Z)$, \ie $\left[e_2^M(X^\perp,\sigma)\right]=0$ in $H^2(M,\partial M ; \b Q)$. \\

Let $M_1$ and $M_2$ be two compact oriented 3-manifolds. The manifolds $M_1$ and $M_2$ are said to have \textit{identified boundaries} if a collar of $\partial M_1$ in $M_1$ and a collar of $\partial M_2$ in $M_2$ are identified. In this case, ${TM_1}_{|\partial M_1}= \b R n_1 \oplus T \partial M_1$ is naturally identified with ${TM_2}_{|\partial M_2}= \b R n_2 \oplus T \partial M_2$ by an identification that maps the outward normal vector field $n_1$ to $M_1$ to the outward normal vector field $n_2$ to $M_2$. \\

If $\tau_1$ and $\tau_2$ are parallelizations of two compact oriented 3-manifolds $M_1$ and $M_2$ with identified boundaries such that $\tau_1$ and $\tau_2$ coincide on $\partial M_1 \simeq \partial M_2$, then the \textit{first relative Pontrjagin number of $\tau_1$ and $\tau_2$} is an element $p_1(\tau_1,\tau_2)$ of $\b Z$ which corresponds to the Pontrjagin obstruction to extending a specific trivialization $\tau(\tau_1,\tau_2)$ of $TW \otimes \b C$ defined on the boundary of a cobordism $W$ from $M_1$ to $M_2$ with signature zero (see Subsection~\ref{ssec_defpara} or \cite[Subsection 4.1]{lescopcombing}). In the case of a parallelization $\tau$ of a closed oriented 3-manifold $M$, we get an absolute version. The \textit{Pontrjagin number $p_1(\tau)$ of $\tau$} is the relative Pontrjagin number $p_1(\tau_\emptyset,\tau)$ where $\tau_\emptyset$ is the parallelization of the empty set. Hence, for two parallelizations $\tau_1$ and $\tau_2$ of some closed oriented 3-manifolds,
$$
p_1(\tau_1,\tau_2) = p_1(\tau_2) - p_1(\tau_1).
$$

In \cite{lescopcombing}, using an interpretation of the variation of Pontrjagin numbers of parallelizations as an intersection of chains, C. Lescop showed that such a variation can be computed using only the first vectors of the parallelizations. This led her to the following theorem, which contains a definition of the \textit{Pontrjagin numbers} for torsion combings of closed oriented 3-manifolds.
\begin{theorem}[{\cite[Theorem 1.2 \& Subsection 4.3]{lescopcombing}}] \label{thm_defp1X}
Let $M$ be a closed oriented 3-manifold. There exists a unique map 
$$
p_1 \ : \ \lbrace \mbox{homotopy classes of torsion combings of } M \rbrace \longrightarrow \b Q
$$
such that :
\begin{enumerate}[(i)]
\item for any combing $X$ on $M$ such that $X$ extends to a parallelization $\tau$ of $M$ :
$$
p_1([X])=p_1(\tau),
$$
\item if $X$ and $Y$ are two transverse torsion combings of $M$, then 
$$
p_1([Y])-p_1([X])= 4 \cdot lk(L_{X=Y},L_{X=-Y}).
$$
\end{enumerate}
Furthermore, $p_1$ coincides with the Gompf invariant : for any torsion combing $X$,
$$
p_1([X])=\theta_G(X^\perp).
$$
\end{theorem}

In this article, we study the variations of the Pontrjagin numbers of torsion combings of compact oriented 3-manifolds with respect to specific surgeries, called \textit{Lagrangian-preserving surgeries}, which are defined as follows.  \\

A \textit{rational homology handlebody of genus $g\in \b N$}, or $\b Q$HH for short, is a compact oriented 3-manifold with the same homology with coefficients in $\b Q$ as the standard genus $g$ handlebody. Note that the boundary of a genus $g$ rational homology handlebody is homeomorphic to the standard closed connected oriented surface of genus $g$. The \textit{Lagrangian} of a $\b Q$HH $A$ is 
$$
\go L_A := \mbox{ker}\left(i^A_* : H_1(\partial A; \b Q) \longrightarrow H_1(A;\b Q)\right)
$$
where $i^A$ is the inclusion of $\partial A$ into $A$. An \textit{LP$_\b Q$-surgery datum} in a compact oriented 3-manifold $M$ is a triple $(A,B,h)$, or $(\sfrac{B}{A})$ for short, where $A \subset M$, where $B$ and $A$ are rational homology handlebodies and where $h : \partial A \rightarrow \partial B$ is an identification homeomorphism, called \textit{LP$_\b Q$-identification}, such that $h_*(\go L_A)=\go L_B$. Performing the \textit{LP$_\b Q$-surgery} associated with the datum $(A,B,h)$ in $M$ consists in constructing the manifold :
$$
M \left( \sfrac{B}{A} \right) = \left( M \setminus \mathring A \right) \ \bigcup_h \ B.
$$

If $(M,X)$ is a compact oriented 3-manifold equipped with a combing, if $(A,B,h)$ is an LP$_\b Q$-surgery datum in $M$, and if $X_{B}$ is a combing of $B$ that coincides with $X$ on $\partial A \simeq \partial B$, then $(A,B,h,X_{B})$, or $(\sfrac{B}{A},X_B)$ for short, is an \textit{LP$_\b Q$-surgery datum in $(M,X)$}. Performing the \textit{LP$_\b Q$-surgery} associated with the datum $(A,B,h,X_{B})$ in $(M,X)$ consists in constructing the manifold $M \left( \sfrac{B}{A} \right)$ equipped with the combing :
$$
X(\sfrac{B}{A}) = \left\lbrace
\begin{aligned}
& X &  &\mbox{on $M\setminus \mathring A$},  \\
& X_{B} & &\mbox{on $B$.}
\end{aligned}
\right.
$$

The main result of this article is a variation formula for Pontrjagin numbers -- see Theorem~\ref{thm_D2nd} below -- which reads as follows in the special case of compact oriented 3-manifolds without boundary.
\begin{theorem} \label{thm_D2nd0}
Let $(M,X)$ be a closed oriented 3-manifold equipped with a combing and let $\lbrace (\sfrac{B_i}{A_i},X_{B_i}) \rbrace_{i \in \lbrace 1,2 \rbrace}$ be two disjoint LP$_\b Q$-surgeries in $(M,X)$ (\ie $A_1$ and $A_2$ are disjoint). For all $I \subset \lbrace 1,2 \rbrace$, let $(M_I,X^I)$ be the combed manifold obtained by performing the surgeries associated to the data $ \lbrace (\sfrac{B_i}{A_i},X_{B_i})\rbrace_{i \in I}$. If $\lbrace X^I \rbrace_{I \subset \lbrace 1,2 \rbrace}$ is a family of torsion combings of the $\lbrace M_I \rbrace_{I\subset \lbrace 1,2 \rbrace}$, then
$$
\sum_{I\subset\lbrace 1 , 2 \rbrace} (-1)^{|I|} p_1([X^I]) = - 2 \cdot lk_M \left(L_{\lbrace X^I \rbrace}(\sfrac{B_1}{A_1}), L_{\lbrace X^I \rbrace}(\sfrac{B_2}{A_2})\right),
$$
where the right-hand side of the equality is defined as follows. For all $i \in \lbrace 1 , 2 \rbrace$, let
$$
H_1(A_i;\b Q) \stackrel{i^{A_i}_*}{\longleftarrow} \frac{H_1(\partial A_i;\b Q)}{\go L_{A_i}} \stackrel{h_*}{=} \frac{H_1(\partial B_i;\b Q)}{\go L_{B_i}}  \stackrel{i^{B_i}_*}{\longrightarrow} H_1(B_i;\b Q)
$$
be the sequence of isomorphisms induced by the inclusions $i^{A_i}$ and $i^{B_i}$. There exists a unique homology class $L_{\lbrace X^I\rbrace}(\sfrac{B_i}{A_i})$ in $H_1(A_i; \b Q)$ such that for any nonvanishing section $\sigma_i$ of $X^\perp_{|\partial A_i}$~:
$$
L_{\lbrace X^I \rbrace}(\sfrac{B_i}{A_i}) = i^{A_i}_* \circ (i^{B_i}_*)^{-1}\left(P\big(e_2^{B_i}(X_{B_i}^\perp, \sigma_i)\big)\right) - P\big(e_2^{A_i}(X^\perp_{|A_i}, \sigma_i)\big),
$$
where $P$ stands for Poincaré duality isomorphisms from $H^2(A_i,\partial A_i;\b Q)$ to $H_1(A_i;\b Q)$ or from $H^2(B_i,\partial B_i;\b Q)$ to $H_1(B_i;\b Q)$. Furthermore, the homology classes $L_{\lbrace X^I \rbrace}(\sfrac{B_1}{A_1})$ and $L_{\lbrace X^I \rbrace}(\sfrac{B_2}{A_2})$ are mapped to zero in $H_1(M;\b Q)$ and the map 
$$
lk_M : \mbox{\textup{ker}}\big(H_1(A_1;\b Q) \rightarrow H_1(M; \b Q)\big) \times \mbox{\textup{ker}}\big(H_1(A_2;\b Q) \rightarrow H_1(M; \b Q)\big) \longrightarrow \b Q
$$
is well-defined.
\end{theorem}

\begin{exampleemempty}
Consider $\b S^3$ equipped with a parallelization $\tau : \b S^3 \times \b R^3 \rightarrow T\b S^3$ which extends the standard parallelization of the unit ball. In this ball, consider a positive Hopf link and let $A_1 \sqcup A_2$ be a tubular neighborhood of this link. Let $X$ be the combing $\tau(e_1)=\tau(.,e_1)$, where $e_1=(1,0,0)\in \b S^2$, and let $B_1 = A_1$ and $B_2=A_2$. Identify $A_1$ and $A_2$ with $\b D^2 \times \b S^1$ and consider a smooth map $g : \b D^2 \rightarrow \b S^2$ such that $g(\partial \b D^2) = e_1$, and such that $-e_1$ is a degree 1 regular value of $g$ with a single preimage $\omega$. Finally, for $i \in \lbrace 1,2 \rbrace$, let $X_{B_i}$ be the combing : 
$$
X_{B_i} : 
\left\lbrace
\begin{aligned}
\b D^2 \times \b S^1 &\longrightarrow UM \\
(z,u) &\longmapsto \tau((z,u),g(u)).
\end{aligned}
\right.
$$ 
In this case, $L_{X(\lbrace \sfrac{B_i}{A_i} \rbrace_{i \in \lbrace 1,2 \rbrace})=-X} = L_{X_{B_1}=-X_{|A_1}} \cup L_{X_{B_2}=-X_{|A_2}}$, and, for $i \in \lbrace 1,2 \rbrace$, using the identification of $A_i$ with $\b D^2 \times \b S^1$, the link $L_{X_{B_i}=-X_{|A_i}}$ reads $\lbrace \omega \rbrace \times \b S^1$. As we will see in Proposition~\ref{prop_linksinhomologyI}, for $i \in \lbrace 1,2 \rbrace$, $L_{\lbrace X^I \rbrace}(\sfrac{B_i}{A_i}) = 2 [L_{X_{B_i}=-X_{|A_i}}]$. Eventually,
$$
\sum_{I\subset\lbrace 1 , 2 \rbrace} (-1)^{|I|} p_1([X^I]) = -8.
$$
\end{exampleemempty}

In general, for an LP$_\b Q$-surgery datum $(\sfrac{B}{A})$ in a compact oriented 3-manifold $M$, a trivialization of $TM_{|(M\setminus \mathring A)}$ cannot be extended as a parallelization of $M(\sfrac{B}{A})$. It follows that LP$_\b Q$-surgeries cannot be expressed as local moves on parallelized compact oriented 3-manifolds. This makes computing the variation of Pontrjagin numbers of torsion combings under LP$_\b Q$-surgeries tricky since Pontrjagin numbers of torsion combings are defined with respect to Pontrjagin numbers of parallelizations.  \\

However, if $M$ is a compact oriented 3-manifold and if $\rho$ is a trivialization of $TM_{|\partial M}$, then the obstruction to finding a parallelization of $M$ which coincides with $\rho$ on $\partial M$ is an element of $H^2(M,\partial M; \sfrac{\b Z}{2\b Z})$ -- hence, its Poincaré dual is an element $[\gamma]$ of $H_1(M;\sfrac{\b Z}{2\b Z})$ -- and it is possible to get around such an obstruction thanks to the notion of \textit{pseudo-parallelization} developed by C.~Lescop. Let us postpone the formal definition to Subsection~\ref{ssec_defppara} (see also \cite{lescopcube}) and, for the time being, let us just mention that a pseudo-parallelization $\bar\tau$ of a compact oriented 3-manifold $M$ is a triple $(N(\gamma); \tau_e, \tau_d)$ where $N(\gamma)$ is a framed tubular neighborhood of a link $\gamma$ in $\mathring M$, $\tau_e$ is a parallelization of $M\setminus N(\gamma)$ and $\tau_d : N(\gamma)\times \b R^3 \rightarrow TN(\gamma)$ is a parallelization of $N(\gamma)$ such that there exists a section $E_1^d$ of $UM$ :
$$
E_1^d :
\left\lbrace
\begin{aligned}
m \in M\setminus \mathring N(\gamma) &\longmapsto  \tau_e(m,e_1) \\
m \in N(\gamma) &\longmapsto  \tau_d(m,e_1).
\end{aligned}
\right.
$$
Let us finally mention that $\bar\tau$ also determines a section $E_1^g$ of $UM$ which coincides with $E_1^d$ on $M\setminus \mathring N(\gamma)$. The sections $E_1^d$ and $E_1^g$ are the \textit{Siamese sections} of $\bar\tau$ and the link $\gamma$ is the \textit{link of the pseudo-parallelization $\bar\tau$}. \\

To a pseudo-parallelization, C. Lescop showed that it is possible to associate a \textit{complex trivialization} up to homotopy, see Definition~\ref{def_complextriv}. This leads to a natural extension of the notion of first relative Pontrjagin numbers of parallelizations to pseudo-parallelizations. Furthermore, as in the case of parallelizations, a pseudo-parallelization $\bar\tau$ of a compact oriented 3-manifold $M$ admits \textit{pseudo-sections} $\bar\tau(M\times\lbrace v\rbrace)$ which are 3-chains of $UM$, for all $v \in \b S^2$. In the special case $v =e_1$ the pseudo-section $\bar\tau(M\times \lbrace e_1 \rbrace)$ of $\bar\tau$ can be written as :
$$
\bar\tau(M\times \lbrace e_1 \rbrace) = \frac{E_1^d(M) + E_1^g(M)}{2}.
$$
A combing $(X,\sigma)$ of $M$ is said to be \textit{compatible with $\bar\tau$} if $(X,\sigma)$ is $\partial$-compatible with $(E_1^d,{E_2^e}_{|\partial M})$ and $(E_1^g,{E_2^e}_{|\partial M})$, where $E_2^e$ is the second vector of $\tau_e$, and if 
$$
L_{E_1^d=X} \cap L_{E_1^g=-X}=\emptyset \mbox{ \ and \ } L_{E_1^g=X} \cap  L_{E_1^d=-X}=\emptyset.
$$
If $(X,\sigma)$ and $\bar\tau$ are compatible, then $\rho(X)=\bar\tau_{|\partial M}$ and we get two disjoint rational combinations of oriented links in $\mathring M$ :
$$
L_{\bar\tau =  X} = \frac{L_{E_1^d= X} + L_{E_1^g= X}}{2} \mbox{ \ and \ } L_{\bar\tau = - X} = \frac{L_{E_1^d=- X} + L_{E_1^g=- X}}{2}.
$$
Pseudo-parallelizations allow us to revisit the definition of Pontrjagin numbers and to generalize it to torsion combings of compact oriented 3-manifolds with non empty boundary as follows. Let $P_{\b S^2}$ denote the standard projection from $W\times \b S^2$ to $\b S^2$, for any manifold $W$. 
\begin{lemma} \label{lem1}
Let $(X,\sigma)$ be a torsion combing of a compact oriented 3-manifold $M$, let $\bar\tau$ be a pseudo-parallelization of $M$, and let $E_1^d$ and $E_1^g$ be the Siamese sections of $\bar\tau$. If $\bar\tau$ and $(X,\sigma)$ are compatible, then the expression
$$
4\cdot lk_M(L_{\bar\tau=X} , L_{\bar\tau=-X}) - lk_{\b S^2} \left( e_1-(-e_1) \ , \ P_{\b S^2} \circ \tau_d^{-1} \circ X(L_{{E_1}^d={-E_1}^g})  \right)
$$
depends only on the homotopy class of $(X,\sigma)$. It will be denoted $p_1(\bar\tau,[X])$ and its opposite will be written $p_1([X],\bar\tau)$.
\end{lemma}

\begin{theorem} \label{thm_defp1Xb} 
Let $(X_1,\sigma_{X_1})$ and $(X_2,\sigma_{X_2})$ be torsion combings of two compact oriented 3-manifolds $M_1$ and $M_2$ with identified boundaries such that $(X_1,\sigma_{X_1})$ and $(X_2,\sigma_{X_2})$ coincide on the boundary. For $i\in\lbrace 1,2\rbrace$, let $\bar\tau_i$ be a pseudo-parallelization of $M_i$ such that $\bar\tau_i$ and $(X_i,\sigma_{X_i})$ are compatible. The expression
$$
p_1([X_1],[X_2])=  p_1( [X_1],\bar\tau_1) + p_1(\bar\tau_1, \bar\tau_2) + p_1(\bar\tau_2, [X_2]) 
$$
depends only on the homotopy classes of $(X_1,\sigma_{X_1})$ and $(X_2,\sigma_{X_2})$, and it defines \textup{the first relative Pontrjagin number of $(X_1,\sigma_{X_1})$ and $(X_2,\sigma_{X_2})$}. Moreover, if $M_1$ and $M_2$ are closed, then
$$
p_1([X_1],[X_2])=p_1([X_2])-p_1([X_1]).
$$
\end{theorem}

Under the assumptions of Theorem \ref{thm_D2nd0}, we see that it would be impossible to naively define $p_1([X_{|A_1}],[{X_{\lbrace 1 \rbrace}}_{|B_1}])$ as $p_1([X])-p_1([{X_{\lbrace 1 \rbrace}}])$, where $X$ extends $X_{|A_1}$ to the closed manifold $M$, and ${X_{\lbrace 1 \rbrace}}$ extends ${X_{\lbrace 1 \rbrace}}_{|B_1}$ in the same way to $M(\sfrac{B_1}{A_1})$. Indeed Theorem \ref{thm_D2nd0} and the example that follows it show that the expression $\left(p_1([X])-p_1([{X_{\lbrace 1 \rbrace}}])\right)$ depends on the combed manifold $(M,X)$ into which $(A_1,X_{|A_1})$ has been embedded. It even depends on the combing $X$ that extends the combing $X_{|A_1}$ of $A_1$ to $M$ for the fixed manifold $M$ of this example, since 
$$
\left( p_1([X]) - p_1([{X_{\lbrace 1 \rbrace}}]) \right) - \left( p_1([{X_{\lbrace 2 \rbrace}}]) - p_1([{X_{\lbrace 1,2 \rbrace}}]) \right) =-8
$$ 
there. \\

Theorem \ref{thm_defp1Xb} translates as follows in the closed case and it bridges a gap between the two dissimilar generalizations of the Pontrjagin numbers of parallelizations for pseudo-parallelizations and for torsion combings in closed oriented 3-manifolds.
\begin{corollary} \label{cor_p1Xppara}
Let $X$ be a torsion combing of a closed oriented 3-manifold $M$ and let \linebreak $\bar\tau (N(\gamma);\tau_e,\tau_d)$ be a pseudo-parallelization of $M$. Let $E_1^d$ and $E_1^g$ denote the Siamese sections of $\bar\tau$. If $X$ and $\bar \tau$ are compatible, then
$$
\begin{aligned}
p_1([X])
= p_1(\bar \tau) &+ 4\cdot lk_M(L_{\bar\tau=X} , L_{\bar\tau=-X}) \\
&- lk_{\b S^2} \left( e_1-(-e_1) \ , \ P_{\b S^2} \circ \tau_d^{-1} \circ X(L_{{E_1}^d={-E_1}^g})  \right).
\end{aligned}
$$
\end{corollary}

Another special case is when genuine parallelizations can be used. The closed case with genuine parallelizations is nothing but C. Lescop's definition of the Pontrjagin number of torsion combings in closed oriented 3-manifolds stated above.
\begin{corollary} \label{cor_compactpara}
Let $(X_1,\sigma_1)$ and $(X_2,\sigma_2)$ be torsion combings of two compact oriented 3-manifolds $M_1$ and $M_2$ with identified boundaries such that $(X_1,\sigma_1)$ and $(X_2,\sigma_2)$ coincide on the boundary. If, for $i \in \lbrace 1 , 2 \rbrace$, $\tau_i \italicegal (E_1^i,E_2^i,E_3^i)$ is a parallelization of $M_i$ such that $(X_i,\sigma_i)$ and $(E_1^i,{E_2^i}_{|\partial M_i})$ are $\partial$-compatible, then
$$
p_1([X_1],[X_2]) = p_1(\tau_1 , \tau_2) + 4 \cdot lk_{M_2}(L_{E_1^{2}=X_2} \ , \ L_{E_1^{2}=-X_2}) - 4 \cdot lk_{M_1}(L_{E_1^{1}=X_1} \ , \ L_{E_1^{1}=-X_1}).
$$
\end{corollary}

Finally, for torsion combings defined on a fixed compact oriented 3-manifold (which may have boundary), we have the following simple variation formula, as in the closed case.
\begin{theorem} \label{formuleplus}
If $(X, \sigma)$ and $(Y,\sigma)$ are $\partial$-compatible torsion combings of a compact oriented 3-manifold $M$, then
$$
p_1([X],[Y]) = 4 \cdot lk_M(L_{X=Y}, L_{X=-Y}).
$$
\end{theorem}

\def\Spinc{\mbox{spin$^c$}}

Let $M$ be a compact connected oriented 3-manifold. For all section $\sigma$ of $TM_{|\partial M}$, let $\Spinc(M,\sigma)$ denote the \textit{set of $\Spinc$-structures on $M$ relative to $\sigma$}, \ie the set of homotopy classes on $M \setminus \lbrace \omega \rbrace$ of combings $(X,\sigma)$ of $M$, where $\omega$ is any point in $\mathring M$  (see \cite{gmdeloup}, for a detailed presentation of $\Spinc$-structures). Thanks to Theorem~\ref{formuleplus}, it is possible to classify the torsion combings of a fixed $\Spinc$-structure up to homotopy, thus generalizing a property of the Gompf invariant in the closed case. I thank Gw\'{e}na\"{e}l Massuyeau for suggesting this statement.
\begin{theorem} \label{GM}
Let $(X,\sigma)$ and $(Y,\sigma)$ be $\partial$-compatible torsion combings of a compact connected oriented 3-manifold $M$ which represent the same $\Spinc$-structure. The combings $(X,\sigma)$ and $(Y,\sigma)$ are homotopic relatively to the boundary if and only if $p_1([X],[Y]) =0$.
\end{theorem}

The key tool in the proof of Theorem~\ref{thm_defp1Xb} is the following generalization of the interpretation of the variation of the Pontrjagin numbers of parallelizations as an algebraic intersection of three chains.
\begin{theorem} \label{prop_varasint}
Let $\tau$ and $\bar \tau$ be two pseudo-parallelizations of a compact oriented 3-manifold $M$ that coincide on $\partial M$ and whose links are disjoint. For any $v \in \b S^2$, there exists a 4-chain $C_4(\tau,\bar\tau ;v)$ of $[0,1]\times UM$ transverse to the boundary of $[0,1] \times UM$ such that
$$
\partial C_4(\tau,\bar\tau ;v) = \lbrace 1 \rbrace \times \bar\tau(M\times\lbrace v \rbrace) - \lbrace 0 \rbrace \times \tau(M\times \lbrace v \rbrace) - [0,1]\times \tau(\partial M \times \lbrace v \rbrace)
$$
and for any $x,y$ and $z$ in $\b S^2$ with pairwise different distances to $e_1$ :
$$
p_1(\tau,\bar\tau)= 4 \cdot \langle C_4(\tau,\bar\tau ;x),C_4(\tau,\bar\tau ;y),C_4(\tau,\bar\tau ;z) \rangle_{[0,1]\times UM}
$$
for any triple of pairwise transverse $C_4(\tau,\bar\tau,x)$, $C_4(\tau,\bar\tau,y)$ and $C_4(\tau,\bar\tau,z)$ that satisfy the hypotheses above.
\end{theorem}

Our general variation formula for Pontrjagin numbers of torsion combings reads as follows for all compact oriented 3-manifolds.

\begin{theorem}\label{thm_D2nd}
Let $(M,X)$ be a compact oriented 3-manifold equipped with a combing, \linebreak let $\lbrace (\sfrac{B_i}{A_i},X_{B_i}) \rbrace_{i \in \lbrace 1,2 \rbrace}$ be two disjoint LP$_\b Q$-surgeries in $(M,X)$, and, for all $I \subset \lbrace 1 , 2 \rbrace$, \linebreak let $X^I = X(\lbrace \sfrac{B_i}{A_i} \rbrace_{i \in I})$.  If $\lbrace X^I \rbrace_{I \subset \lbrace 1,2 \rbrace}$ is a family of torsion combings of the manifolds \linebreak $M_I=M(\lbrace \sfrac{B_i}{A_i} \rbrace_{i \in I})$, then
$$
p_1([X^{\lbrace 2 \rbrace}],[X^{\lbrace 1, 2 \rbrace}])-p_1([X],[X^{\lbrace 1 \rbrace}]) = - 2 \cdot lk_M \left(L_{\lbrace X^I \rbrace}(\sfrac{B_1}{A_1}), L_{\lbrace X^I \rbrace}(\sfrac{B_2}{A_2})\right),
$$
where the right-hand side is defined as in Theorem~\ref{thm_D2nd0}.
\end{theorem}

A direct consequence of this variation formula is that the extended Gompf invariant for torsion combings of compact oriented 3-manifolds is a degree two finite type invariant with respect to LP$_\b Q$-surgeries.
\begin{corollary} \label{cor_FTcombings}
Let $(M,X)$ be a compact oriented 3-manifold equipped with a combing, let $\lbrace ( \sfrac{B_i}{A_i} , X_{B_i} ) \rbrace_{i\in \lbrace 1, \ldots, k\rbrace}$ be a family of disjoint LP$_\b Q$-surgeries in $(M,X)$, and, for all $I \subset \lbrace 1 , \ldots, k \rbrace$, let $(M_I,X^I)$ be the combed manifold obtained by performing the surgeries associated to the data $\lbrace (\sfrac{B_i}{A_i}, X_{B_i}) \rbrace_{i \in I}$. If $k\geqslant 3$, and if $\lbrace X^I \rbrace_{I \subset \lbrace 1, \ldots, k \rbrace}$ is a family of torsion combings of the $\lbrace M_I \rbrace_{I \subset \lbrace 1, \ldots, k \rbrace}$, then
$$
\sum_{I \subset \lbrace 2 , \ldots , k \rbrace} (-1)^{\card(I)} \ p_1 \left( [X^I] , [X^{I\cup\lbrace 1 \rbrace}] \right)=0.
$$
If $\partial M=\emptyset$, this reads
$$
\sum_{I \subset \lbrace 1 , \ldots , k \rbrace} (-1)^{\card(I)} \ p_1 \left( [X^I] \right)=0.
$$
\end{corollary}

In the first section of this article, we give details on Lagrangian-preserving surgeries, combings and pseudo-parallelizations. Then, we review the definitions of Pontrjagin numbers of parallelizations and pseudo-parallelizations. The second section ends with a proof of Theorem~\ref{prop_varasint}. The third section is devoted to the proof of Theorem~\ref{thm_defp1Xb} and Theorem~\ref{GM}. Finally, we study the variations of Pontrjagin numbers with respect to Lagrangian-preserving surgeries, and finish the last section by proving Theorem~\ref{thm_D2nd}.\\

\begin{large}\textbf{Acknowledgments.}\end{large} First, let me thank C. Lescop and J.-B. Meilhan for their thorough guidance and support. I also thank M. Eisermann and G. Massuyeau for their careful reading and their useful remarks. 

\renewcommand{\thetheorem}{\arabic{section}.\arabic{theorem}}

\section{More about ...}
\subsection{Lagrangian-preserving surgeries} \label{ssec_LPsurgeries}

Let us first note three easy lemmas, the proofs of which are left to the reader.

\begin{lemma} \label{prop_redefLPs}
Let $(\sfrac{B}{A})$ be an LP$_\b Q$-surgery datum in a compact oriented 3-manifold $M$ and let $L_1$ and $L_2$ be links in $M \minusens \mathring A$. If $L_1$ and $L_2$ are rationally null-homologous in $M$, then they are null-homologous in $M(\sfrac{B}{A})$ and
$$
lk_{M(\sfrac{B}{A})}(L_1,L_2) = lk_M(L_1,L_2).
$$
\end{lemma}
\iffalse
\begin{proof}
For $i \in \lbrace 1,2 \rbrace$, let $\Sigma_i$ be a 2-chain in $M$ such that $\partial \Sigma_i=L_i$ and $\Sigma_i$ is transverse to $\partial A$. Since $(\sfrac{B}{A})$ is an LP$_\b Q$-surgery, $[\Sigma_i\cap \partial A]= 0$ in $H_1(B; \b Q)$. Therefore, for all $i \in \lbrace 1,2 \rbrace$, there exists a 2-chain $\Sigma'_i$ in $M(\sfrac{B}{A})$ such that
$$
\partial \Sigma'_i = L_i \mbox{ \ and \ } \Sigma'_i \cap (M(\sfrac{B}{A})\minusens \mathring B) = \Sigma_i \cap (M \minusens \mathring A).
$$
As a consequence $L_1$ and $L_2$ are null-homologous in $M(\sfrac{B}{A})$ and, since $L_2 \subset M \minusens \mathring A$, it follows that
$$
lk_{M(\sfrac{B}{A})}(L_1,L_2) = \langle \Sigma'_1 , L_2 \rangle_{M(\sfrac{B}{A})} = \langle \Sigma_1 , L_2 \rangle_M = lk_M(L_1,L_2).
$$
\end{proof}
\fi

A \textit{rational homology 3-sphere}, or a $\b Q$HS for short, is a closed oriented 3-manifold with the same homology with rational coefficients as $\b S^3$.

\begin{lemma} \label{prop-LP2}
Let $(\sfrac{B}{A})$ be an LP$_\b Q$-surgery in a compact oriented 3-manifold $M$. If $M$ is a $\b Q$HS, then $M(\sfrac{B}{A})$ is a $\b Q$HS.
\end{lemma}
\iffalse
\begin{proof}
Let $M$ be a $\b Q$HS and let $A$ be a $\b Q$HH of genus $g\in \b N$. Using the Mayer-Vietoris sequence associated to $M=A\cup (M\minusens\mathring A)$ shows that $M\minusens \mathring A$ is a $\b Q$HH of genus $g$ and that the inclusions of $\partial A$ into $A$ and into $M\setminus \mathring A$ induce an isomorphism
$$
H_1(\partial A ; \b Q) \simeq H_1(A; \b Q)\oplus H_1(M\minusens\mathring A ; \b Q).
$$
For details, see \cite[Sublemma 4.6]{moussardFTIQHS}. It follows that $M(\sfrac{B}{A})\minusens \mathring B$ is also a genus $g$ $\b Q$HH and, since $(\sfrac{B}{A})$ is an LP$_\b Q$-surgery, the inclusions of $\partial B$ into $B$ and into $M(\sfrac{B}{A})\setminus \mathring B$ induce an isomorphism
$$
H_1(\partial B ; \b Q) \simeq H_1(B; \b Q)\oplus H_1(M(\sfrac{B}{A})\minusens\mathring B ; \b Q).
$$
Using this isomorphism in the Mayer-Vietoris sequence associated to the splitting \linebreak $M(\sfrac{B}{A})=B\cup (M(\sfrac{B}{A})\minusens\mathring B)$ shows that $M(\sfrac{B}{A})$ is a $\b Q$HS.
\end{proof}
\fi

\begin{lemma} \label{phom}
If $A$ is a compact connected orientable 3-manifold with connected boundary and if the map $i^A_* : H_1(\partial A ; \b Q) \rightarrow H_1(A; \b Q)$ induced by the inclusion of $\partial A$ into $A$ is surjective, then $A$ is a rational homology handlebody.
\end{lemma}
\iffalse
\begin{proof}
First, for such a manifold $A$, we have $H_0(A;\b Q)\simeq \b Q$ and $H_3(A;\b Q)\simeq 0$. Second, using the hypothesis on $i^A_*$ in the exact sequence associated to $(A, \partial A)$, we get that $H_1(A,\partial A;\b Q)= 0$. Using Poincaré duality and the universal coefficient theorem, it follows that
$$
H_2(A;\b Q) \simeq H^1(A,\partial A;\b Q) \simeq \mbox{Hom}(H_1(A,\partial A;\b Q),\b Q) = 0.
$$
Moreover, we get the following exact sequence from the exact sequence associated to $(A,\partial A)$~:
$$
0\rightarrow H_2(A,\partial A; \b Q) \rightarrow H_1(\partial A;\b Q) \rightarrow H_1(A;\b Q) \rightarrow 0.
$$
It follows that $\mbox{dim}(H_2(A,\partial A;\b Q))+\mbox{dim}(H_1(A;\b Q))=\mbox{dim}(H_1(\partial A;\b Q)) = 2g$, where $g$ denotes the genus of $\partial A$. However,  
$$
H_2(A,\partial A;\b Q) \simeq H^1(A;\b Q) \simeq \mbox{Hom}(H_1(A;\b Q),\b Q),
$$
hence $\mbox{dim}(H_1(A;\b Q))=g$.
\end{proof}
\fi

\begin{proposition}
Let $A$ be a compact submanifold with connected boundary of a $\b Q$HS $M$, let $B$ be a compact oriented 3-manifold and let $h:\partial A \rightarrow \partial B$ be a homeomorphism. If the surgered manifold $M(\sfrac{B}{A})$ is a $\b Q$HS and if 
$$ lk_{M(\sfrac{B}{A})}(L_1,L_2) = lk_M(L_1,L_2)
$$
for all disjoint links $L_1$ and $L_2$ in $M \minusens \mathring A$, then $(\sfrac{B}{A})$ is an LP$_\b Q$-surgery.
\end{proposition}
\begin{proof}
Using the Mayer-Vietoris exact sequences associated to $M=A\cup (M\minusens\mathring A)$ and \linebreak $M(\sfrac{B}{A})=B\cup (M(\sfrac{B}{A})\minusens\mathring B)$, we get that the maps $i_*^A : H_1(\partial A ; \b Q) \longrightarrow H_1(A; \b Q)$ and \linebreak $i_*^B : H_1(\partial B ; \b Q) \longrightarrow H_1(B; \b Q)$ induced by the inclusions of $\partial A$ and $\partial B$ into $A$ and $B$ are surjective. Using Lemma~\ref{phom}, it follows that $A$ and $B$ are rational homology handlebodies. Moreover, $A$ and $B$ have the same genus since $h : \partial A \rightarrow \partial B$ is a homeomorphism. \\

Let $P_{\go L_A}$ and $P_{\go L_B}$ denote the projections from $H_1(\partial A;\b Q)$ onto $\go L_A$ and $\go L_B$, respectively, with kernel $\go L_{M\setminus \mathring A}$. Consider a collar $[0,1]\times \partial A$ of $\partial A$ such that $\lbrace 0 \rbrace \times \partial A \simeq \partial A$ and note that for all 1-cycles $x$ and $y$ of $\partial A$ :
$$
\langle P_{\go L_A}(y), x \rangle_{\partial A} 
= lk_M(\lbrace 1 \rbrace \times y, \lbrace 0 \rbrace \times x )
= lk_{M(\sfrac{B}{A})}(\lbrace 1 \rbrace \times y, \lbrace 0 \rbrace \times x)
=\langle P_{\go L_B}(y), x \rangle_{\partial B},
$$
so that $P_{\go L_B}=P_{\go L_A}$ and $h_*(\go L_A)=\go L_B$.
\end{proof}

\subsection{Combings}
\begin{proposition} \label{prop_linksandsigns}
If $X$ and $Y$ are $\partial$-compatible combings of a compact oriented 3-manifold $M$, then
$$
L_{X=Y}=L_{Y=X} \mbox{ \ and \ } L_{X=Y}= - L_{-X=-Y}.
$$
\end{proposition}
\begin{proof}
First, by definition, the link $L_{X=Y}$ is the projection of the intersection of the sections $X(\mathring M)$ and $Y(\mathring M)$. This intersection is oriented so that 
$$
NX(\mathring M) \oplus NY(\mathring M) \oplus T(X(\mathring M)\cap Y(\mathring M))
$$
orients $UM$, fiberwise. Since the normal bundles $NX(\mathring M)$ and $NY(\mathring M)$ have dimension 2, the isomorphism permuting them is orientation-preserving so that $L_{X=Y}=L_{Y=X}$. Second, $(-X)(\mathring M)\cap(-Y)(\mathring M)$ is the image of $X(\mathring M)\cap Y(\mathring M)$ under the map $\iota$ from $UM$ to itself which acts on each fiber as the antipodal map. This map reverses the orientation of $UM$ as well as the coorientations of $X(\mathring M)$ and $Y(\mathring M)$, \ie
$$
\begin{aligned}
&N(-X)(M)=-\iota(NX(M)), \ N(-Y)(M)=-\iota(NY(M)). \\ 
\end{aligned}
$$
Since $N(-X)(M) \oplus N(-Y)(M) \oplus T((-X)(M)\cap (-Y)(M))$ has the orientation of $UM$
$$
T((-X)(M)\cap (-Y)(M)) = -\iota(T(X(M)\cap Y(M))).
$$
Hence, $L_{X=Y}= - L_{-X=-Y}$.
\end{proof}

\begin{definition}
Let $M$ be a compact oriented 3-manifold and let $L$ be a link in $\mathring M$. Define the \textit{blow up of $M$ along $L$} as the 3-manifold $Bl(M,L)$ constructed from $M$ in which $L$ is replaced by its unit normal bundle in $M$. The 3-manifold $Bl(M,L)$ inherits a canonical differential structure. See \cite[Definition 3.5]{lescopcombing} for a detailed description.
\end{definition}

\begin{lemma} \label{lem_phomotopy}
Let $X$ and $Y$ be $\partial$-compatible combings of a compact oriented 3-manifold $M$. There exists a 4-chain $\bar F(X,Y)$ of $UM$ with boundary~:
$$
\partial \bar F(X,Y) = Y(M) - X(M) + UM_{| L_{X=-Y}}.
$$
\end{lemma}
\begin{proof}
To construct the desired 4-chain, start with the partial homotopy from $X$ to $Y$
$$
\tilde F(X,Y) :
\left\lbrace
\begin{aligned}
\ [0,1] \times (M\setminus L_{X=-Y}) &\longrightarrow UM\\
(s,m) & \longmapsto \left( m , \l H_X^Y(s,m) \right)
\end{aligned}
\right.
$$
where $\l H_X^Y(s,m)$ is the unique point of the shortest geodesic arc from $X(m)$ to $Y(m)$ such that
$$
d_{\b S^2}(X(m),\l H_X^Y(s,m)) = s \cdot d_{\b S^2}(X(m),Y(m))
$$
where $d_{\b S^2}$ denotes the usual distance on $\b S^2$. Next, extend the map
$$
(s,m) \longmapsto \l H_X^Y(s,m)
$$
on the blow up of $M$ along $L_{X=-Y}$. The section $X$ induces a map
$$
X : NL_{X=-Y} \longrightarrow -Y^\perp(L_{X=-Y})
$$
which is a diffeomorphism on a neighborhood of $\lbrace 0 \rbrace \times L_{X=-Y}$ since $X$ and $Y$ are $\partial$-compatible combings. Furthermore, this diffeomorphism is orientation-preserving by definition of the orientation on $L_{X=-Y}$. So, for $n \in UN_mL_{X=-Y}$, $\l H_X^Y(s,n)$ can be defined as the unique point at distance $s\pi$ from $X(m)$ on the unique half great circle from $X(m)$ to $Y(m)$ through $T_m X(n)$. Thanks to transversality again, the set  $\lbrace \l H_X^Y(s,n) \ | \ s \in [0,1], \ n \in UN_mL_{X=-Y} \rbrace$ is a whole sphere $\b S^2$ for any fixed $m \in L_{X=-Y}$, so that
$$
\partial \tilde F(X,Y)([0,1]\times Bl(M,L_{X=-Y})) =  Y(M) - X(M) + \partial_{int}
$$
where $\partial_{int} \simeq L_{X=-Y} \times \b S^2 $ (see \cite[Proof of Proposition 3.6]{lescopcombing} for the orientation of $\partial_{int}$). Finally, let $\bar F(X,Y)=\tilde F(X,Y)([0,1]\times Bl(M,L_{X=-Y}))$.
\end{proof}

If $X$ and $Y$ are $\partial$-compatible combings of a compact oriented 3-manifold $M$ and if $\sigma$ is a nonvanishing section of $X^\perp_{|\partial M}$, let $\l H^{-Y}_{X,\sigma}$ denote the map from $[0,1] \times (M\minusens L_{X=Y})$ to $UM$ such that, for all $(s,m)$ in $[0,1] \times \partial M$, $\l H^{-Y}_{X,\sigma}(s,m)$ is the unique point at distance $s\pi$ from $X(m)$ on the unique geodesic arc starting from $X(m)$ in the direction of $\sigma(m)$ to $-X(m)=-Y(m)$ and, for all $(s,m)$ in $[0,1] \times (\mathring M \setminus L_{X=Y})$, $\l H^{-Y}_{X,\sigma}(s,m)$ is the unique point on the shortest geodesic arc from $X(m)$ to $-Y(m)$ such that
$$
d_{\b S^2}(X(m),\l H^{-Y}_{X,\sigma}(s,m)) = s \cdot d_{\b S^2}(X(m),-Y(m)).
$$
As in the previous proof, $\l H^{-Y}_{X,\sigma}$ may be extended as a map from $[0,1]\times Bl(M,L_{X=Y})$ to $UM$. In the case of $X=Y$, for all section $\sigma$ of $X^\perp$, nonvanishing on $\partial M$, let $L_{\sigma=0}$ denote the oriented link $\lbrace m \in M \ | \ \sigma(m)= 0 \rbrace$ and define a map $\l H_{X,\sigma}^{-X}$ as the map from $[0,1] \times (M\minusens L_{\sigma=0})$ to $UM$ such that, for all $(s,m)$ in $[0,1] \times (M\minusens L _{\sigma=0})$, $\l H_{X,\sigma}^{-X}(s,m)$ is the unique point at distance $s\pi$ from $X(m)$ on the unique geodesic arc starting from $X(m)$ in the direction of $\sigma(m)$ to $-X(m)$. Note that $L_{\sigma=0}\cap\partial M = \emptyset$, and $[L_{\sigma=0}]=P(e^M_2(X,\sigma_{|\partial M}))$. Here again, $\l H_{X,\sigma}^{-X}$ may be extended as a map from $[0,1]\times Bl(M,L_{\sigma=0})$ to $UM$. \\

In order to simplify notations, if $A$ is a submanifold of a compact oriented 3-manifold $M$, we may implicitly use a parallelization of $M$ to write $UM_{|A}$ as $A\times \b S^2$.

\begin{proposition} \label{prop_links}
If $(X,\sigma)$ and $(Y,\sigma)$ are $\partial$-compatible combings of a compact oriented 3-manifold $M$, then, in $H_3(UM;\b Z)$,
$$
\begin{aligned}
\ [L_{X=-Y} \times \b S^2 ] &= [X(M) - Y(M)] \\
\ [L_{X=Y} \times \b S^2 ] &= [X(M) - (-Y)(M) + \l H^{-Y}_{X,\sigma} ([0,1]\times \partial M)].
\end{aligned}
$$
\end{proposition}
\begin{proof}
The first identity is a direct consequence of Lemma~\ref{lem_phomotopy}. The second one can be obtained using a similar construction. Namely, construct a 4-chain $\bar F(X,-Y)$ using the partial homotopy from $X$ to $-Y$ :
$$
\tilde F(X,-Y) : \left\lbrace
\begin{aligned}
 \ [0,1] \times (M\setminus L_{X=Y}) &\longrightarrow UM\\
(s,m) & \longmapsto \left( m , \l H^{-Y}_{X,\sigma}(s,m) \right).
\end{aligned}
\right.
$$
As in the proof of Lemma~\ref{lem_phomotopy}, $\tilde F(X,-Y)$ can be extended to $[0,1] \times Bl(M,L_{X=Y})$. Finally, we get a 4-chain $\bar F(X,-Y)$ of $UM$ with boundary :
$$
\partial \bar F(X,-Y) = (-Y)(M) - X(M) - \l H^{-Y}_{X,\sigma} ([0,1]\times \partial M) + UM_{|L_{X=Y}}.
$$
\end{proof}

\begin{proposition} \label{prop_euler}
Let $X$ be a combing of a compact oriented 3-manifold $M$ and let \linebreak $P : H^2(M, \partial M;\b Z) \rightarrow H_1(M;\b Z)$ be the Poincaré duality isomorphism. If $M$ is closed, then, in $H_3(UM;\b Z)$,
$$
[P(e_2(X^\perp)) \times \b S^2 ] = [X(M) - (-X)(M)],
$$
where $[P(.)\times S^2]$ abusively denotes the homology class of the preimage of a representative of $P(.)$ under the bundle projection $UM \rightarrow M$. In general, if $\sigma$ is a section of $X^\perp$ such that $L_{\sigma=0}\cap\partial M = \emptyset$ then, in $H_3(UM;\b Z)$,
$$
[ P(e_2^M(X^\perp,\sigma_{|\partial M})) \times \b S^2 ] = [X(M) - (-X)(M) +\l H_{X,\sigma}^{-X}([0,1]\times \partial M)].
$$
\end{proposition}
\begin{proof}
Recall that $P(e_2^M(X^\perp,\sigma_{|\partial M}))=[L_{\sigma=0}]$. Perturbing $X$ by using $\sigma$, construct a section $Y$ homotopic to $X$ that coincides with $X$ on $\partial M$ and such that $[L_{X=Y}] = P(e_2^M(X^\perp,\sigma_{|\partial M}))$. Using Proposition~\ref{prop_links},
$$
[L_{X=Y} \times \b S^2 ] = [X(M) - (-Y)(M) + \l H^{-Y}_{X,\sigma_{|\partial M}} ([0,1]\times \partial M)],
$$
so that
$$
[ P(e_2^M(X^\perp,\sigma_{|\partial M})) \times \b S^2 ] = [X(M) - (-X)(M) +\l H_{X,\sigma}^{-X}([0,1]\times \partial M)].
$$
\end{proof}

\begin{proposition} \label{prop_linksinhomologyI}
If $(X,\sigma)$ and $(Y,\sigma)$ are $\partial$-compatible combings of a compact oriented 3-manifold $M$, then, in $H_1(M;\b Z)$, 
$$
\begin{aligned}
2 \cdot [L_{X=-Y}] &= P(e_2^M(X^\perp,\sigma)) -  P(e_2^M(Y^\perp,\sigma)),  \\
2 \cdot [L_{X=Y}] &= P(e_2^M(X^\perp,\sigma)) + P(e_2^M(Y^\perp,\sigma)).
\end{aligned}
$$
\end{proposition}
\begin{proof}
Extend $\sigma$ as a section $\bar\sigma$ of $X^\perp$. Using Propositions~\ref{prop_linksandsigns},~\ref{prop_links}~and~\ref{prop_euler}, we get, in $H_3(UM;\b Z)$,
 $$
 \begin{aligned}
 2 \ \cdot \ [L_{X=-Y}\times \b S^2] &= [L_{X=-Y} \times \b S^2] - [L_{-X=Y} \times \b S^2] \\
 &= [X(M)-Y(M)]-[(-X)(M)-(-Y)(M)] \\
 &= [X(M)-Y(M)-(-X)(M)+(-Y)(M)  \\
 & \hspace{5mm}+\l H_{X,\bar\sigma}^{-X}([0,1]\times \partial M)  -\l H_{X,\bar\sigma}^{-X} ([0,1]\times \partial M) ] \\
 &= [X(M)-(-X)(M)+\l H_{X,\bar\sigma}^{-X}([0,1]\times \partial M)] \\
 &- [Y(M)-(-Y)(M)+\l H_{X,\bar\sigma}^{-X}([0,1]\times \partial M) ] \\
 &= [P(e_2^M(X^\perp,\sigma)) \times \b S^2]  - [P(e_2^M(Y^\perp,\sigma)) \times \b S^2 ],
 \end{aligned}
 $$
 $$
 \begin{aligned}
 2 \cdot [L_{X=Y}\times \b S^2] &= [L_{X=Y} \times \b S^2] - [L_{-X=-Y} \times \b S^2]  \\
 &= [X(M) - (-Y)(M)+\l H_{X,\bar\sigma_{|\partial M}}^{-Y} ([0,1]\times \partial M)]\\
 &- [(-X)(M) - Y(M) +\l H_{-X,\bar\sigma_{|\partial M}}^{Y} ([0,1]\times \partial M)] \\
 &= [ P(e_2^M(X^\perp,\sigma)) \times \b S^2 ] - [ P(e_2^M((-Y)^\perp,\sigma)) \times \b S^2 ] \\
 &= [ P(e_2^M(X^\perp,\sigma)) \times \b S^2 ] + [ P(e_2^M(Y^\perp,\sigma)) \times \b S^2 ].
 \end{aligned}
 $$
  \iffalse &= [X(M) - (-Y)(M)+\l H_{X,\sigma}^{-X} ([0,1]\times \partial M)]\\
 &- [(-X)(M) - Y(M) + \l H_{-Y,\sigma}^{Y} ([0,1]\times \partial M)] \\
 &= [X(M)-(-X)(M) + \l H_{X,\sigma}^{-X} ([0,1]\times \partial M)] \\
 &- [(-Y)(M)-Y(M)+ \l H_{-Y,\sigma}^{Y} ([0,1]\times \partial M)] \\ \fi
\end{proof}

\begin{remark} 
If $M$ is a compact oriented 3-manifold and if $\sigma$ is a trivialization of $TM_{|\partial M}$, then the set $\Spinc(M,\sigma)$ is a $H^2(M,\partial M; \b Z)$-affine space and the map
$$
c : \left\lbrace
\begin{aligned}
\Spinc(M,\sigma) &\longrightarrow H^2(M,\partial M; \b Z) \\
 [X]^c &\longmapsto e_2^M(X^\perp, \sigma)
\end{aligned}
\right.
$$
is affine over the multiplication by 2. Moreover, $[X]^c-[Y]^c \in H^2(M,\partial M; \b Z) \simeq H_1(M; \b Z)$ is represented by $L_{X=-Y}$, hence $2 \cdot [L_{X=-Y}] = P(e_2^M(X^\perp,\sigma)) -  P(e_2^M(Y^\perp,\sigma))$. See \cite[Section 1.3.4]{gmdeloup} for a detailed presentation using this point of view. Both Proposition \ref{prop_linksinhomologyI} and Corollary \ref{corrplus} below are already-known results. For instance, Corollary \ref{corrplus} is also present in \cite{lescopcombing} (Lemma 2.16).
\end{remark}

\begin{corollary} \label{corrplus}
If $X$ and $Y$ are transverse combings of a closed oriented 3-manifold $M$, then, in $H_1(M;\b Z)$,
$$
\begin{aligned}
2 \cdot [L_{X=-Y}] &= P(e_2(X^\perp)) - P(e_2(Y^\perp)), \\
2 \cdot [L_{X=Y}] &= P(e_2(X^\perp)) + P(e_2(Y^\perp)).
\end{aligned}
$$
\end{corollary}

\subsection{Pseudo-parallelizations} \label{ssec_defppara}
\hspace{-3mm} A \textit{pseudo-parallelization} $\bar\tau \hspace{-1mm} = \hspace{-1mm} (N(\gamma); \tau_e, \tau_d)$ of a compact oriented 3-manifold $M$ is a triple~where
\begin{enumerate}[\textbullet]
\setlength{\itemsep}{0pt}
\setlength{\parskip}{5pt}
\item $\gamma$ is a link in $\mathring{M}$,
\item $N(\gamma)$ is a tubular neighborhood of $\gamma$ with a given product struture : 
$$
N(\gamma)\simeq [a,b] \times \gamma \times [-1,1],
$$
\item $\tau_e$ is a genuine parallelization of $\smash{M\setminus\mathring{N(\gamma)}}$,
\item $\tau_d$ is a genuine parallelization of $N(\gamma)$ such that
$$
\tau_d = \left\lbrace
\begin{aligned}
& \tau_e &\mbox{ on } \partial (\left[ a , b \right] \times \gamma \times \left[ -1 , 1 \right]) \setminus \lbrace b \rbrace \times \gamma \times \left[ -1 , 1 \right] \\
& \tau_e \circ \l T_\gamma &\mbox{ on }  \lbrace b \rbrace \times \gamma \times \left[ -1 , 1 \right]
\end{aligned}
\right.
$$
where $\l T_\gamma$ is 
$$
\l T_\gamma :
\left\lbrace
\begin{aligned}
( [a,b] \times \gamma \times \left[ -1 , 1 \right] )\times \b R^3 & \longrightarrow ([a,b] \times \gamma \times \left[ -1 , 1 \right]) \times \b R^3 \\
((t,c,u),v)  &\longmapsto ((t,c,u),R_{e_1, \pi+\theta(u)}(v)).
\end{aligned}
\right.
$$
where $R_{e_1, \pi+\theta(u)}$ is the rotation of axis $e_1$ and angle $\pi+\theta(u)$, and where $\theta : [-1,1] \rightarrow [-\pi,\pi]$ is a smooth increasing map constant equal to $\pi$ on the interval $[-1,-1+\varepsilon]$ ($\varepsilon \in ]0,\sfrac{1}{2}[$), and such that $\theta(-x)=-\theta(x)$. 
\end{enumerate}
 
Note that a pseudo-parallelization whose link is empty is a parallelization.

\begin{lemma} \label{lem_extendpparallelization}
If $M$ is a compact oriented 3-manifold with boundary and if $\rho$ is a trivialization of $TM_{|\partial M}$, there exists a pseudo-parallelization $\bar\tau$ of $M$ that coincides with $\rho$ on $\partial M$.
\end{lemma}
\begin{proof}
The obstruction to extending the trivialization $\rho$ as a parallelization of $M$ can be represented by an element $[\gamma] \in H_1(M;\pi_1(SO(3)))$ where $\gamma$ is a link in $M$. It follows that $\rho$ can be extended on $M\setminus \mathring {N(\gamma)}$ where $N(\gamma)$ is a tubular neighborhood of $\gamma$. Finally, according to \cite[Lemma 10.2]{lescopcube}, it is possible to extend $\rho$ as a pseudo-parallelization on each torus of~$N(\gamma)$.
\end{proof}

Thanks to Lemma \ref{lem_extendpparallelization}, an LP$_\b Q$-surgery in a rational homology 3-sphere equipped with a pseudo-parallelization can be seen as a local move. This is not the case for an LP$_\b Q$-surgery in a rational homology 3-sphere equipped with a genuine parallelization. \\

\iffalse
Unlike the case of a parallelized rational homology 3-sphere, in the case of a rational homology 3-sphere equipped with a pseudo-parallelization, an LP$_\b Q$-surgery can be seen as a local move thanks to Lemma \ref{lem_extendpparallelization} above.
\fi

Before we move on to the definition of pseudo-sections \ie the counterpart of sections of parallelizations for pseudo-parallelizations, we need the following.
\begin{definition} \label{def_addinner}
Let $\bar \tau = (N(\gamma); \tau_e, \tau_d)$ be a pseudo-parallelization of a compact oriented 3-manifold. An \textit{additional inner parallelization} is a map $\tau_g$ such that
$$
\tau_g :
\left\lbrace
\begin{aligned}
\ [a,b]\times \gamma \times [-1,1] \times \b R^3  &\longrightarrow TN(\gamma)  \\
((t,c,u),v) & \longmapsto \tau_d \left( \l T_\gamma^{-1} ((t,c,u),\l F(t,u)(v))\right)
\end{aligned}
\right.
$$
where, choosing $\varepsilon \in \ ]0,\sfrac{1}{2}[$, $\l F$ is a map such that
$$
\l F :
\left\lbrace
\begin{aligned}
\ [a,b]\times[-1,1] & \longrightarrow SO(3) \\
(t,u) & \longmapsto  \left\lbrace
  \begin{aligned}
  & \mbox{Id}_{SO(3)} & \mbox{for $|u|>1-\varepsilon$} \\
  & R_{e_1,\pi + \theta(u)}  & \mbox{for $t < a + \varepsilon$} \\
  & R_{e_1,-\pi - \theta(u)} & \mbox{for $t > b - \varepsilon$}
  \end{aligned} \right.
\end{aligned}
\right.
$$
which exists since $\pi_1(SO(3))\hspace{-1mm}=\sfrac{\b Z}{2 \b Z}$ and which is well-defined up to homotopy since $\pi_2(SO(3))\hspace{-1mm}=~\hspace{-2mm}0$.
\end{definition}

From now on, we will always consider pseudo-parallelizations together with an additional inner parallelization. Finally, note that if $\bar \tau=(N(\gamma); \tau_e, \tau_d, \tau_g)$ is a pseudo-parallelization of a compact oriented 3-manifold together with an additional inner parallelization, then :
\begin{enumerate}[\textbullet]
\setlength{\itemsep}{0pt}
\setlength{\parskip}{5pt}
\item the parallelizations $\tau_e$, $\tau_d$ and $\tau_g$ agree on $\partial N(\gamma) \setminus \lbrace b \rbrace \times \gamma \times [-1,1]$,
\item $\tau_g = \tau_e \circ \l T_\gamma^{-1}$ on $\lbrace b \rbrace \times \gamma \times [-1,1]$.
\end{enumerate}

\begin{definition}
A \textit{pseudo-section} of a pseudo-parallelization of a compact oriented 3-manifold $M$ together with an additional inner parallelization, $\bar \tau=(N(\gamma); \tau_e, \tau_d, \tau_g)$, is a 3-cycle of \linebreak $(UM,UM_{|\partial M})$ of the following form :
$$
\begin{aligned}
\bar \tau (M\times \lbrace v \rbrace ) &= \tau_e((M\setminus \mathring N(\gamma))\times \lbrace v \rbrace ) \\
&+ \frac{ \tau_d(N(\gamma)\times \lbrace v \rbrace) + \tau_g(N(\gamma)\times \lbrace v \rbrace) + \tau_e( \lbrace b \rbrace \times \gamma \times C_2(v)) }{2}
\end{aligned}
$$
where $v\in \b S^2$ and $C_2(v)$ is the 2-chain of $\left[ -1 , 1 \right] \times \b S^1(v)$ of Figure~\ref{fig_C2v}, where $\b S^1(v)$ stands for the circle of $\b S^2$ that lies on the plane orthogonal to $e_1$ and passes through $v$. Note that :
$$
\begin{aligned}
\partial C_2(v) &= \lbrace (u,  R_{e_1,\pi+\theta(u)} (v)) \ |  u \in \left[ -1 , 1 \right]  \rbrace \\
&+ \lbrace (u,  R_{e_1,-\pi-\theta(u)} (v)) \ |  u \in \left[ -1 , 1 \right] \rbrace - 2 \cdot \left[ -1, 1 \right] \times \lbrace v \rbrace.
\end{aligned}
$$
\end{definition}

\begin{center}
\definecolor{zzttqq}{rgb}{0.6,0.2,0}
\begin{tikzpicture}[line cap=round,line join=round,>=triangle 45,x=1.0cm,y=1.0cm]
\clip(-2.25,-1) rectangle (13,2.25);
\fill[color=zzttqq,fill=zzttqq,fill opacity=0.1] (0,0) -- (4,2) -- (0,2) -- cycle;
\fill[color=zzttqq,fill=zzttqq,fill opacity=0.1] (7.5,2) -- (7.5,0) -- (11.5,0) -- cycle;
\draw (0,2)-- (4,2);
\draw (4,0)-- (4,2);
\draw (4,0)-- (0,0);
\draw (0,0)-- (0,2);
\draw (7.5,0)-- (7.5,2);
\draw (7.5,0)-- (11.5,0);
\draw (11.5,0)-- (11.5,2);
\draw (7.5,2)-- (11.5,2);
\draw (0,0)-- (4,2);
\draw (0,2)-- (4,0);
\draw (7.5,2)-- (11.5,0);
\draw (11.5,2)-- (7.5,0);
\draw (-2,1.25) node[anchor=north west] {$C_2(v)=$};
\draw (-0.5,-0.25) node[anchor=north west] {$-1$};
\draw (3.75,-0.25) node[anchor=north west] {$1$};
\draw (7,-0.25) node[anchor=north west] {$-1$};
\draw (11.25,-0.25) node[anchor=north west] {$1$};
\draw [->] (0,0) -- (4,0);
\draw [->] (7.5,0) -- (11.5,0);
\draw [->] (11.5,0) -- (11.5,2);
\draw [->] (4,0) -- (4,2);
\draw (4,2) node[anchor=north west] {$\b S^1(v)$};
\draw (11.5,2) node[anchor=north west] {$\b S^1(v)$};
\draw [color=zzttqq] (0,0)-- (4,2);
\draw [color=zzttqq] (4,2)-- (0,2);
\draw [color=zzttqq] (0,2)-- (0,0);
\draw [color=zzttqq] (7.5,2)-- (7.5,0);
\draw [color=zzttqq] (7.5,0)-- (11.5,0);
\draw [color=zzttqq] (11.5,0)-- (7.5,2);
\draw (5.58,1.13) node[anchor=north west] {$-$};
\end{tikzpicture}
\captionof{figure}{The 2-chain $C_2(v)$ where we cut the annulus $\left[ -1,1 \right] \times \b S^1(v)$ along $\left[ -1, 1 \right] \times \lbrace v \rbrace$.} \label{fig_C2v}
\end{center}

\begin{definition}
If $\bar\tau = (N(\gamma);\tau_e,\tau_d, \tau_g)$ is a pseudo-parallelization of a compact oriented 3-manifold $M$, let the \textit{Siamese sections} of $\bar\tau$ denote the following sections of $UM$:
$$
E_1^{d}:
\left\lbrace
\begin{aligned}
 m \in M\setminus \mathring N(\gamma) & \longmapsto \tau_e(m,e_1) \\
 m \in N(\gamma) & \longmapsto \tau_d(m,e_1)
\end{aligned}
\right.
\hspace{3mm}\mbox{ and }\hspace{3mm}
E_1^{g} :
\left\lbrace
\begin{aligned}
 m \in M\setminus \mathring N(\gamma) & \longmapsto \tau_e(m,e_1) \\
 m \in N(\gamma) & \longmapsto \tau_g(m,e_1).
\end{aligned}
\right.
$$
\end{definition}

As already mentioned in the introduction, note that when $\bar\tau = (N(\gamma);\tau_e,\tau_d, \tau_g)$ is a pseudo-parallelization of a compact oriented 3-manifold $M$, its pseudo-section at $e_1$ reads
$$
\bar\tau(M\times \lbrace e_1 \rbrace) = \frac{E_1^d(M)+E_1^g(M)}{2}
$$
where $E_1^d$ and $E_1^g$ are the Siamese sections of $\bar\tau$.

\section{From parallelizations to pseudo-parallelizations}
\subsection{Pontrjagin numbers of parallelizations} \label{ssec_defpara}
In this subsection we review the definition of first relative Pontrjagin numbers for parallelizations of compact connected oriented 3-manifolds. For a detailed presentation of these objects we refer to \cite[Section 5]{lescopEFTI} and \cite[Subsection~4.1]{lescopcombing}. \\
  
Let $C_1$ and $C_2$ be compact connected oriented 3-manifolds with identified boundaries. Recall that a \textit{cobordism from $C_1$ to $C_2$} is a compact oriented 4-manifold $W$ whose boundary reads
$$
\partial W = - C_1 \bigcup_{\partial C_1 \simeq \lbrace 0 \rbrace \times C_1 } -[0,1] \times \partial C_1 \bigcup_{\partial C_2 \simeq \lbrace 1 \rbrace \times C_1} C_2.
$$
Moreover, we require $W$ to be identified with $[0,1[\times C_1$ or $]0,1]\times C_2$ on collars of $\partial W$. \\

Recall that any compact oriented 3-manifold bounds a compact oriented 4-manifold, so that a cobordism from $C_1$ to $C_2$ always exists. Also recall that the signature of a 4-manifold is the signature of the intersection form on its second homology group with real coefficients and that any 4-manifold can be turned into a 4-manifold with signature zero by performing connected sums with copies of $\pm \b C P^2$. So let us fix a connected cobordism $W$ from $C_1$ to $C_2$ with signature zero. \\

Now consider a parallelization $\tau_1$, resp. $\tau_2$, of $C_1$, resp. $C_2$. Define the vector field $\vec n$ on a collar of $\partial W$ as follows. Let $\vec n$ be the unit tangent vector to $[0,1]\times \lbrace x \rbrace$ where $x \in C_1$ or $C_2$. Define $\tau(\tau_1, \tau_2)$ as the trivialization of $TW \otimes \b C$ over $\partial W$ obtained by stabilizing $\tau_1$ or $\tau_2$ into $\vec n \oplus \tau_1$ or $\vec n \oplus \tau_2$ and tensoring with $\b C$. In general, this trivialization does not extend as a parallelization of $W$. This leads to a Pontrjagin obstruction class $p_1(W;\tau(\tau_1, \tau_2))$ in $H^4(W, \partial W, \pi_3(SU(4)))$. Since $\pi_3(SU(4)) \simeq \b Z$, there exists $p_1(\tau_1, \tau_2)\in \b Z$ such that $p_1(W;\tau(\tau_1, \tau_2))=p_1(\tau_1, \tau_2)[W,\partial W]$. Let us call $p_1(\tau_1, \tau_2)$ the \textit{first relative Pontrjagin number of $\tau_1$ and $\tau_2$}. \\

Similarly, define the \textit{Pontrjagin number $p_1(\tau)$ of a parallelization $\tau$} of a closed connected oriented 3-manifold $M$, by taking a connected oriented 4-manifold $W$ with boundary $M$, a collar of $\partial W$ identified with $]0,1] \times M$ and $\vec n$ as the outward normal vector field over $\partial W$. \\

We have not actually defined the sign of the Pontrjagin numbers. We will not give details here on how to define it, instead we refer to \cite[\S 15]{MS} or \cite[p.44]{lescopEFTI}. Let us only mention that $p_1$ is the opposite of the second Chern class $c_2$ of the complexified tangent bundle.

\subsection{Pontrjagin numbers for pseudo-parallelizations}
\begin{definition} \label{def_complextriv}
Let $\bar \tau=(N(\gamma); \tau_e, \tau_d, \tau_g)$ be a pseudo-parallelization of a compact oriented 3-manifold $M$, a \textit{complex trivialization} $\bar \tau _{\b C}$ associated to $\bar\tau$ is a trivialization of $TM \otimes \b C$ such that~:
\begin{enumerate}[\textbullet]
 \setlength{\itemsep}{0pt}
 \setlength{\parskip}{5pt}
 \item $\bar \tau_\b C$ is \textit{special} (\ie its determinant is one everywhere) with respect to the trivialization of the determinant bundle induced by the orientation of $M$,
 \item on $M\setminus \mathring{N(\gamma)}$, $\bar \tau_\b C = \tau_e \otimes 1_\b C$,
 \item for $m=(t,c,u)\in [a,b] \times \gamma \times [-1,1]$, $\bar \tau_\b C (m,.)= (m,\tau_d (t,c,u)(\l G(t,u)(.)))$,
\end{enumerate}
where $\l G$ is a map so that :
$$
\l G : \left\lbrace
\begin{aligned}
\ [a,b] \times [-1,1] &\longrightarrow \ SU(3) \\
(t,u) &\longmapsto \left\lbrace
			    \begin{aligned}
			    & \mbox{Id}_{SU(3)}      &\mbox{for $|u| > 1 - \varepsilon$} \\
			    & \mbox{Id}_{SU(3)}   &\mbox{for $t < a+\varepsilon$} \\
			    &  R_{e_1,-\pi-\theta(u)}    &\mbox{for $t>b -\varepsilon$}.
			    \end{aligned}
		    \right. 
\end{aligned}
\right.
$$
Note that such a smooth map $\l G$ on $[a,b]\times[-1,1]$ exists since $\pi_1(SU(3))=\lbrace 1 \rbrace$. Moreover, $\l G$ is well-defined up to homotopy since $\pi_2(SU(3))=\lbrace 0 \rbrace$.
\end{definition}

Pseudo-parallelizations, or pseudo-trivializations, have been first used in \cite[Section 4.3]{lescopKT}, but they have been first defined in \cite[Section 10]{lescopcube}. Note that our conventions are slightly different.

\begin{definition}
Let $\bar\tau_1$ and $\bar\tau_2$ be pseudo-parallelizations of two compact oriented 3-manifolds $M_1$ and $M_2$ with identified boundaries and let $W$ be a cobordism from $M_1$ to $M_2$ with signature zero. As in the case of genuine parallelizations, define a trivialization $\tau(\bar\tau_1, \bar\tau_2)$ of $TW \otimes \b C$ over $\partial W$ using the special complex trivializations $\bar\tau_{1,\b C}$ and $\bar\tau_{2,\b C}$ associated to $\bar\tau_1$ and $\bar\tau_2$, respectively. The \textit{first relative Pontrjagin number of $\bar\tau_1$ and $\bar\tau_2$} is the Pontrjagin obstruction $p_1(\bar\tau_1,\bar\tau_2)$ to extending the trivialization $\tau(\bar\tau_1, \bar\tau_2)$ as a trivialization of $TW\otimes \b C$ over $W$.
\end{definition}

Finally, if $\bar\tau$ is a pseudo-parallelization of a closed oriented 3-manifold $M$, then define the \textit{Pontrjagin number $p_1(\bar\tau)$ of the pseudo-parallelization $\bar\tau$} as $p_1(\tau_\emptyset, \bar \tau)$ as before. 

\subsection[Variation of $p_1$ as an intersection of three 4-chains]{Variation of $p_1$ as an intersection of three 4-chains}
In this subsection, we give a proof of Theorem~\ref{prop_varasint}, which expresses the relative Pon\-trjagin numbers (resp. the variation of Pontrjagin numbers) of pseudo-parallelizations in compact (resp. closed) oriented 3-manifolds as an algebraic intersection of three 4-chains.

\begin{lemma} \label{simppara}
If $\bar \tau=(N(\gamma);\tau_e,\tau_d,\tau_g)$ is a pseudo-parallelization of a compact oriented 3-manifold $M$, if $E_1^d$ and $E_1^g$ are its Siamese sections and if $E_2^e$ denotes the second vector of $\tau_e$, then $P(e_2^M({E_1^d}^\perp,{E_2^e}_{|\partial M})) = -P(e_2^M({E_1^g}^\perp,{E_2^e}_{|\partial M})) = [\gamma]$ in $H_1(M;\b Z)$.
\end{lemma}
\begin{proof}
Since $E_1^d$ and $E_1^e$ coincide on $M\setminus \mathring N(\gamma)$, the obstruction to extending $E_2^e$ as a section of ${E_1^d}^\perp$ is the obstruction to extending $E_2^e$ as a section of ${E_1^d}^\perp_{|N(\gamma)}$. However, parallelizing $N(\gamma)$ with $\tau_d$ and using that 
$$
\tau_d = \left\lbrace
\begin{aligned}
 & \tau_e &\mbox{ on } \partial (\left[ a , b \right] \times \gamma \times \left[ -1 , 1 \right]) \setminus \lbrace b \rbrace \times \gamma \times \left[ -1 , 1 \right] \\
 & \tau_e \circ \l T_\gamma &\mbox{ on }  \lbrace b \rbrace \times \gamma \times \left[ -1 , 1 \right]
\end{aligned}
\right.
$$
we get that $E_2^e$ induces a degree +1 map ${E_2^e}_{|\alpha} : \alpha \rightarrow \b S^1$ on any meridian $\alpha$ of $N(\gamma)$. It follows that
$$
P(e_2^M({E_1^d}^\perp,{E_2^e}_{|\partial M})) = + [\gamma].
$$
Similarly, parallelizing $N(\gamma)$ with $\tau_g$, $E_2^e$ induces a degree -1 map ${E_2^e}_{|\alpha} : \alpha \rightarrow \b S^1$ on any meridian $\alpha$ of $N(\gamma)$, so that
$$
P(e_2^M({E_1^g}^\perp,{E_2^e}_{|\partial M})) = - [\gamma].
$$
\end{proof}

Recall that for a combing $X$ of a compact oriented 3-manifold $M$ and a pseudo-parallelization $\bar\tau$ of $M$, if $X$ and $\bar\tau$ are compatible, then
$$
L_{\bar\tau=X}=\frac{L_{E_1^d=X}+L_{E_1^g=X}}{2} \mbox{ \ \ and \ \ } L_{\bar\tau=-X}=\frac{L_{E_1^d=-X}+L_{E_1^g=-X}}{2}
$$
where $E_1^d$ and $E_1^g$ denote the Siamese sections of $\bar\tau$.

\begin{lemma} \label{nulexcep}
Let $\bar \tau=(N(\gamma);\tau_e,\tau_d,\tau_g)$ be a pseudo-parallelization of a compact oriented 3-manifold $M$. If $X$ is a torsion combing of $M$ compatible with $\bar \tau$, then $L_{\bar\tau=X}$ and $L_{\bar\tau=-X}$ are rationally null-homologous in $M$.
\end{lemma}
\begin{proof}
Let $E_1^d$ and $E_1^g$ be the Siamese sections of $\bar\tau$. Using Proposition~\ref{prop_linksinhomologyI} and the fact that $X$ is a torsion combing, we get, in $H_1(M;\b Q)$,
$$
\begin{aligned}
 \ 2 \cdot [L_{X=-E_1^d}+L_{X=-E_1^g} ] &=  [-P(e_2^M({E_1^d}^\perp,{E_2^e}_{|\partial M})) -P(e_2^M({E_1^g}^\perp,{E_2^e}_{|\partial M})) ] \\
 \ 2 \cdot [L_{X=E_1^d}+L_{X=E_1^g} ] &= [P(e_2^M({E_1^d}^\perp,{E_2^e}_{|\partial M})) + P(e_2^M({E_1^g}^\perp,{E_2^e}_{|\partial M})) ]
\end{aligned}
$$
where $E_2^e$ is the second vector of $\tau_e$. Conclude with Lemma~\ref{simppara}.
\end{proof}

\begin{definition} \label{def_Ft}
Let $X$ and $Y$ be $\partial$-compatible combings of a compact oriented 3-manifold $M$. For all $t \in [0,1]$, let $\bar F_t(X,Y)$ denote the 4-chain of $[0,1]\times UM$~:
$$
\bar F_t(X,Y) = [0,t] \times X(M) + \lbrace t \rbrace \times \bar F(X,Y) + [t,1] \times Y(M),
$$
where $\bar F(X,Y)$ is a 4-chain of $UM$ as in Lemma~\ref{lem_phomotopy}. Note that :
$$
\partial \bar F_t(X,Y) = \lbrace 1 \rbrace \times Y(M) - \lbrace 0 \rbrace \times X(M) - [0,1] \times X(\partial M) + \lbrace t \rbrace \times UM_{L_{X=-Y}}.
$$
\end{definition}

\begin{lemma} \label{lem_transfert}
Let $\bar \tau=(N(\gamma);\tau_e,\tau_d,\tau_g)$ be a pseudo-parallelization of a compact oriented 3-manifold $M$. If $X$ is a torsion combing of $M$ compatible with $\bar \tau$, then there exist 4-chains of $[0,1]\times UM$, $C_4^\pm(\bar\tau,X)$ and $C_4^\pm(X,\bar\tau)$, with boundaries :
$$
\begin{aligned}
 \partial C_4^\pm(\bar\tau,X) &= \lbrace 1 \rbrace \times (\pm X)(M) - \lbrace 0 \rbrace \times \bar\tau(M\times \lbrace \pm e_1 \rbrace ) - [0,1] \times (\pm X)(\partial M) \\
 \partial C_4^\pm(X,\bar\tau) &= \lbrace 1 \rbrace \times \bar\tau(M\times \lbrace \pm e_1 \rbrace ) - \lbrace 0 \rbrace \times (\pm X)(M) - [0,1] \times (\pm X)(\partial M) \\
\end{aligned}
$$
\end{lemma}
\begin{proof}
Let $E_1^d$ and $E_1^g$ be the Siamese sections of $\bar\tau$ and just set
$$
\begin{aligned}
C_4^\pm(\bar\tau,X) &= \sfrac{1}{2} \cdot \left( \bar F_t( \pm E_1^d, \pm X) + \bar F_t( \pm E_1^g, \pm X) - \lbrace t \rbrace \times UM_{|\Sigma(\pm e_1)} \right) \\
C_4^\pm(X,\bar\tau) &= \sfrac{1}{2} \cdot \left( \bar F_t( \pm X, \pm E_1^d) + \bar F_t( \pm X, \pm E_1^g) - \lbrace t \rbrace \times UM_{|-\Sigma(\pm e_1)} \right)
\end{aligned}
$$
where the 4-chains $\bar F_t$ are as in Definition~\ref{def_Ft} and where $\Sigma(\pm e_1)$ are rational 2-chains of $M$ bounded by $\pm(L_{E_1^d=-X}+L_{E_1^g=-X})$, which are rationally null-homologous according to Lemma~\ref{nulexcep}.
\end{proof}

\begin{remark} \label{rmkpcase}
Recall that a genuine parallelization $\tau$ of a compact oriented 3-manifold is a pseudo-parallelization whose link is empty. In such a case, $E_1^d$ and $E_1^g$ are the first vector $E_1^\tau$ of the parallelization $\tau$ and the chains $C_4^{\pm}$ can be simply defined as
$$
\begin{aligned}
C_4^\pm(\tau,X) &= \bar F_t( \pm E_1^\tau, \pm X) - \lbrace t \rbrace \times UM_{|\Sigma(\pm e_1)} \\
C_4^\pm(X,\tau) &= \bar F_t( \pm X, \pm E_1^\tau) - \lbrace t \rbrace \times UM_{|-\Sigma(\pm e_1)}
\end{aligned}
$$
where the 4-chains $\bar F_t$ are as in Definition~\ref{def_Ft} and where $\Sigma(\pm e_1)$ are rational 2-chains of $M$ bounded by $\pm L_{E_1^\tau=-X}$.
\end{remark}

\begin{lemma} \label{lem_C4pme1}
Let $\tau$ and $\bar \tau$ be two pseudo-parallelizations of a compact oriented 3-manifold $M$ that coincide on $\partial M$ and whose links are disjoint. For all $v\in \b S^2$, there exists a 4-chain $C_4(M,\tau,\bar\tau ;v)$ of $[0,1] \times UM$ such that
$$
 \partial C_4(M,\tau,\bar\tau ;v) = \lbrace 1 \rbrace \times \bar \tau (M \times \lbrace v \rbrace)- \lbrace 0 \rbrace \times \tau (M \times \lbrace v \rbrace) - [0,1] \times \tau (\partial M\times \lbrace v \rbrace).
$$
\end{lemma}
\begin{proof}
Let us write $C_4(\tau,\bar\tau ;v)$ instead of $C_4(M,\tau,\bar\tau ;v)$ when there is no ambiguity. Since the 3-chains $\partial C_4(\tau,\bar\tau ;v)$, where $v \in \b S^2$, are homologous, it is enough to prove the existence of $C_4(\tau,\bar\tau ;e_1)$. First, let $X$ be a combing of $M$ such that $X$ is compatible with $\tau$ and $\bar \tau$. In general, this combing is not a torsion combing. Second, let $E_1^d$ and $E_1^g$ (resp. $\bar E_1^d$ and $\bar E_1^g$) denote the Siamese section of $\tau$, (resp. $\bar\tau$) and set 
$$
\begin{aligned}
\bar F(\tau,X) &= \sfrac{1}{2} \cdot \left( \bar F( E_1^d,X) + \bar F( E_1^g,X) \right) \\
\bar F(X,\bar\tau) &= \sfrac{1}{2} \cdot \left( \bar F(X, \bar E_1^d) + \bar F(X, \bar E_1^g)\right).
\end{aligned}
$$
These chains have boundaries :
$$
\begin{aligned}
\partial \bar F(\tau,X) &= X(M) - \tau(M\times \lbrace e_1 \rbrace) + UM_{|L_{\tau=-X}} \\
\partial \bar F(X,\bar\tau) &= \bar\tau(M\times \lbrace e_1 \rbrace) - X(M) + UM_{|-L_{\bar\tau=-X}}. 
\end{aligned}
$$
Hence, for all $t \in [0,1]$, the 4-chain of $[0,1]\times UM$
$$
\bar F_t(\tau, \bar\tau; e_1) = [0,t] \times \tau(M\times \lbrace e_1 \rbrace) + \lbrace t \rbrace \times \left(\bar F(\tau,X) + \bar F(X,\bar\tau)\right) + [t,1] \times \bar\tau(M\times \lbrace e_1 \rbrace)
$$
has boundary :
$$
\begin{aligned}
\partial \bar F_t(\tau, \bar\tau; e_1) &= \lbrace 1 \rbrace \times \bar\tau(M\times \lbrace e_1 \rbrace) - \lbrace 0 \rbrace \times \tau(M\times \lbrace e_1 \rbrace) \\
&- [0,1] \times \tau(\partial M\times \lbrace e_1 \rbrace) + \lbrace t \rbrace \times UM_{|L_{\tau=-X}\cup - L_{\bar\tau=-X}}.
\end{aligned}
$$
Thanks to Proposition~\ref{prop_linksinhomologyI} and Lemma~\ref{simppara}, in $H_1(M;\b Q)$ :
$$
\begin{aligned}
 4 \cdot [ L_{\tau=-X} ] &= 2\cdot [L_{E_1^d=-X} + L_{E_1^g=-X}] \\
 &= [-2 \cdot P(e_2^M(X^\perp,{E_2^e}_{|\partial M}))\hspace{-1mm}+\hspace{-1mm}P(e_2^M({E_1^d}^\perp,{E_2^e}_{|\partial M}))\hspace{-1mm}+\hspace{-1mm}P(e_2^M({E_1^g}^\perp,{E_2^e}_{|\partial M}))] \\
 &= -2 \cdot[P(e_2^M(X^\perp,{E_2^e}_{|\partial M}))]
\end{aligned}
$$
where $E_2^e$ is the second vector of $\tau_e$. Similarly, $2\cdot [L_{\bar\tau=-X}]=-[P(e_2^M(X^\perp,{E_2^e}_{|\partial M}))]$ in $H_1(M;\b Q)$. So, the link $L_{\tau=-X}\cup - L_{\bar\tau=-X}$ is rationally null-homologous in $M$, \ie there exists a rational 2-chain $\Sigma(\tau,\bar\tau)$ such that $\partial \Sigma(\tau,\bar\tau) = L_{\tau=-X}\cup - L_{\bar\tau=-X}$. Hence, we get a 4-chain $C_4(\tau,\bar\tau;e_1)$ as desired by setting 
$$
C_4(\tau,\bar\tau;e_1) = \bar F_t(\tau, \bar\tau; e_1) - \lbrace t \rbrace \times UM_{|\Sigma(\tau,\bar\tau)}.
$$
\end{proof}

\begin{lemma} \label{lem_welldefined}
Let $\tau$ and $\bar \tau$ be two pseudo-parallelizations of a compact oriented 3-manifold $M$ that coincide on $\partial M$. If $x$, $y$ and $z$ are points in $\b S^2$ with pairwise different distances to $e_1$,  then there exist pairwise transverse 4-chains $C_4(\tau,\bar\tau ;x)$, $C_4(\tau,\bar\tau ;y)$ and $C_4(\tau,\bar\tau ;z)$ as in Lemma~\ref{lem_C4pme1} and the algebraic intersection $\langle C_4(\tau,\bar\tau ;x), C_4(\tau,\bar\tau ;y), C_4(\tau,\bar\tau ;z) \rangle_{[0,1]\times UM}$ only depends on $\tau$ and~$\bar\tau$.
\end{lemma}
\begin{proof}
Pick any $x$, $y$ and $z$ in $\b S^2$ with pairwise different distances to $e_1$ and consider some 4-chains $C_4(\tau,\bar\tau ;x)$, $C_4(\tau,\bar\tau ;y)$ and $C_4(\tau,\bar\tau ;z)$ such that, for $v \in \lbrace x , y , z \rbrace$,
$$
\partial C_4(\tau,\bar\tau ;v) = \lbrace 1 \rbrace \times \bar\tau(M\times\lbrace v \rbrace) - \lbrace 0 \rbrace \times \tau (M\times \lbrace v \rbrace ) - [0,1] \times \tau (\partial M \times \lbrace v \rbrace ).
$$
The intersection of $C_4(\tau,\bar\tau;x)$, $C_4(\tau,\bar\tau;y)$ and $C_4(\tau,\bar\tau;z)$ is in the interior of $[0,1]\times UM$. The algebraic triple intersection of these three 4-chains only depends on the fixed boundaries and on the homology classes of the 4-chains. The space $H_4([0,1]\times UM ; \b Q)$ is generated by the classes of 4-chains $\Sigma \times \b S^2$ where $\Sigma$ is a surface in $M$. If $\Sigma \times \b S^2$ is such a 4-chain, then
$$
\begin{aligned}
 & \langle C_4(\tau,\bar\tau ;x) + \Sigma\times \b S^2, C_4(\tau,\bar\tau ;y), C_4(\tau,\bar\tau ;z) \rangle - \langle C_4(\tau,\bar\tau ;x) , C_4(\tau,\bar\tau ;y), C_4(\tau,\bar\tau ;z) \rangle \\
 &= \langle \Sigma\times \b S^2, C_4(\tau,\bar\tau ;y), C_4(\tau,\bar\tau ;z) \rangle \\
 &= \left\lbrace
 \begin{aligned}
 & \langle \Sigma\times \b S^2, [0,1] \times \tau(M \times \lbrace y \rbrace ) , [0,1] \times \tau(M \times \lbrace z \rbrace) \rangle  \mbox{ , pushing $\Sigma \times \b S^2$ near $0$,}  \\
 & \langle \Sigma\times \b S^2, [0,1] \times \bar\tau(M \times \lbrace y \rbrace ) , [0,1] \times \bar \tau(M \times \lbrace z \rbrace) \rangle  \mbox{ , pushing $\Sigma \times \b S^2$ near $1$.} \\
 \end{aligned}
 \right.
\end{aligned}
$$
Hence, $\langle \Sigma\times \b S^2, C_4(\tau,\bar\tau ;y), C_4(\tau,\bar\tau ;z) \rangle $ is independent of $\tau$ and $\bar\tau$. So, use Lemma~\ref{lem_extendpparallelization} to extend a trivialization of $TM_{|\Sigma}$ as a pseudo-parallelization $\tau'$ that coincides with $\tau$ and $\bar\tau$ on $\partial M$. Considering this pseudo-parallelization we get
$$
\langle \Sigma\times \b S^2, C_4(\tau,\bar\tau ;y), C_4(\tau,\bar\tau ;z) \rangle \italicegal \langle \Sigma\times \b S^2, [0,1] \times \tau'(M \times \lbrace y \rbrace ) , [0,1] \times \tau'(M \times \lbrace z \rbrace) \rangle \italicegal 0,
$$
so that the algebraic triple intersection of the three chains $C_4(\tau,\bar\tau;x)$, $C_4(\tau,\bar\tau;y)$ and $C_4(\tau,\bar\tau;z)$ only depends on their fixed boundaries.
\end{proof}

\begin{proof}[Proof of Theorem~\ref{prop_varasint}]
Let $\tau$ and $\bar \tau$ be two pseudo-parallelizations of a compact oriented 3-manifold $M$ that coincide on $\partial M$ and whose links are disjoint. To conclude the proof of Theorem~\ref{prop_varasint}, we have to prove that for any $x$, $y$ and $z$ in $\b S^2$ with pairwise different distances to $e_1$ :
$$
p_1(\tau, \bar\tau)= 4 \cdot \langle C_4(M,\tau,\bar\tau ;x),C_4(M,\tau,\bar\tau ;y),C_4(M,\tau,\bar\tau ;z) \rangle_{[0,1]\times UM}.
$$

First, we know from \cite[Lemma 10.9]{lescopcube} that this is true if $M$ is a $\b Q$HH of genus 1. Notice that it is also true if $M$ embeds in such a manifold. Indeed, if $\b H$ is a $\b Q$HH of genus 1 and if $M$ embeds in $\b H$ then, using Lemma~\ref{lem_extendpparallelization} and using that $\tau$ and $\bar \tau$ coincide on $\partial M$, there exists a pseudo-parallelization $\check \tau$ of $\b H \setminus \mathring M$ such that
$$
\bar \tau_{\b H} : \left\lbrace
\begin{aligned}
m \in M &\mapsto \bar\tau(m,.) \\
m \in \b H\setminus \mathring M & \mapsto \check \tau(m,.)
\end{aligned} \right.
\mbox{ \ and \ }
\tau_{\b H} : \left\lbrace
\begin{aligned}
m \in M &\mapsto \tau(m,.) \\
m \in \b H\setminus \mathring M & \mapsto \check \tau(m,.)
\end{aligned} \right.
$$
are pseudo-parallelizations of $\b H$. Furthermore, for any $v \in \b S^2$, let $C_4(\b H,v)$ be the 4-chain of $[0,1]\times U\b H$ :
$$
C_4(\b H,\tau_{\b H},\bar\tau_{\b H} ;v) = C_4(M,\tau,\bar\tau ;v)\cup [0,1] \times \check \tau( (\b H \setminus \mathring M) \times \lbrace v \rbrace )
$$
where $C_4(M,\tau,\bar\tau ;v)$ is as in Lemma~\ref{lem_C4pme1}. The boundary of $C_4(\b H,\tau_{\b H},\bar\tau_{\b H} ;v)$ is~:
$$
\partial C_4(\b H,\tau_{\b H},\bar\tau_{\b H} ;v) = \lbrace 1 \rbrace \times \bar\tau_{\b H}(M\times \lbrace v \rbrace ) -  \lbrace 0 \rbrace \times \tau_{\b H}(M\times \lbrace v \rbrace ) - [0,1] \times \tau_{\b H}(\partial M\times \lbrace v \rbrace ).
$$
Using the definition of Pontrjagin numbers of pseudo-parallelizations and the hypothesis on $\b H$, it follows that if $x,y$ and $z$ are points in $\b S^2$ with pairwise different distances to $e_1$ :
$$
\begin{aligned}
p_1(\tau,\bar\tau) = p_1(\tau_{\b H},\bar\tau_{\b H})
= 4 \cdot \langle C_4(\b H,\tau_{\b H},\bar\tau_{\b H} ;x) , C_4(\b H,\tau_{\b H},\bar\tau_{\b H} ;y) , C_4(\b H,\tau_{\b H},\bar\tau_{\b H} ;z) \rangle.
\end{aligned}
$$
Now note that : 
$$
\langle [0,1] \times \check \tau (\b H\setminus \mathring M \times \lbrace x \rbrace) , [0,1] \times \check \tau (\b H\setminus \mathring M \times \lbrace y \rbrace ), [0,1] \times \check \tau (\b H\setminus \mathring M \times \lbrace z \rbrace) \rangle = 0.
$$
Indeed, if $\check \tau = (N(\check \gamma); \check \tau_e,\check \tau_d,\check \tau_g)$, for all $v\in \b S^2$ the pseudo-section of $\check\tau$ reads
$$
\begin{aligned}
\check \tau (M\times \lbrace v \rbrace ) &= \check\tau_e((M\setminus \mathring N(\check\gamma))\times \lbrace v \rbrace ) \\
&+ \frac{\check \tau_d(N(\check\gamma)\times \lbrace v \rbrace) + \check\tau_g(N(\check\gamma)\times \lbrace v \rbrace) + \check\tau_e( \lbrace b \rbrace \times \check\gamma \times C_2(v)) }{2}.
\end{aligned}
$$
The 3-chains $\check\tau_e((M\setminus \mathring N(\check\gamma))\times \lbrace v \rbrace )$, for $v\in \lbrace x,y,z \rbrace$, are pairwise disjoint since $\check\tau_e$ is a genuine parallelization and since $x,y$ and $z$ are pairwise distinct points in $\b S^2$. Moreover, the 3-chains $\check\tau_e( \lbrace b \rbrace \times \check\gamma \times C_2(v))$, for $v\in \lbrace x,y,z\rbrace$, are also pairwise disjoint since they are subsets of the $\check\tau_e(\lbrace b \rbrace \times \check\gamma \times \b S^1(v))$ , $v \in \lbrace x,y,z \rbrace$, which are pairwise disjoint since $x,y$ and $z$ have pairwise different distances to $e_1$. Finally, we have~:
$$
\begin{aligned}
\langle \check\tau_d(N(\check \gamma) \times \lbrace x\rbrace)+\check\tau_g(N(\check \gamma) \times \lbrace x\rbrace) , \check\tau_d(N(\check \gamma) \times \lbrace y\rbrace)+\check\tau_g(N(\check \gamma) \times \lbrace y\rbrace) , \\ \check\tau_d(N(\check \gamma) \times \lbrace z\rbrace)+\check\tau_g(N(\check \gamma) \times \lbrace z\rbrace) \rangle 
&= 0
\end{aligned}
$$
since a triple intersection between the 3-chains 
$$
\lbrace \check\tau_d(N(\check \gamma) \times \lbrace v\rbrace)+\check\tau_g(N(\check \gamma) \times \lbrace v\rbrace)\rbrace_{v\in \lbrace x,y,z\rbrace}
$$ would be contained in an intersection between two of the $\lbrace \check\tau_d(N(\check \gamma)\times \lbrace v\rbrace)\rbrace_{v\in \lbrace x,y,z \rbrace}$ or between two of the $\lbrace \check\tau_g(N(\check \gamma)\times \lbrace v\rbrace)\rbrace_{v\in \lbrace x,y,z \rbrace}$ which must be empty since $\check\tau_d$ and $\check\tau_g$ are genuine parallelizations. It follows that 
$$
\begin{aligned}
p_1(\tau,\bar\tau) &= 4 \cdot\langle C_4(\b H,\tau_{\b H},\bar\tau_{\b H} ;x) , C_4(\b H,\tau_{\b H},\bar\tau_{\b H} ;y) , C_4(\b H,\tau_{\b H},\bar\tau_{\b H} ;z) \rangle \\
&= 4 \cdot\langle C_4(M,\tau,\bar\tau ;x) , C_4(M,\tau,\bar\tau ;y) , C_4(M,\tau,\bar\tau ;z) \rangle.
\end{aligned}
$$
Using the same construction, note also that it is enough to prove the statement when $M$ is a closed oriented 3-manifold since any oriented 3-manifold embeds into a closed one. \\

Let us finally prove Theorem~\ref{prop_varasint} when $M$ is a closed oriented 3-manifold. Consider a Heegaard splitting $M = H_1 \cup_\Sigma H_2$ such that there is a collar $\Sigma\times[0,1] \subset H_2$ of $\Sigma$ verifying
$$
N(\bar \gamma) \cap \left(\Sigma\times[0,1] \right) = \emptyset \mbox{ \ \ and \ \ } N( \gamma) \cap \left(\Sigma\times[0,1] \right) = \emptyset
$$
where $\gamma$ and $\bar\gamma$ are the links of $\tau$ and $\bar\tau$, respectively, and such that $\Sigma=\Sigma\times \lbrace 0 \rbrace$. Such a splitting can be obtained by considering a triangulation of $M$ containing $\gamma$ and $\bar\gamma$ in its 1-skeleton, and then defining $H_1$ as a tubular neighborhood of this 1-skeleton. 

Using Lemma~\ref{lem_extendpparallelization}, we can construct a pseudo-parallelization $\tau^c$ of $\Sigma\times [0,1]$ such that $\tau^c$ coincides with $\bar\tau$ on $\Sigma \times \lbrace 1 \rbrace$ and with $\tau$ on $\Sigma \times \lbrace 0 \rbrace$. Then, write $H'_1 = H_1 \cup (\Sigma\times[0,1])$ and $H'_2=H_2\setminus( \Sigma\times[0,1[)$ -- see Figure~\ref{figHH} -- and set
$$
\check \tau : \left\lbrace
\begin{aligned}
 (m,v) \in UH_1 &\longmapsto \tau(m,v) \\
 (m,v) \in U(\Sigma\times[0,1]) &\longmapsto \tau^c(m,v) \\
 (m,v) \in UH'_2 &\longmapsto \bar\tau(m,v).
\end{aligned}
\right.
$$
\begin{center}
\definecolor{zzttqq}{rgb}{0.6,0.2,0}
\begin{tikzpicture}[line cap=round,line join=round,>=triangle 45,x=1.0cm,y=0.75cm]
\clip(-3.5,-3.5) rectangle (5.5,3.5);
\fill[line width=0pt,color=zzttqq,fill=zzttqq,fill opacity=0.15] (-3,2) -- (0,2) -- (0,-2) -- (-3,-2) -- cycle;
\fill[line width=0pt,color=zzttqq,fill=zzttqq,fill opacity=0.05] (0,2) -- (2,2) -- (2,-2) -- (0,-2) -- cycle;
\fill[line width=0pt,color=zzttqq,fill=zzttqq,fill opacity=0.15] (2,2) -- (5,2) -- (5,-2) -- (2,-2) -- cycle;
\draw (0,2)-- (0,-2);
\draw (2,2)-- (2,-2);
\draw (0,2.4)-- (0,2.6);
\draw (0,2.6)-- (5,2.6);
\draw (5,2.4)-- (5,2.6);
\draw (2,-2.4)-- (2,-2.6);
\draw (2,-2.6)-- (-3,-2.6);
\draw (-3,-2.6)-- (-3,-2.4);
\begin{scriptsize}
\draw (-1.75,0.25) node[anchor=north west] {$H_1$};
\draw (-1.75,-0.5) node[anchor=north west] {$\tau$};
\draw (3.25,0.25) node[anchor=north west] {$H'_2$};
\draw (3.25,-0.5) node[anchor=north west] {$\bar\tau$};
\draw (0.75,-0.5) node[anchor=north west] {$\tau^c$};
\draw (-0.75,-2.75) node[anchor=north west] {$H'_1$};
\draw (2.44,3.11) node[anchor=north west] {$H_2$};
\draw (-0.5,-2) node[anchor=north west] {$\Sigma\times\lbrace 0\rbrace$};
\draw (1.5,2.5) node[anchor=north west] {$\Sigma\times \lbrace 1 \rbrace$};
\end{scriptsize}
\end{tikzpicture}
\captionof{figure}
{}\label{figHH}
\end{center}
For $v \in \b S^2$, consider some 4-chains $C_4(H_1,\tau,\check\tau ;v)$, $C_4(H_2,\tau,\check\tau ;v)$, $C_4(H'_1,\check\tau,\bar\tau ;v)$ and $C_4(H'_2,\check\tau,\bar\tau ;v)$ of $[0,1]\times UH_1$, $[0,1]\times UH_2$, $[0,1]\times UH'_1$ and $[0,1]\times UH'_2$, respectively, such that :
$$
\begin{aligned}
 \partial C_4(H_1,\tau,\check\tau ;v) &\italicegal \lbrace 1 \rbrace \times \check \tau (H_1 \times \lbrace v \rbrace) - \lbrace 0 \rbrace \times \tau (H_1\times \lbrace v \rbrace ) - [0,1] \times \tau (\partial H_1 \times \lbrace v \rbrace) \\
  \partial C_4(H_2,\tau,\check\tau ;v) &\italicegal \lbrace 1 \rbrace \times \check \tau (H_2 \times \lbrace v \rbrace) - \lbrace 0 \rbrace \times \tau (H_2\times \lbrace v \rbrace ) - [0,1] \times \tau (\partial H_2 \times \lbrace v \rbrace) \\ 
\end{aligned}
$$
and 
$$
\begin{aligned}
 \partial C_4(H'_1,\check\tau,\bar\tau ;v) &\italicegal \lbrace 1 \rbrace \times \bar \tau (H'_1 \times \lbrace v \rbrace) - \lbrace 0 \rbrace \times \check \tau (H'_1\times \lbrace v \rbrace ) - [0,1] \times \check \tau (\partial H'_1 \times \lbrace v \rbrace) \\
  \partial C_4(H'_2,\check\tau,\bar\tau ;v) &\italicegal\lbrace 1 \rbrace \times \bar \tau (H'_2 \times \lbrace v \rbrace) - \lbrace 0 \rbrace \times \check \tau (H'_2\times \lbrace v \rbrace ) - [0,1] \times \check \tau (\partial H'_2 \times \lbrace v \rbrace). \\
\end{aligned}
$$
Since $H_1$ and $H_2$ embed in rational homology balls, for any $x,y$ and $z$ in $\b S^2$ with pairwise different distances to $e_1$ 
$$
\begin{aligned}
 p_1(\tau_{|H_1},\check\tau_{|H_1}) &= 4 \cdot\langle C_4(H_1,\tau,\check\tau ;x) , C_4(H_1,\tau,\check\tau ;y) , C_4(H_1,\tau,\check\tau ;z) \rangle_{[0,1]\times UH_1} \\
 p_1(\tau_{|H_2},\check\tau_{|H_2}) &= 4 \cdot\langle C_4(H_2,\tau,\check\tau ;x)  ,  C_4(H_2,\tau,\check\tau ;y) ,  C_4(H_2,\tau,\check\tau ;z) \rangle_{[0,1]\times UH_2} \\
\end{aligned}
$$
so that, using $C_4(M,\tau,\check\tau;v) = C_4(H_1,\tau,\check\tau;v)+C_4(H_2,\tau,\check\tau;v)$ for $v\in \b S^2$, 
$$
p_1(\tau,\check\tau) = 4 \cdot\langle C_4(M,\tau,\check\tau ;x) , \ C_4(M,\tau,\check\tau ;y) , \ C_4(M,\tau,\check\tau ;z) \rangle_{[0,1]\times UM}.
$$
Similarly, since $H'_1$ and $H'_2$ embed in rational homology balls, for any $x,y$ and $z$ in $\b S^2$ with pairwise different distances to $e_1$ 
$$
\begin{aligned}
 p_1(\check\tau_{|H'_1},\bar\tau_{|H'_1}) &= 4 \cdot\langle C_4(H'_1,\check\tau,\bar\tau ;x) , C_4(H'_1,\check\tau,\bar\tau ;y) , C_4(H'_1,\check\tau,\bar\tau ;z)   \rangle _{[0,1]\times UH'_1} \\
 p_1(\check\tau_{|H'_2},\bar\tau_{|H'_2}) &= 4 \cdot\langle C_4(H'_2,\check\tau,\bar\tau ;x) ,  C_4(H'_2,\check\tau,\bar\tau ;y) ,  C_4(H'_2,\check\tau,\bar\tau ;z)\rangle_{[0,1]\times UH'_2} \\
\end{aligned}
$$
so that, using $C_4(M,\check\tau,\bar\tau;v) = C_4(H'_1,\check\tau,\bar\tau;v)+C_4(H'_2,\check\tau,\bar\tau;v)$ for $v\in \b S^2$,
$$
p_1(\check\tau,\bar\tau) = 4 \cdot\langle C_4(M,\check\tau,\bar\tau;x), C_4(M,\check\tau,\bar\tau;y), C_4(M,\check\tau,\bar\tau;z) \rangle_{[0,1]\times UM}.
$$
Eventually, reparameterizing and stacking $C_4(M,\tau,\check\tau;v)$ and $C_4(M,\check\tau,\bar\tau;v) $, for all $v \in \b S^2$ we get a 4-chain $C_4(M,\tau,\bar\tau; v)$ of $[0,1]\times UM$ such that
$$
\partial C_4(M,\tau,\bar\tau;v) = \lbrace 1 \rbrace \times \bar \tau (M \times \lbrace v \rbrace) - \lbrace 0 \rbrace \times \tau(M\times \lbrace v \rbrace) - [0,1]\times \tau (\partial M \times \lbrace v \rbrace)
$$
and such that for any $x,y$ and $z$ in $\b S^2$ with pairwise different distances to $e_1$ 
$$
p_1(\tau, \bar \tau)= 4 \cdot\langle C_4(M,\tau,\bar\tau;x), C_4(M,\tau,\bar\tau;y), C_4(M,\tau,\bar\tau;z) \rangle_{[0,1]\times UM}.
$$
\end{proof}

\section{From pseudo-parallelizations to torsion combings}
\subsection{Variation of $p_1$ as an intersection of two 4-chains}
\begin{definition}
Let $M$ be a compact oriented 3-manifold. A trivialization $\rho$ of $TM_{|\partial M}$ is \textit{admissible} if there exists a section $X$ of $UM$ such that $(X,\rho)$ is a torsion combing of $M$.
\end{definition}

\begin{lemma} \label{lem_HtMb}
Let $M$ be a compact oriented 3-manifold, let $\rho$ be an admissible trivialization of $TM_{|\partial M}$ and let $S_1, S_2 , \ldots , S_{\beta_1(M)}$ be surfaces in $M$ comprising a basis of $H_2(M;\b Q)$. The subspace $H_T^\rho(M)$ of $H_2(UM;\b Q)$ generated by $\lbrace [X(S_1)], \ldots, [X(S_{\beta_1(M)})] \rbrace$ where $X$ is a section of $UM$ such that $(X,\rho)$ is a torsion combing of $M$, only depends on $\rho$.
\end{lemma}
\begin{proof}
Let $(Y,\rho)$ be another choice of torsion combing of $M$. Assume, without loss of generality, that $(X,\rho)$ and $(Y,\rho)$ are $\partial$-compatible, and let $C(X,Y)$ be the 4-chain of $UM$ 
$$
C(X,Y) = \bar F(X,Y) - UM_{|\Sigma_{X=-Y}}
$$
constructed using Lemma~\ref{lem_phomotopy} and Proposition~\ref{prop_linksinhomologyI}, which provide $\bar F(X,Y)$ and a 2-chain $\Sigma_{X=-Y}$ of $M$ bounded by $L_{X=-Y}$, respectively. For $i \in \lbrace 1,2,\ldots,\beta_1(M)\rbrace$,
$$
Y(S_i) - X(S_i) = \partial (C(X,Y) \cap UM_{|S_i}).
$$
\end{proof}

\begin{lemma} \label{lem_evaluationII}
Let $M$ be a compact oriented 3-manifold, let $\rho$ be an admissible trivialization of $TM_{|\partial M}$ and let $(X,\rho)$ and $(Y,\rho)$ be $\partial$-compatible torsion combings of $M$. There exists a 4-chain $C_4(X,Y)$ of $[0,1]\times UM$ such that
$$
\partial C_4(X,Y) = \lbrace 1 \rbrace \times Y(M) - \lbrace 0 \rbrace \times X(M) - [0,1] \times X(\partial M).
$$
For any such chain $C_4(X,Y)$, if $C$ is a 2-cycle of $[0,1]\times UM$ then,
 $$
 [C] = \langle C , C_4(X,Y) \rangle_{[0,1] \times UM} [S] \mbox{ \  in $H_2([0,1]\times UM ; \b Q)/H_T^\rho(M)$,}
 $$
 where $[S]$ is the homology class of the fiber of $UM$ in $H_2([0,1]\times UM ; \b Q)$.
\end{lemma}
\begin{proof}
Observe that $H_2(UM;\b Q)$ is generated by the family $\lbrace [Z(S_1)], \ldots, [Z(S_{\beta_1(M)})], [S]  \rbrace$ where $S_1, \ldots , S_{\beta_1(M)}$ are surfaces in $M$ comprising a basis of $H_2(M;\b Q)$ and where $Z$ is a torsion combing of $M$ that coincides with $X$ and $Y$ on $\partial M$. Let $C_4(X,Y)$ be the 4-chain
$$
C_4(X,Y) = \bar F_t(X,Y) - \lbrace t \rbrace \times UM_{|\Sigma_{X=-Y}}
$$
where $\bar F_t(X,Y)$ is a 4-chain as in Definition~\ref{def_Ft} and $\Sigma_{X=-Y}$ is a 2-chain of $M$ bounded by $L_{X=-Y}$ provided by Proposition~\ref{prop_linksinhomologyI}. The chain $C_4(X,Y)$ has the desired boundary. Note that $\langle [S], C_4(X,Y) \rangle = 1$. Moreover, $\langle [Z(\Sigma)], C_4(X,Y) \rangle = 0$ for any surface $\Sigma$ in $M$. Indeed, notice that
$$
\langle [Z(\Sigma)], C_4(X,Y) \rangle = \left\lbrace
\begin{aligned}
&\langle [Z(\Sigma)], [0,1]\times X(M) \rangle  \mbox{ , pushing $Z(\Sigma)$ before $t$,}\\
&\langle [Z(\Sigma)], [0,1]\times Y(M) \rangle  \mbox{ , pushing $Z(\Sigma)$ after $t$.}
\end{aligned}
\right.
$$
As a consequence, $\langle [Z(\Sigma)], C_4(X,Y) \rangle$ is independent of $X$ and $Y$. Let us prove that it is possible to construct a torsion combing $Z'$ that coincides with $X$ and $Y$ on $\partial M$ and such that 
$$
\langle [Z(\Sigma)], [0,1]\times Z'(M) \rangle=0.
$$
Using the parallelization $\rho=(E_1^\rho,E_2^\rho,E_3^\rho)$ of $\partial M$ induced by $X$, define a homotopy 
$$
\l Z : [0,1] \times \partial M \rightarrow [0,1]\times UM_{|\partial M}
$$
from $-Z_{|\partial M}$ to $Z_{|\partial M}$ along the unique geodesic arc passing through $E_3^\rho$. Since $\Sigma$ sits in $\mathring{M}$, we can get a collar $\l C\simeq [0,1]\times \partial M$ of $\partial M$ such that $\l C \cap \Sigma = \emptyset$ and $\lbrace 1 \rbrace \times \partial M = \partial M$. Finally set $Z'$ to coincide with $-Z$ on $M\setminus \mathring{\l C}$ and with the homotopy $\l Z$ on the collar. The combing $(Z',\rho)$ is a torsion combing. Indeed, $E_2^\rho$ can be extended as a nonvanishing section of ${Z'}^\perp_{|\l C}$ so that 
$$
e_2^M(Z'^\perp,E_2^\rho)=e_2^{M\setminus \mathring{\l C}}(Z'^\perp,E_2^\rho)=e_2^M(-Z^\perp,E_2^\rho)=-e_2^M(Z^\perp,E_2^\rho).
$$
Finally, using the torsion combing $Z'$, we get $\langle [Z(\Sigma)], [0,1]\times Z'(M) \rangle = 0$. \\

To conclude the proof, assume that $C'_4(X,Y)$ is a 4-chain with same boundary as the chain $C_4(X,Y)$ we constructed, and let $C$ be a 2-cycle of $[0,1]\times UM$. The 2-cycle $C$ is homologous to a 2-cycle in $\lbrace 1 \rbrace \times UM$. Similarly, $(C'_4(X,Y) - C_4(X,Y))$ is homologous to a 4-cycle in $\lbrace 0 \rbrace \times UM$. Hence, $\langle C , \ C'_4(X,Y) - C_4(X,Y) \rangle = 0$.
\end{proof}

\begin{lemma} \label{bord}
Let $\tau$ and $\bar \tau$ be two pseudo-parallelizations of a compact oriented 3-manifold $M$ that coincide on $\partial M$. Let $C_4(\tau,\bar\tau;\pm e_1)$ denote 4-chains of $[0,1]\times UM$ as in Theorem~\ref{prop_varasint} for $v=\pm e_1$. If the 4-chains $C_4(\tau,\bar\tau;\pm e_1)$ are transverse to each other, then
 $$
 \begin{aligned}
  \partial (C_4(\tau,\bar\tau;e_1)\cap C_4(\tau,\bar\tau;-e_1)) &= \sfrac{1}{4} \cdot \lbrace 1 \rbrace \times \left( \bar E_1^d(L_{\bar E_1^d=-\bar E_1^g}) - (-\bar E_1^d)(L_{\bar E_1^d=-\bar E_1^g}) \right) \\ 
  &- \sfrac{1}{4} \cdot \lbrace 0 \rbrace \times \left( E_1^d(L_{E_1^d=-E_1^g}) - (-E_1^d)(L_{E_1^d=-E_1^g}) \right)
 \end{aligned}
 $$
 where $E_1^d$ and $E_1^g$, resp. $\bar E_1^d$ and $\bar E_1^g$, are the Siamese sections of $\tau$, resp. $\bar\tau$.
\end{lemma}
\begin{proof}
Since $\tau$ and $\bar \tau$ coincide with a trivialization of $TM_{|\partial M}$ on $\partial M$, we have
$$
\begin{aligned}
\partial (C_4(\tau,\bar\tau;e_1)\cap C_4(\tau,\bar\tau;-e_1)) 
&= \lbrace 1 \rbrace \times \left( \bar\tau(M\times \lbrace e_1 \rbrace)\cap \bar\tau(M\times \lbrace -e_1 \rbrace) \right) \\ 
&- \lbrace 0 \rbrace \times \left( \tau(M\times \lbrace e_1 \rbrace)\cap \tau(M\times \lbrace -e_1 \rbrace) \right) \\
&=\sfrac{1}{4} \cdot \lbrace 1 \rbrace \times \left( \bar E_1^d(L_{\bar E_1^d=-\bar E_1^g}) + \bar E_1^g(L_{\bar E_1^g=-\bar E_1^d}) \right) \\
&- \sfrac{1}{4} \cdot \lbrace 0 \rbrace \times \left( E_1^d(L_{E_1^d=-E_1^g}) + E_1^g(L_{E_1^g=-E_1^d}) \right) .
\end{aligned}
$$
\iffalse
&=  ( \lbrace 1 \rbrace \times UM ) \bigcap \left( C_4(\tau,\bar\tau;e_1)\cap C_4(\tau,\bar\tau;-e_1) \right) \\
&- (\lbrace 0 \rbrace \times UM ) \bigcap \left( C_4(\tau,\bar\tau;e_1)\cap C_4(\tau,\bar\tau;-e_1) \right) \\

&=\sfrac{1}{4} \cdot \lbrace 1 \rbrace \times \left( \bar E_1^d(L_{\bar E_1^d=-\bar E_1^g}) - (-\bar E_1^d)(L_{\bar E_1^d=-\bar E_1^g}) \right) \\
&- \sfrac{1}{4} \cdot \lbrace 0 \rbrace \times \left( E_1^d(L_{E_1^d=-E_1^g}) - (-E_1^d)(L_{E_1^d=-E_1^g}) \right).
\fi
\end{proof}

\begin{definition} \label{omega}
Let $\bar\tau=(N(\gamma); \tau_e, \tau_d, \tau_g)$ be a pseudo-parallelization of a compact oriented 3-manifold $M$, and let $E_1^d$ and $E_1^g$ denote its Siamese sections. Recall from Definition \ref{def_addinner} that the map 
$$
\tau_d^{-1} \circ \tau_g : [a,b] \times \gamma \times [-1,1] \times \b R^3 \rightarrow [a,b] \times \gamma \times [-1,1] \times \b R^3
$$
is such that
$$
\begin{aligned}
&\forall t \in [a,b], \ u \in [-1,1], \ c \in \gamma, \ v \in \b R^3 : \\
&\tau_d^{-1} \circ \tau_g ((t,c,u),v) = \l T_\gamma^{-1} ((t,c,u),\l F(t,u)(v)),
\end{aligned}
$$
Hence, $L_{E_1^d=-E_1^g}$ consists in parallels of $\gamma$ of the form $\lbrace t \rbrace \times \gamma \times \lbrace u \rbrace$. For all component $L$ of $L_{E_1^d=-E_1^g}$, there exists a point $e_2^L$ in $\b S^1(e_2)$ such that $L \times \lbrace e_2^L \rbrace =\tau_d^{-1} \circ \tau_g (L \times \lbrace e_2 \rbrace)$. Choose a point $e_2^\Omega$ in $\b S^1(e_2)$ distinct from $e_2$ and from the points $e_2^L$. Finally, set
$$
\Omega(\bar\tau)=-\tau_d(L_{E_1^d=-E_1^g}\times[-e_1,e_1]_{e_2^\Omega})
$$
where $[-e_1,e_1]_{e_2^\Omega}$ is the geodesic arc from $-e_1$ to $e_1$ passing through $e_2^\Omega$. The 2-chain $\Omega(\bar\tau)$ can be seen as the projection of a homotopy from $-E_1^d$ to $E_1^d$ over $L_{E_1^d=-E_1^g}$. Note that 
$$
\partial \Omega(\bar\tau)= (-E_1^d)(L_{E_1^d=-E_1^g})- E_1^d(L_{E_1^d=-E_1^g}).
$$
The choice of $e_2^\Omega$ ensures that $\Omega(\bar\tau) \cap \bar\tau(M\times \lbrace e_2 \rbrace)=\emptyset$. Note that $\Omega(\tau)=\emptyset$ when $\tau$ is a genuine parallelization. 
\end{definition}

\begin{definition} \label{def_pgo}
Let $M$ be a compact oriented 3-manifold and let $\rho$ be an admissible trivialization of $TM_{|\partial M}$. Let $\tau$ and $\bar\tau$ be pseudo-parallelizations of a compact oriented 3-manifold $M$ which coincide with $\rho$ on $TM_{|\partial M}$ and let $C_4(\tau,\bar\tau;\pm e_1)$ denote 4-chains of $[0,1]\times UM$ as in Theorem~\ref{prop_varasint}. Set
$$
\go P(\tau,\bar\tau) = \lbrace 0 \rbrace \times \Omega(\tau) + 4 \cdot (C_4(\tau,\bar\tau;e_1) \cap C_4(\tau,\bar\tau;-e_1)) - \lbrace 1 \rbrace \times \Omega(\bar\tau).
$$
When $(X,\rho)$ is a torsion combing of $M$, let $C_4^+(X,\bar\tau)$ and $C_4^-(X,\bar\tau)$ be 4-chains of $[0,1]\times UM$ as in Lemma~\ref{lem_transfert} and set
$$
\begin{aligned}
\go P(X,\bar\tau) = 4\cdot(C_4^+(X,\bar\tau)\cap C_4^-(X,\bar\tau)) - \lbrace 1 \rbrace \times \Omega(\bar\tau), \\
\go P(\bar\tau,X) = \lbrace 0 \rbrace \times \Omega(\bar\tau) + 4\cdot(C_4^+(\bar\tau,X)\cap C_4^-(\bar\tau,X)).
\end{aligned}
$$
\end{definition}

According to Lemma~\ref{bord} and Definition~\ref{omega}, the 4-chains $\go P(\lambda,\mu)$ of Definition~\ref{def_pgo} above are cycles. In the remaining of this section, we prove that their classes read $p_1(\lambda,\mu)[S]$ in $H_2([0,1]\times UM;\b Q)/H_T^\rho(M)$.

\begin{proposition} \label{prop_ppara}
Let $M$ be a compact oriented 3-manifold, let $\rho$ be an admissible trivialization of $TM_{|\partial M}$ and let $\tau$ and $\bar \tau$ be two pseudo-parallelizations of $M$ that coincide with $\rho$ on $\partial M$. Under the assumptions of Definition~\ref{def_pgo}, the class of $\go P(\tau,\bar\tau)$ in $H_2([0,1]\times UM ; \b Q)/H_T^\rho(M)$ equals $p_1(\tau,\bar\tau)[S]$ where $[S]$ is the homology class of the fiber of $UM$ in $H_2([0,1]\times UM ; \b Q)$.
\end{proposition}
\begin{proof}
The class of $\go P(\tau,\bar\tau)$ in $H_2([0,1]\times UM ; \b Q)/H_T^\rho(M)$ is
$$
 \left[ \go P(\tau,\bar\tau) \right] = \langle \go P(\tau,\bar\tau) , C_4(X,Y) \rangle \cdot [S]
$$
where $(X,\rho)$ and $(Y,\rho)$ are $\partial$-compatible torsion combings of $M$ and where $C_4(X,Y)$ is any 4-chain of $[0,1]\times UM$ as in Lemma~\ref{lem_evaluationII}. Let us construct a specific $C_4(X,Y)$ as follows. Let $C_4(\tau,\bar\tau;e_2)$ be as in Theorem~\ref{prop_varasint} where $e_2=(0,1,0)$. Since, $\partial C_4(\tau,\bar\tau;e_1)$ and $\partial C_4(\tau,\bar\tau;e_2)$ are homologous, it is possible to reparameterize and to stack the 4-chains $C_4^+(X,\tau)$, $C_4(\tau,\bar\tau;e_2)$ and $C_4^+(\bar\tau,Y)$ where the chains $C_4^+(X,\tau)$ and $C_4^+(\bar\tau,Y)$ are as in Lemma~\ref{lem_transfert}. It follows that, in $H_2([0,1]\times UM ; \b Q)/H_T^\rho(M)$,
$$
\begin{aligned}
\left[ \go P(\tau,\bar\tau) \right] &= \langle \go P(\tau,\bar\tau) , C_4(X,Y) \rangle  [S] \\
 &= \langle \go P(\tau,\bar\tau) , C_4(\tau,\bar\tau;e_2) \rangle  [S]  \\
 &= 4 \cdot \langle C_4(\tau,\bar\tau;e_1) \cap C_4(\tau,\bar\tau;-e_1) ,C_4(\tau,\bar\tau;e_2) \rangle  [S] \\
 &+ \langle \lbrace 0 \rbrace \times \Omega(\tau)-  \lbrace 1 \rbrace \times \Omega(\bar\tau),C_4(\tau,\bar\tau;e_2) \rangle [S].
\end{aligned}
$$
Now, note that $ \langle \lbrace 0 \rbrace \times \Omega(\tau)-  \lbrace 1 \rbrace \times \Omega(\bar\tau),C_4(\tau,\bar\tau;e_2) \rangle=0$ since 
$$
\Omega(\tau) \cap \tau(M\times \lbrace e_2 \rbrace) = \emptyset \mbox{ \ and \ } \Omega(\bar \tau) \cap \bar \tau(M\times \lbrace e_2 \rbrace) = \emptyset.
$$
Hence,
$$
\left[ \go P(\tau,\bar\tau) \right]= 4 \cdot \langle C_4(\tau,\bar\tau;e_1),C_4(\tau,\bar\tau;-e_1),C_4(\tau,\bar\tau;e_2) \rangle [S] = p_1(\tau,\bar\tau) [S].
$$
\end{proof}

\subsection[Pontrjagin numbers for combings of compact 3-manifolds]{Pontrjagin numbers for combings of compact 3-manifolds \\ Proof of Theorem~\ref{thm_defp1Xb}}
\begin{lemma} \label{cor_reformcombingsb}
Let $(X,\rho)$ be a torsion combing of a compact oriented 3-manifold $M$. Let $\bar\tau$ be a pseudo-parallelization of $M$ compatible with $X$. Let $\go P(\bar\tau,X)$ be as in Definition~\ref{def_pgo}. The class $[\go P(\bar\tau,X)]$ in $H_2([0,1]\times UM ; \b Q) / H_T^\rho(M)$ only depends on $\bar\tau$ and on the homotopy class of $X$. It will be denoted by $\tilde p_1(\bar\tau,[X])[S]$.
\end{lemma}
\begin{proof}
Let $\tau$ be another pseudo-parallelization of $M$ which is compatible with $X$. Let $C_4^+(X,\tau)$ and $C_4^-(X,\tau)$ be fixed choices of 4-chains of $[0,1]\times UM$ as in Lemma~\ref{lem_transfert}. Using these 4-chains, construct the cycle $\go P(X,\tau)$ as in Definition~\ref{def_pgo}. Then, in the space $H_2([0,1]\times UM ; \b Q) / H_T^\rho(M)$, we have
$$
\begin{aligned}
[\go P(\bar\tau,X)] + [\go P(X,\tau)] &= [\go P(\bar\tau,X) + \go P(X,\tau)] \\
&= [ \lbrace 0 \rbrace \times \Omega(\bar\tau) + 4\cdot(C_4^+(\bar\tau,X)\cap C_4^-(\bar\tau,X)) \\
& \hspace{7mm} + 4\cdot(C_4^+(X,\tau)\cap C_4^-(X,\tau)) - \lbrace 1 \rbrace \times \Omega(\tau) ].
\end{aligned}
$$
By reparameterizing and stacking $C_4^+(\bar\tau,X)$ and $C_4^+(X,\tau)$, resp. $C_4^-(\bar\tau,X)$ and $C_4^-(X,\tau)$, we get a 4-chain $C_4(\bar\tau,\tau,e_1)$, resp. $C_4(\bar\tau,\tau,-e_1)$, as in Lemma~\ref{lem_C4pme1}. It follows that
$$
\begin{aligned}
[\go P(\bar\tau,X)] + [\go P(X,\tau)] & \italicegal [ \lbrace 0 \rbrace \times \Omega(\bar\tau) + 4\cdot(C_4(\bar\tau,\tau,e_1) \cap C_4(\bar\tau,\tau,e_1)) - \lbrace 1 \rbrace \times \Omega(\tau) ] \\
&\italicegal [\go P(\bar\tau, \tau)]
\end{aligned}
$$
or, equivalently, $[\go P(\bar\tau,X)] = [\go P(\bar\tau, \tau)] - [\go P(X,\tau)]$. This proves the statement since $\go P(X,\tau)$ is independent of the choices made for $C_4^+(\bar\tau,X)$ and $C_4^-(\bar\tau,X)$, and since, according to Proposition~\ref{prop_ppara}, the class $[\go P(\bar\tau, \tau)]$ is independent of the choices for $C_4(\bar\tau,\tau,e_1)$ and $C_4(\bar\tau,\tau,-e_1)$.
\end{proof}

\begin{proposition} \label{cor_reformcombings}
If $\bar\tau$ is a pseudo-parallelization of a closed oriented 3-manifold $M$ and if $X$ is a torsion combing of $M$ compatible with $\bar\tau$, then
$$
 \tilde p_1(\bar\tau,[X]) = p_1([X])-p_1(\bar\tau).
$$
\end{proposition}
\begin{proof}
According to Lemma~\ref{cor_reformcombingsb}, $\tilde p_1(\bar\tau,[X])$ is independent of the choices for $C_4^+(\bar\tau,X)$ and $C_4^-(\bar\tau,X)$. Let us construct convenient 4-chains $C_4^+(\bar\tau,X)$ and $C_4^-(\bar\tau,X)$. Let $\tau$ be a ge\-nuine parallelization of $M$. Thanks to Theorem~\ref{prop_varasint}, there exist two 4-chains of $[0,1]\times UM$, $C_4(\bar\tau,\tau;e_1)$ and $C_4(\bar\tau,\tau;-e_1)$, such that
$$
\begin{aligned}
 \partial C_4(\bar\tau,\tau;e_1) &= \lbrace 1 \rbrace \times \tau(M\times \lbrace e_1 \rbrace) - \lbrace 0 \rbrace \times \bar \tau(M\times \lbrace e_1 \rbrace), \\
 \partial C_4(\bar\tau,\tau;-e_1) &= \lbrace 1 \rbrace \times \tau(M\times \lbrace -e_1 \rbrace) - \lbrace 0 \rbrace \times \bar \tau(M\times \lbrace -e_1 \rbrace).
\end{aligned}
$$
Furthermore, as in Remark~\ref{rmkpcase}, construct two 4-chains $C_4^+(\tau,X)$ and $C_4^-(\tau,X)$ as
$$
\begin{aligned}
 C_4^+(\tau, X) &= \bar F_{t_1}(E_1^\tau, X) - \lbrace t_1 \rbrace  \times UM_{|\Sigma_{E_1^\tau=-X}}\\
 C_4^-(\tau, X) &= \bar F_{t_2}(-E_1^\tau, -X) - \lbrace t_2 \rbrace  \times UM_{|\Sigma_{-E_1^\tau=X}}
\end{aligned}
$$
where $E_1^\tau$ stands for the first vector of the parallelization $\tau$, where $t_1, t_2\in \ ]0,1[$, and where $\Sigma_{E_1^\tau=-X}$ and $\Sigma_{-E_1^\tau=X}$ are 2-chains with boundaries $L_{E_1^\tau=-X}$ and $L_{-E_1^\tau=X}$, respectively. Eventually, define $C_4^+(\bar\tau,X)$, resp. $C_4^-(\bar\tau,X)$, by reparameterizing and stacking the chains $C_4(\bar\tau,\tau;e_1)$ and $ C_4^+(\tau, X)$, resp. $C_4(\bar\tau,\tau;-e_1)$ and $ C_4^-(\tau, X)$. \\

Let us finally compute $[\lbrace 0 \rbrace \times \Omega(\bar\tau)+4\cdot(C_4^+(\bar\tau,X) \cap C_4^-(\bar\tau,X))]$. By construction, we have :
$$
\begin{aligned}
&[\lbrace 0 \rbrace \times \Omega(\bar\tau)+4 \cdot (C_4^+(\bar\tau,X) \cap C_4^-(\bar\tau,X))] \\
&\italicegal [\lbrace 0 \rbrace \times \Omega(\bar\tau)+4 \cdot (C_4(\bar\tau,\tau;e_1) \cap C_4(\bar\tau,\tau;-e_1))] + 4 [C_4^+(\tau, X) \cap C_4^-(\tau, X)],
\end{aligned}
$$
so that, using Proposition~\ref{prop_ppara},
$$
\begin{aligned}
 &[\lbrace 0 \rbrace \times \Omega(\bar\tau)+4 \cdot (C_4^+(\bar\tau,X) \cap C_4^-(\bar\tau,X))] \\
 &= (p_1(\tau)-p_1(\bar\tau))[S] + 4 \cdot  [C_4^+(\tau, X) \cap C_4^-(\tau, X)].
\end{aligned}
$$
Now, using Definition~\ref{def_Ft},
$$
\begin{aligned}
 C_4^+(\tau, X)   = \ & \left[ 0 , t_1 \right] \times E_1^\tau(M)
 + \lbrace t_1 \rbrace \times \bar F(E_1^\tau,X) \\
 + \ &  \left[ t_1, 1 \right] \times X(M) - \lbrace t_1 \rbrace \times UM_{|\Sigma_{E_1^\tau=-X}}\\
 C_4^-(\tau, X)  = \ & \left[ 0 , t_2 \right] \times (-E_1^\tau)(M)
 + \lbrace t_2 \rbrace \times \bar F(-E_1^\tau,-X) \\
 + \ & \left[ t_2, 1 \right] \times (-X)(M) - \lbrace t_2 \rbrace \times UM_{|\Sigma_{-E_1^\tau=X}}
\end{aligned}
$$
so that, assuming $t_1 < t_2$ without loss of generality,
$$
\begin{aligned}
C_4^+(\tau, X) \cap C_4^-(\tau, X) 
= - \ & \lbrace t_1 \rbrace \times (-E_1^\tau)\left( \Sigma_{E_1^\tau=-X} \right) \\
+ \ & [t_1 , t_2] \times (-E_1^\tau)\left( L_{-E_1^\tau = X} \right) \\
- \ & \lbrace t_2 \rbrace \times X(\Sigma_{-E_1^\tau=X}).
\end{aligned}
$$
It follows that, using Theorem~\ref{thm_defp1X} and Lemma~\ref{lem_evaluationII} with $C_4(E_1^\tau,E_1^\tau)=[0,1] \times E_1^\tau(M)$,
$$
\begin{aligned}
4 \cdot [C_4^+(\tau, X) \cap C_4^-(\tau, X)] & \italicegal 4 \cdot \langle C_4^+(\tau, X) , C_4^-(\tau, X) , [0,1] \times E_1^\tau(M) \rangle [S] \\
&\italicegal4 \cdot lk(L_{E_1^\tau = X},L_{E_1^\tau = - X}) [S] \\
&\italicegal(p_1([X]) - p_1(\tau))[S] \\
\end{aligned}
$$
in $H_2([0,1]\times UM; \b Q)/H_T(M)$, and, eventually,
$$
\begin{aligned}
 [\lbrace 0 \rbrace \times \Omega(\bar\tau)&+4 \cdot (C_4^+(\bar\tau,X) \cap C_4^-(\bar\tau,X))] \\
 &= (p_1(\tau)-p_1(\bar\tau))[S] + (p_1([X]) - p_1(\tau))[S] \\
 &=(p_1([X]) - p_1(\bar\tau))[S].
\end{aligned}
$$
\end{proof}

\begin{lemma} \label{fond}
If $(X,\rho)$ is a torsion combing of a compact oriented 3-manifold $M$ and if \linebreak $\bar\tau=(N(\gamma); \tau_e,\tau_d,\tau_g)$ is a pseudo-parallelization of $M$ compatible with $X$, then
$$
\begin{aligned}
\tilde p_1(\bar\tau, [X]) &= lk_M\left(L_{E_1^d=X} + L_{E_1^g=X} \ , \ L_{E_1^d=-X} + L_{E_1^g=-X} \right)  \\
&- lk_{\b S^2}\left(e_1-(-e_1), \ P_{\b S^2} \circ \tau_d^{-1} \circ X(L_{E_1^d=-E_1^g})\right)
\end{aligned}
$$
where $E_1^d$ and $E_1^g$ denote the Siamese sections of $\bar\tau$.
\end{lemma}
\begin{proof}
We just have to evaluate the class of the 4-cycle
$$
\go P(\bar\tau,X) = \lbrace 0 \rbrace \times \Omega(\bar\tau)+4\cdot(C_4^+(\bar\tau,X)\cap C_4^-(\bar\tau,X))
$$
in $H_2([0,1]\times UM ; \b Q) / H_T^\rho(M)$ for convenient 4-chains $C_4^+(\bar\tau,X)$ and $C_4^-(\bar\tau,X)$ with the prescribed boundaries. Let $t_1$ and $t_2$ in $]0,1[$, with $t_1 > t_2$, and set 
$$
\begin{aligned}
 C_4^+(\bar\tau,X) &= \sfrac{1}{2} \cdot \left( \bar F_{t_1}(E_1^d,X)+\bar F_{t_1}(E_1^g,X) - \lbrace t_1 \rbrace \times UM_{|\Sigma(e_1)} \right)\\
 C_4^-(\bar\tau,X) &= \sfrac{1}{2} \cdot \left( \bar F_{t_2}(-E_1^d,-X)+\bar F_{t_2}(-E_1^g,-X) - \lbrace t_2 \rbrace \times UM_{|\Sigma(-e_1)} \right)
\end{aligned}
$$
where the chains $\bar F_t$ are as in Definition~\ref{def_Ft} and where, using Lemma~\ref{nulexcep}, $\Sigma(e_1)$ and $\Sigma(-e_1)$ are 2-chains of $M$ so that
$$
\partial \Sigma(\pm e_1) = \pm(L_{E_1^d=-X} + L_{E_1^g=-X}).
$$
These 4-chains do have the expected boundaries. Let us now describe $C_4^+(\bar\tau,X) \cap C_4^-(\bar\tau,X)$ :
\begin{enumerate}[\textbullet]
\item on $[0,t_2[$ : The intersection between $C_4^+(\bar\tau,X)$ and $C_4^-(\bar\tau,X)$ is
$$
\sfrac{1}{4} \cdot [0,t_2[ \ \times \ E_1^d(L_{E_1^d=-E_1^g}) + \sfrac{1}{4} \cdot  [0,t_2[ \ \times \ E_1^g(L_{E_1^g=-E_1^d}).
$$

\item on $]t_2 , t_1[$ : The intersection between $C_4^+(\bar\tau,X)$ and $C_4^-(\bar\tau,X)$ is 
$$
 \sfrac{1}{2} \ \cdot \ ]t_2,t_1[ \ \times \ (-X)(L_{E_1^d=-X}) + \sfrac{1}{2} \ \cdot \ ]t_2,t_1[ \ \times \ (-X)(L_{E_1^g=-X}).
$$

\item on $]t_1,1]$ : There is no intersection between $C_4^+(\bar\tau,X)$ and $C_4^-(\bar\tau,X)$ since they consist in $]t_1,1]\times X(M)$ and $]t_1,1]\times (-X)(M)$.

\item at $t_2$ : The intersection between $C_4^+(\bar\tau,X)$ and $C_4^-(\bar\tau,X)$ is
$$
  \begin{aligned}
  &\sfrac{1}{2} \cdot \lbrace t_2 \rbrace \times E_1^d(M) \cap \lbrace t_2 \rbrace \times  \bar F(-E_1^g,-X) \\
  + &\sfrac{1}{2} \cdot \lbrace t_2 \rbrace \times E_1^g(M) \cap \lbrace t_2 \rbrace \times  \bar F(-E_1^d,-X) \\
  - & \sfrac{1}{4} \cdot \lbrace t_2 \rbrace \times E_1^d(\Sigma(-e_1)) - \sfrac{1}{4} \cdot \lbrace t_2 \rbrace \times E_1^g(\Sigma(-e_1))
  \end{aligned}
$$
\item at $t_1$ : The intersection between $C_4^+(\bar\tau,X)$ and $C_4^-(\bar\tau,X)$ is
$$
\begin{aligned}
  &\sfrac{1}{2}\cdot \lbrace t_1 \rbrace \times  \bar F(E_1^d,X) \cap \lbrace t_1 \rbrace \times (-X)(M) \\
  + &\sfrac{1}{2}\cdot \lbrace t_1 \rbrace \times \bar F(E_1^g,X) \cap \lbrace t_1 \rbrace \times (-X)(M) \\
  - & \sfrac{1}{2} \cdot \lbrace t_1 \rbrace \times (-X)(\Sigma(e_1)).
  \end{aligned}
$$
\end{enumerate}
It follows that :
$$
\begin{aligned}
&\langle \sfrac{1}{4} \cdot \lbrace 0 \rbrace \times \Omega(\bar\tau)+ C_4^+(\bar\tau,X) \cap C_4^-(\bar\tau,X) , [0,1] \times X(M) \rangle_{[0,1]\times UM}\\
&= \sfrac{1}{4} \cdot  \langle [0,t_2[ \ \times \ E_1^d(L_{E_1^d=-E_1^g}) + [0,t_2[ \ \times \ E_1^g(L_{E_1^g=-E_1^d}) , [0,1]\times X(M) \rangle_{[0,1]\times UM} \\
&+ \sfrac{1}{2} \cdot \langle \lbrace t_2 \rbrace \times E_1^d(M) \cap \lbrace t_2 \rbrace \times  \bar F(-E_1^g,-X) , [0,1] \times X(M) \rangle_{[0,1]\times UM} \\
&+\sfrac{1}{2} \cdot \langle \lbrace t_2 \rbrace \times E_1^g(M) \cap \lbrace t_2 \rbrace \times  \bar F(-E_1^d,-X) , [0,1] \times X(M) \rangle_{[0,1]\times UM} \\
&- \sfrac{1}{4} \cdot  \langle \lbrace t_2 \rbrace \times E_1^d(\Sigma(-e_1)) + \lbrace t_2 \rbrace \times E_1^g(\Sigma(-e_1)) , [0,1] \times X(M) \rangle_{[0,1]\times UM} \\
&+ \sfrac{1}{4} \cdot \langle \Omega(\bar\tau) , X(M) \rangle_{UM}.
\end{aligned}
$$
Since $L_{E_1^g=X} \cap  L_{E_1^d=-X}$ and $L_{E_1^g=-X} \cap  L_{E_1^d=X}$ are empty :
$$
\langle [0,t_2[ \ \times \ E_1^d(L_{E_1^d=-E_1^g}) + [0,t_2[ \ \times \ E_1^g(L_{E_1^g=-E_1^d}) , [0,1]\times X(M) \rangle_{[0,1]\times UM} = 0.
$$
Furthermore, note that if $(m,v)$ is an intersection point of
$$E_1^d(M)\cap  \bar F(-E_1^g,-X) \cap X(M)
$$
then, in particular, $v=E_1^d(m)=X(m)$ so that $-E_1^g(m)$ and $-X(m)$ are not antipodal since $L_{E_1^d=X} \cap L_{E_1^g=-X}=\emptyset$. It follows that $v=E_1^d(m)=X(m)$ should also sit on the shortest geodesic arc from $-E_1^g(m)$ to $-X(m)$. Since such a configuration is impossible, this triple intersection is empty, thus
$$
\langle \lbrace t_2 \rbrace \times E_1^d(M) \cap \lbrace t_2 \rbrace \times \bar F(-E_1^g,-X) , [0,1] \times X(M) \rangle_{[0,1]\times UM}  = 0.
$$
Similarly,
$$
\langle \lbrace t_2 \rbrace \times E_1^g(M) \cap \lbrace t_2 \rbrace \times \bar F(-E_1^d,-X) , [0,1] \times X(M) \rangle_{[0,1]\times UM} = 0.
$$
Now, we have
$$
\begin{aligned}
&\langle \lbrace t_2 \rbrace \times E_1^d(\Sigma(-e_1)) + \lbrace t_2 \rbrace \times E_1^g(\Sigma(-e_1)) , [0,1] \times X(M) \rangle \\
&= - lk_M(L_{E_1^d=X} + L_{E_1^g=X}, \ L_{E_1^d=-X} + L_{E_1^g=-X}).
\end{aligned}
$$
Furthermore, recall Definition~\ref{omega}
$$
\begin{aligned}
\langle \Omega(\bar\tau) , X(M) \rangle_{UM} &= \langle \tau_d^{-1} (\Omega(\bar\tau)) , \tau_d^{-1} \circ X(L_{E_1^d=-E_1^g}) \rangle_{L_{E_1^d=-E_1^g} \times \b S^2} \\
&= - \langle L_{E_1^d=-E_1^g} \hspace{-0.5mm} \times \hspace{-0.5mm} [-e_1,e_1]_{e_2^\Omega} \ , \ \tau_d^{-1} \circ X(L_{E_1^d=-E_1^g}) \rangle_{L_{E_1^d=-E_1^g} \times \b S^2}
\end{aligned}
$$
where $[-e_1,e_1]_{e_2^\Omega}$ is the geodesic arc of $\b S^2$ from $-e_1$ to $e_1$ passing through $e_2^\Omega$. Now, $L_{E_1^d=-E_1^g} \times \b S^2$ is oriented and an intersection 
$$
(m,v)\hspace{-0.5mm}\in\hspace{-0.5mm}L_{E_1^d=-E_1^g} \hspace{-0.5mm} \times \hspace{-0.5mm} [-e_1,e_1]_{e_2^\Omega} \cap \tau_d^{-1} \circ X(L_{E_1^d=-E_1^g})
$$ 
is positive when 
$$
T_{(m,v)}(L_{E_1^d=-E_1^g} \times [-e_1,e_1]_{e_2^\Omega})\oplus T_{(m,v)}(\tau_d^{-1}\circ X(L_{E_1^d=-E_1^g}))=T_{(m,v)}(L_{E_1^d=-E_1^g}\times\b S^2)
$$
as an oriented sum, which is equivalent to 
$$
T_{v}([-e_1,e_1]_{e_2^\Omega})\oplus T_{v}(P_{\b S^2}\circ\tau_d^{-1}\circ X(L_{E_1^d=-E_1^g}))=T_{v}(\b S^2)
$$
as an oriented sum, where $P_{\b S^2}$ is the standard projection from $M\times \b S^2$ to $\b S^2$. See Figure \ref{orientation}.
\begin{center}
\definecolor{ccqqtt}{rgb}{0.8,0,0.2}
\definecolor{zzttqq}{rgb}{0.6,0.2,0}
\definecolor{ccqqqq}{rgb}{0.8,0,0}
\begin{tikzpicture}[line cap=round,line join=round,>=triangle 45,x=1.0cm,y=0.5cm,scale=0.9]
\clip(9,-2.5) rectangle (23,10);
\fill[line width=0pt,color=zzttqq,fill=zzttqq,fill opacity=0.1] (13,4) -- (14.5,4) -- (14.5,-1) -- (13,-1) -- cycle;
\fill[line width=0pt,dotted,color=zzttqq,fill=zzttqq,fill opacity=0.1] (16,4) -- (20,4) -- (20,-1) -- (16,-1) -- cycle;
\fill[line width=0pt,color=zzttqq,fill=zzttqq,fill opacity=0.1] (21.5,4) -- (23,4) -- (23,-1) -- (21.5,-1) -- cycle;
\draw [line width=1.4pt] (16,4)-- (20,4);
\draw [line width=1.4pt] (20,-1)-- (20,8);
\draw [line width=1.4pt] (20,8)-- (22,9);
\draw [->] (16,4) -- (20,4);
\draw [->,line width=1.4pt] (21.5,4) -- (23,4);
\draw [line width=1.4pt] (14.5,-1)-- (14.5,8);
\draw [line width=1.4pt] (14.5,8)-- (16.5,9);
\draw [->,line width=1.4pt,color=ccqqqq] (21.5,6) -- (23,6);
\draw [->,line width=1.4pt,color=ccqqqq] (13,6) -- (14.5,6);
\draw [->,line width=1.4pt] (13,4) -- (14.5,4);
\draw [->] (16,-1) -- (16,4);
\draw [->] (21.5,-1) -- (21.5,4);
\draw[line width=1.4pt,color=ccqqqq] (19.12,4.3) -- (19.11,4.24) -- (19.1,4.18) -- (19.1,4.12) -- (19.09,4.05) -- (19.08,3.99) -- (19.07,3.93) -- (19.07,3.87) -- (19.06,3.81) -- (19.05,3.74) -- (19.05,3.68) -- (19.04,3.62) -- (19.04,3.56) -- (19.03,3.49) -- (19.02,3.43) -- (19.02,3.37) -- (19.01,3.31) -- (19.01,3.25) -- (19,3.18) -- (18.99,3.12) -- (18.99,3.06) -- (18.98,3) -- (18.98,2.94) -- (18.97,2.88) -- (18.96,2.82) -- (18.96,2.76) -- (18.95,2.7) -- (18.95,2.65) -- (18.94,2.59) -- (18.93,2.53) -- (18.93,2.48) -- (18.92,2.42) -- (18.91,2.37) -- (18.9,2.31) -- (18.9,2.26) -- (18.89,2.21) -- (18.88,2.16) -- (18.87,2.11) -- (18.86,2.06) -- (18.85,2.01) -- (18.85,1.97) -- (18.84,1.92) -- (18.83,1.88) -- (18.82,1.83) -- (18.81,1.79) -- (18.79,1.75) -- (18.78,1.71) -- (18.77,1.67) -- (18.76,1.64) -- (18.75,1.6) -- (18.73,1.57) -- (18.72,1.54) -- (18.71,1.51) -- (18.69,1.48) -- (18.68,1.45) -- (18.66,1.42) -- (18.65,1.4) -- (18.63,1.38) -- (18.61,1.36) -- (18.59,1.34) -- (18.58,1.32) -- (18.56,1.31) -- (18.54,1.29) -- (18.52,1.28) -- (18.5,1.27) -- (18.48,1.27) -- (18.45,1.26) -- (18.43,1.26) -- (18.41,1.26) -- (18.38,1.26) -- (18.36,1.27) -- (18.33,1.27) -- (18.31,1.28) -- (18.28,1.29) -- (18.25,1.3) -- (18.22,1.32) -- (18.19,1.34) -- (18.16,1.36) -- (18.13,1.38) -- (18.1,1.41) -- (18.07,1.44) -- (18.03,1.47) -- (18.02,1.48) -- (18.01,1.49) -- (18,1.5) -- (18,1.5) -- (18,1.5) -- (18,1.5) -- (18,1.5) -- (18,1.5) -- (18,1.5) -- (18,1.5) -- (18,1.5)(20,6) -- (20,6) -- (20,6) -- (20,6) -- (20,6) -- (20,6) -- (20,6) -- (20,6) -- (20,6) -- (20,6) -- (20,6) -- (20,6) -- (20,6) -- (20,6) -- (20,6) -- (19.99,6) -- (19.99,6) -- (19.98,6) -- (19.95,6) -- (19.92,5.99) -- (19.89,5.99) -- (19.86,5.98) -- (19.82,5.97) -- (19.79,5.96) -- (19.77,5.94) -- (19.74,5.93) -- (19.71,5.91) -- (19.68,5.89) -- (19.66,5.87) -- (19.63,5.84) -- (19.61,5.82) -- (19.58,5.79) -- (19.56,5.76) -- (19.54,5.73) -- (19.52,5.7) -- (19.49,5.67) -- (19.47,5.64) -- (19.45,5.6) -- (19.44,5.56) -- (19.42,5.52) -- (19.4,5.48) -- (19.38,5.44) -- (19.36,5.4) -- (19.35,5.36) -- (19.33,5.31) -- (19.32,5.26) -- (19.3,5.22) -- (19.29,5.17) -- (19.27,5.12) -- (19.26,5.07) -- (19.25,5.02) -- (19.24,4.97) -- (19.22,4.91) -- (19.21,4.86) -- (19.2,4.8) -- (19.19,4.75) -- (19.18,4.69) -- (19.17,4.63) -- (19.16,4.58) -- (19.15,4.52) -- (19.14,4.46) -- (19.13,4.4) -- (19.12,4.34) -- (19.12,4.3);
\draw[line width=1.4pt,dash pattern=on 2pt off 4pt,color=ccqqtt] (17.44,2.52) -- (17.44,2.54) -- (17.43,2.56) -- (17.42,2.58) -- (17.41,2.6) -- (17.4,2.63) -- (17.4,2.65) -- (17.39,2.67) -- (17.38,2.69) -- (17.37,2.72) -- (17.37,2.74) -- (17.36,2.76) -- (17.35,2.78) -- (17.34,2.81) -- (17.33,2.83) -- (17.33,2.85) -- (17.32,2.88) -- (17.31,2.9) -- (17.3,2.92) -- (17.3,2.95) -- (17.29,2.97) -- (17.28,3) -- (17.27,3.02) -- (17.26,3.05) -- (17.26,3.07) -- (17.25,3.1) -- (17.24,3.12) -- (17.23,3.15) -- (17.23,3.18) -- (17.22,3.2) -- (17.21,3.23) -- (17.2,3.25) -- (17.19,3.28) -- (17.19,3.31) -- (17.18,3.33) -- (17.17,3.36) -- (17.16,3.39) -- (17.15,3.42) -- (17.15,3.44) -- (17.14,3.47) -- (17.13,3.5) -- (17.12,3.53) -- (17.12,3.56) -- (17.11,3.59) -- (17.1,3.61) -- (17.09,3.64) -- (17.08,3.67) -- (17.08,3.7) -- (17.07,3.73) -- (17.06,3.76) -- (17.05,3.79) -- (17.05,3.82) -- (17.04,3.85) -- (17.03,3.88) -- (17.02,3.91) -- (17.01,3.94) -- (17.01,3.97) -- (17,3.99) -- (17,4) -- (17,4) -- (17,4) -- (17,4) -- (17,4) -- (17,4) -- (17,4) -- (17,4) -- (17,4) -- (17,4)(17.49,2.04)(18,1.5) -- (18,1.5) -- (18,1.5) -- (18,1.5) -- (18,1.5) -- (18,1.5) -- (18,1.5) -- (18,1.5) -- (18,1.5) -- (18,1.5) -- (18,1.5) -- (18,1.5) -- (18,1.5) -- (18,1.5) -- (18,1.5) -- (18,1.5) -- (18,1.5) -- (18,1.5) -- (17.99,1.51) -- (17.99,1.51) -- (17.98,1.52) -- (17.97,1.53) -- (17.97,1.54) -- (17.96,1.54) -- (17.95,1.55) -- (17.94,1.56) -- (17.93,1.57) -- (17.93,1.58) -- (17.92,1.59) -- (17.91,1.6) -- (17.9,1.61) -- (17.9,1.62) -- (17.89,1.63) -- (17.88,1.64) -- (17.87,1.65) -- (17.86,1.66) -- (17.86,1.67) -- (17.85,1.69) -- (17.84,1.7) -- (17.83,1.71) -- (17.83,1.72) -- (17.82,1.73) -- (17.81,1.74) -- (17.8,1.76) -- (17.79,1.77) -- (17.79,1.78) -- (17.78,1.79) -- (17.77,1.81) -- (17.76,1.82) -- (17.76,1.83) -- (17.75,1.85) -- (17.74,1.86) -- (17.73,1.88) -- (17.72,1.89) -- (17.72,1.9) -- (17.71,1.92) -- (17.7,1.93) -- (17.69,1.95) -- (17.68,1.96) -- (17.68,1.98) -- (17.67,1.99) -- (17.66,2.01) -- (17.65,2.03) -- (17.65,2.04) -- (17.64,2.06) -- (17.63,2.08) -- (17.62,2.09) -- (17.61,2.11) -- (17.61,2.13) -- (17.6,2.14) -- (17.59,2.16) -- (17.58,2.18) -- (17.58,2.19) -- (17.57,2.21) -- (17.56,2.23) -- (17.55,2.25) -- (17.54,2.27) -- (17.54,2.29) -- (17.53,2.3) -- (17.52,2.32) -- (17.51,2.34) -- (17.51,2.36) -- (17.5,2.38) -- (17.49,2.4) -- (17.48,2.42) -- (17.47,2.44) -- (17.47,2.46) -- (17.46,2.48) -- (17.45,2.5) -- (17.44,2.52) -- (17.44,2.52);
\draw[line width=1.4pt,color=ccqqtt] (16.44,5.61) -- (16.44,5.62) -- (16.43,5.63) -- (16.42,5.65) -- (16.41,5.66) -- (16.4,5.67) -- (16.4,5.68) -- (16.39,5.7) -- (16.38,5.71) -- (16.37,5.72) -- (16.37,5.73) -- (16.36,5.74) -- (16.35,5.75) -- (16.34,5.77) -- (16.33,5.78) -- (16.33,5.79) -- (16.32,5.8) -- (16.31,5.81) -- (16.3,5.82) -- (16.3,5.83) -- (16.29,5.83) -- (16.28,5.84) -- (16.27,5.85) -- (16.26,5.86) -- (16.26,5.87) -- (16.25,5.88) -- (16.24,5.88) -- (16.23,5.89) -- (16.23,5.9) -- (16.22,5.91) -- (16.21,5.91) -- (16.2,5.92) -- (16.19,5.92) -- (16.19,5.93) -- (16.18,5.94) -- (16.17,5.94) -- (16.16,5.95) -- (16.15,5.95) -- (16.15,5.96) -- (16.14,5.96) -- (16.13,5.97) -- (16.12,5.97) -- (16.12,5.97) -- (16.11,5.98) -- (16.1,5.98) -- (16.09,5.98) -- (16.08,5.99) -- (16.08,5.99) -- (16.07,5.99) -- (16.06,5.99) -- (16.05,5.99) -- (16.05,6) -- (16.04,6) -- (16.03,6) -- (16.02,6) -- (16.01,6) -- (16.01,6) -- (16,6) -- (16,6) -- (16,6) -- (16,6) -- (16,6) -- (16,6) -- (16,6) -- (16,6) -- (16,6) -- (16,6) -- (16,6)(17,4) -- (17,4) -- (17,4) -- (17,4) -- (17,4) -- (17,4) -- (17,4) -- (17,4) -- (17,4) -- (17,4) -- (17,4) -- (17,4) -- (17,4) -- (17,4) -- (17,4) -- (17,4) -- (17,4) -- (17,4.01) -- (17,4.01) -- (16.99,4.02) -- (16.99,4.04) -- (16.98,4.07) -- (16.97,4.1) -- (16.97,4.13) -- (16.96,4.16) -- (16.95,4.19) -- (16.94,4.22) -- (16.93,4.25) -- (16.93,4.28) -- (16.92,4.31) -- (16.91,4.34) -- (16.9,4.37) -- (16.9,4.4) -- (16.89,4.42) -- (16.88,4.45) -- (16.87,4.48) -- (16.86,4.51) -- (16.86,4.53) -- (16.85,4.56) -- (16.84,4.59) -- (16.83,4.61) -- (16.83,4.64) -- (16.82,4.66) -- (16.81,4.69) -- (16.8,4.71) -- (16.79,4.74) -- (16.79,4.76) -- (16.78,4.79) -- (16.77,4.81) -- (16.76,4.84) -- (16.76,4.86) -- (16.75,4.88) -- (16.74,4.91) -- (16.73,4.93) -- (16.72,4.95) -- (16.72,4.97) -- (16.71,5) -- (16.7,5.02) -- (16.69,5.04) -- (16.68,5.06) -- (16.68,5.08) -- (16.67,5.1) -- (16.66,5.13) -- (16.65,5.15) -- (16.65,5.17) -- (16.64,5.19) -- (16.63,5.21) -- (16.62,5.23) -- (16.61,5.24) -- (16.61,5.26) -- (16.6,5.28) -- (16.59,5.3) -- (16.58,5.32) -- (16.58,5.34) -- (16.57,5.36) -- (16.56,5.37) -- (16.55,5.39) -- (16.54,5.41) -- (16.54,5.42) -- (16.53,5.44) -- (16.52,5.46) -- (16.51,5.47) -- (16.51,5.49) -- (16.5,5.51) -- (16.49,5.52) -- (16.48,5.54) -- (16.47,5.55) -- (16.47,5.57) -- (16.46,5.58) -- (16.45,5.59) -- (16.44,5.61);
\draw (15,10) node[anchor=north west] {$\b S^2$};
\draw (20.5,10) node[anchor=north west] {$\b S^2$};
\draw (9.25,6.8) node[anchor=north west] {$\tau_d^{-1}\circ X(L_{E_1^d = -E_1^g})$};
\draw (14.8,-1.34) node[anchor=north west] {$(-e_1)$};
\draw (20.3,-1.34) node[anchor=north west] {$(-e_1)$};
\draw (15.3,4.4) node[anchor=north west] {$e_1$};
\draw (20.8,4.4) node[anchor=north west] {$e_1$};
\draw (11,4.64) node[anchor=north west] {$L_{E_1^d=-E_1^g}$};
\draw [color=black, line width=1.4pt] (18,1.5)-- ++(-4.5pt,0 pt) -- ++(9.0pt,0 pt) ++(-4.5pt,-4.5pt) -- ++(0 pt,9.0pt);
\draw[color=black] (17.5,0.5) node {$(m,v)$};
\end{tikzpicture}

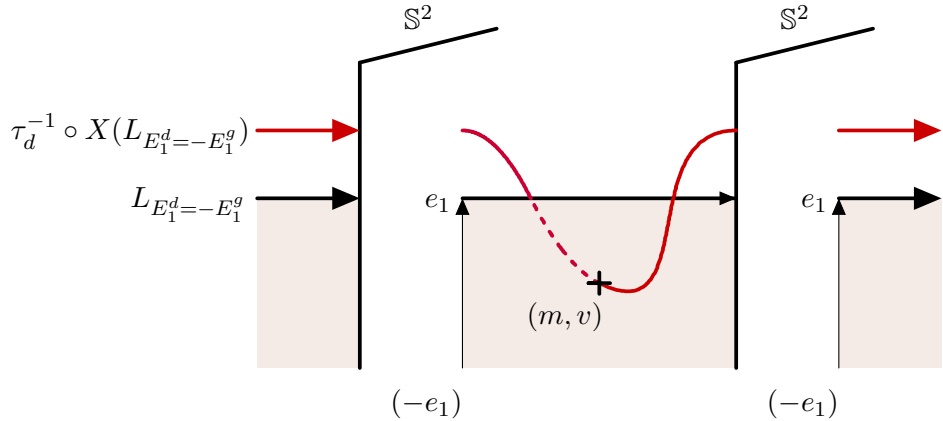
\captionof{figure}{A positive intersection $(m,v)$ between the 2-chain $L_{E_1^d=-E_1^g}\times[-e_1,e_1]_{e_2^\Omega}$ and $\tau_d^{-1} \circ X (L_{E_1^d=-E_1^g})$ in $L_{E_1^d=-E_1^g}\times \b S^2$.\\}\label{orientation}
\end{center}

\noindent
It follows that
$$
\begin{aligned}
\langle \Omega(\bar\tau) , X(M) \rangle_{UM} &= - \langle [-e_1,e_1]_{e_2^\Omega}, \ P_{\b S^2} \circ \tau_d^{-1} \circ X(L_{E_1^d=-E_1^g}) \rangle_{\b S^2} \\
&= - lk_{\b S^2}\left(e_1-(-e_1), \ P_{\b S^2} \circ \tau_d^{-1} \circ X(L_{E_1^d=-E_1^g})\right),
\end{aligned}
$$
\iffalse
and, finally,
$$
\begin{aligned}
\langle \go P(\bar\tau,X), [0,1]\times X(M) \rangle_{[0,1]\times UM} &= lk_M\left(L_{E_1^d=X} + L_{E_1^g=X} \ , \ L_{E_1^d=-X} + L_{E_1^g=-X} \right) \\ 
&- lk_{\b S^2}\left(e_1-(-e_1), \ P_{\b S^2} \circ \tau_d^{-1} \circ X(L_{E_1^d=-E_1^g})\right).
\end{aligned}
$$
\fi
\end{proof}

\begin{proof}[Proof of Lemma~\ref{lem1}]
According to Lemmas~\ref{cor_reformcombingsb}~and~\ref{fond}, Lemma~\ref{lem1} is true for $p_1=\tilde p_1$.
\end{proof}

From now on, if $X$ is a torsion combing of a compact oriented 3-manifold $M$ and if $\bar\tau$ is a pseudo-parallelization of $M$ compatible with $X$, then set 
$$
p_1([X],\bar\tau) = \tilde p_1([X],\bar\tau)  \mbox{ \ and \ } p_1(\bar\tau,[X]) = \tilde p_1(\bar\tau,[X]).
$$
As an obvious consequence, we get the following lemma.
\begin{lemma} \label{lempluplus}
Under the assumptions of Lemma \ref{cor_reformcombingsb}, in $H_2([0,1]\times UM ; \b Q)/H_T^\rho(M)$, the class of $\go P(\bar\tau,X)$ is $p_1(\bar\tau, X)[S]$.
\end{lemma}

\begin{proof}[Proof of Theorem~\ref{thm_defp1Xb}]
Let $X_1$ and $X_2$ be torsion combings of two compact oriented 3-manifolds $M_1$ and $M_2$ with identified boundaries such that $X_1$ and $X_2$ coincide on the boundary. Let also $\bar\tau_1$ and $\bar\tau'_1$ be two pseudo-parallelizations of $M_1$ that extend the trivialization $\rho(X_1)$ and, similarly, let $\bar\tau_2$ be a pseudo-parallelization of $M_2$ that extends the trivialization $\rho(X_2)$. In such a context, let
$$
\begin{aligned}
p_1([X_1],[X_2])(\bar\tau_1,\bar\tau_2) &= p_1([X_1],\bar\tau_1) + p_1(\bar\tau_1,\bar\tau_2) + p_1(\bar\tau_2, [X_2]) \\
p_1([X_1],[X_2])(\bar\tau'_1,\bar\tau_2) &= p_1([X_1],\bar\tau'_1) + p_1(\bar\tau'_1,\bar\tau_2) + p_1(\bar\tau_2, [X_2]) \\
\end{aligned}
$$
and note that
$$
p_1([X_1],[X_2])(\bar\tau_1,\bar\tau_2)  - p_1([X_1],[X_2])(\bar\tau'_1,\bar\tau_2) = p_1(\bar\tau_1,\bar\tau'_1) - \langle \go P(\bar\tau_1,\bar\tau'_1) , [0,1]\times X_1(M) \rangle.
$$
Using Proposition~\ref{prop_ppara}, we get $p_1([X_1],[X_2])(\bar\tau_1,\bar\tau_2)  - p_1([X_1],[X_2])(\bar\tau'_1,\bar\tau_2)=0$. In other words $p_1([X_1],[X_2])(\bar\tau_1,\bar\tau_2)$ is independent of $\bar\tau_1$. Similarly, it is also independent of $\bar\tau_2$ so that we can drop the pseudo-parallelizations from the notation. Eventually, using Lemma~\ref{lem1}, we get the formula of the statement. \\

For the second part of the statement, if $M_1$ and $M_2$ are closed, conclude with Proposition~\ref{cor_reformcombings}, which ensures that $p_1(\bar\tau_i,[X_i]) = p_1([X_i])-p_1(\bar\tau_i)$ for $i\in \lbrace 1,2 \rbrace$.
\end{proof}

Let us now end this section by proving Theorem \ref{formuleplus} and Theorem \ref{GM}, starting with the following.

\begin{lemma} \label{lemmaplus1}
 Let $(X,\rho)$ and $(Y,\rho)$ be $\partial$-compatible torsion combings of a compact oriented 3-manifold $M$. Let $C_4(X,Y)$ and $C_4(-X,-Y)$ be 4-chains of $[0,1]\times UM$ as in Lemma \ref{lem_evaluationII}. The class of $\go P(X,Y) \italicegal 4\big(C_4(X,Y) \cap C_4(-X,-Y)\big)$ in the space $H_2([0,1]\times UM ; \b Q)/H_T^\rho(M)$ reads $p_1([X],[Y]) [S]$ where $[S]$ is the homology class of the fiber of $UM$ in $H_2([0,1]\times UM ; \b Q)$.
\end{lemma}
\begin{proof}
Let $\bar \tau$ be a pseudo-parallelization of $M$ compatible with $X$ and $Y$. By Theorem \ref{thm_defp1Xb}, in $H_2([0,1]\times UM ; \b Q)/H_T^\rho(M)$ :
$$
 p_1([X],[Y])[S] = \big(p_1([X],\bar \tau) + p_1(\bar \tau,[Y]) \big)[S] 
$$
Then, using Lemma \ref{lempluplus},
$$
\begin{aligned}
 p_1([X],[Y])[S] &= [\go P(X,\bar \tau) + \go P(\bar \tau, Y)]\\
 &= [4(C_4^+(X,\bar \tau)\cap C_4^-(X,\bar \tau))  + 4(C_4^+(\bar \tau, Y)\cap C_4^-(\bar \tau, Y)) ].
\end{aligned}
$$
Hence, reparameterizing and stacking $C_4^+(X,\bar\tau)$ and $C_4^+(\bar\tau,Y)$, resp. $C_4^-(X,\bar\tau)$ and $C_4^-(\bar\tau,Y)$, we get a 4-chain $C_4(X,Y)$, resp. $C_4(-X,-Y)$, as in Lemma \ref{lem_evaluationII} and
$$
 p_1([X],[Y])[S]= 4 \cdot [C_4(X,Y)\cap C_4(-X,-Y)].
$$

To conclude the proof, see that if $C'_4(X,Y)$ is a 4-chain of $[0,1]\times UM$ with the same boundary as $C_4(X,Y)$, then $C'_4(X,Y)-C_4(X,Y)$ is homologous to a 4-cycle in $\lbrace 0 \rbrace \times UM$ so that 
$$
\begin{aligned}
&[C'_4(X,Y)  \cap C_4(-X,-Y)] - [C_4(X,Y) \cap C_4(-X,-Y)] \\
&= \left[\big(C'_4(X,Y)- C_4(X,Y) \big) \cap C_4(-X,-Y)\right]
\end{aligned}
$$ 
sits in $H_T^\rho(M)$. So, the class $[C_4(X,Y)\cap C_4(-X,-Y)]$ in $H_2([0,1]\times UM ; \b Q)/H_T^\rho(M)$ is independent of the choices for $C_4(X,Y)$. Similarly, it is independent of the choices for $C_4(-X,-Y)$.
\end{proof}

\begin{proof}[Proof of Theorem \ref{formuleplus}]
According to Lemma \ref{lemmaplus1}, it is enough to evaluate the class of the chain $4\big(C_4(X,Y) \cap C_4(-X,-Y)\big)$ in $H_2([0,1]\times UM ; \b Q)/H_T^\rho(M)$ where $\rho = \rho(X)$ and where $C_4(X,Y)$ and $C_4(-X,-Y)$ are 4-chains of $[0,1]\times UM$ as in Lemma~\ref{lem_evaluationII}. Let us consider the 4-chains
$$
\begin{aligned}
C_4(X,Y) &= \bar F_{t_1}(X,Y) - \lbrace t_1 \rbrace \times UM_{|\Sigma_{X=-Y}},\\
C_4(-X,-Y) &= \bar F_{t_2}(-X,-Y) - \lbrace t_2 \rbrace \times UM_{|\Sigma_{-X=Y}} ,
\end{aligned}
$$
where $0<t_1<t_2<1$, and where $\bar F_{t_1}(X,Y)$, resp. $\bar F_{t_2}(-X,-Y)$, is a 4-chain as in Definition~\ref{def_Ft} and $\Sigma_{X=-Y}$, resp. $\Sigma_{-X=Y}$, is a 2-chain of $M$ bounded by $L_{X=-Y}$, resp. $L_{-X=Y}$, provided by Proposition~\ref{prop_linksinhomologyI}. With these chains,
$$
\begin{aligned}
&C_4(X,Y) \cap C_4(-X,-Y) \\ 
&= -\lbrace t_1 \rbrace \times (-X)(\Sigma_{X=-Y})  + [t_1, t_2] \times (-X)(L_{-X=Y}) - \lbrace t_2 \rbrace \times Y(\Sigma_{-X=Y}).
\end{aligned}
$$
Hence, using Lemma \ref{lem_evaluationII} with $[0,1]\times X(M)$, in $H_2([0,1]\times UM ; \b Q)/H_T^\rho(M)$ :
$$
\begin{aligned}
 [C_4(X,Y)  \cap C_4(-X,-Y)] &= \langle C_4(X,Y)\cap C_4(-X,-Y) , [0,1] \times X(M) \rangle_{[0,1]\times UM}  [S] \\
 &= lk(L_{X=Y},L_{X=-Y}) [S].
\end{aligned}
$$
\iffalse
 &=  \langle - \lbrace t_1 \rbrace \times (-X)(\Sigma_{X=-Y}) + [t_1, t_2] \times (-X)(L_{-X=Y})  \\
 & \hspace{6cm} - \lbrace t_2 \rbrace \times Y(\Sigma_{-X=Y}) , [0,1]\times X(M) \rangle_{[0,1]\times UM}  [S] \\
 &= - lk(L_{X=Y},L_{-X=Y}) [S] \\
\fi
\end{proof}

\iffalse
\begin{proof}
Let $(X,\sigma)$ and $(Y, \sigma)$ be $\partial$-compatible torsion combings of a compact oriented 3-manifold $M$ that represent the same $\Spinc$-structure. By definition, $(X,\sigma)$ and $(Y, \sigma)$ are homotopic on $M\setminus \l B$ where $\l B$ is a 3-ball in $\mathring M$. \linebreak So, there exists a combing $Y'$ homotopic to $Y$ (on $M$) such that $L_{X=Y'} \cap \partial \l B = \emptyset$ and $L_{X=-Y'} \subset \l B$. Let $L^{\l B} = L_{X=Y'} \cap \l B$ and $L^{M\setminus \l B} = L_{X=Y'} \cap (M\setminus \l B)$. The link $L_{X=-Y'}$ bounds a compact oriented surface in $\l B$, hence
$$
\begin{aligned}
p_1([X], [Y]) = p_1([X], [Y']) &= 4\cdot lk(L_{X=Y'}, L_{X=-Y'})\\
&=  4\cdot lk(L^{M \setminus \l B} \sqcup L^{\l B}, L_{X=-Y'})\\
&= 4\cdot lk(L^{\l B}, L_{X=-Y'}).
\end{aligned}
$$

Finally, as in the closed case (see \cite[Subsection 2.3 and Corollary 2.22]{lescopcombing}), $X$ can be extended as a parallelization on $\l B$ so that $Y'$ induces a map from $\l B$ to $\b S^2$. Its class in $\pi_3(\b S^2)\simeq \b Z$ reads $lk(L^{\l B}, L_{X=-Y'})$. Hence, $p_1([X], [Y]) = 0$ if and only if $X$ and $Y$ are homotopic on $M$.
\end{proof}
\fi

\begin{proof}[Proof of Theorem \ref{GM}]
If $X$ and $Y$ are homotopic relatively to the boundary, then $p_1([X],[Y])=0$. \linebreak Conversely, consider two combings $X$ and $Y_0$ in the same $\Spinc$-structure and assume that \linebreak $p_1([X],[Y_0])=0$. Since $Y_0$ is in the same $\Spinc$-structure as $X$, there exists a homotopy from $Y_0$ to a combing $Y_1$ that coincides with $X$ outside a ball $\l B$ in $\mathring M$. \\

Let $\sigma$ be a unit section of $X^{\perp}_{|\l B}$, and let $(X, \sigma, X\wedge\sigma)$ denote the corresponding parallelization over $\l B$. Extend the unit section $\sigma$ as a generic section of $X^{\perp}$ such that $\sigma_{|\partial M}=\sigma(X)$, and deform $Y_1$ to $Y$ where 
$$
Y(m)=\frac{Y_1(m) + \chi(m) \sigma(m)}{\parallel Y_1(m) + \chi(m)\sigma(m)\parallel}
$$
for a smooth map $\chi$ from $M$ to $[0,\varepsilon]$, such that $\chi^{-1}(0)=\partial M$ and $\chi$ maps the complement of a neighborhood of $\partial M$ to $\varepsilon$, where $\varepsilon$ is a small positive real number. The link $L_{X=Y}$ is the disjoint union of $L_{X=Y}\cap \l B$ and a link $L_2$ of $M \setminus \l B$, the link $L_{X=-Y}$ sits in $\l B$, and
$$
0=p_1(X,Y_0)=p_1(X,Y)= 4lk(L_{X=Y},L_{X=-Y})
$$
where $lk(L_{X=Y},L_{X=-Y})=lk(L_{X=Y}\cap \l B,L_{X=-Y})=0$.\\

The parallelization $(X, \sigma, X\wedge\sigma)$ turns the restriction $Y_{|\l B}$ into a map from the ball $\l B$ to $\b S^2 = \b S(\b R X \oplus \b R \sigma \oplus \b R X\wedge\sigma)$ constant on $\partial \l B$, thus into a map from $\l B / \partial \l B \simeq \b S^3$ to $\b S^2$, and it suffices to prove that this map is homotopic to the constant map to prove Theorem \ref{GM}. For this it suffices to prove that this map represents $0$ in $\pi_3(\b S^2) \simeq \b Z$. \\

There is a classical isomorphism from $\pi_3(\b S^2)$ to $\b Z$ that maps the class of a map $g$ from $\b S^3$ to $\b S^2$ to the linking number of the preimages of two regular points of $g$ under $g$ (see \cite{hopf} and \cite[Theorem 2]{pontrjagin}). It is easy to check that this map is well-defined, depends only on the homotopy class of $g$, and is a group morphism on $\pi_3(\b S^2)$ that maps the class of the Hopf fibration $\left((z_1,z_2) \in (\b S^3 \subset \b C^2) \mapsto (\sfrac{z_1}{z_2}) \in (\b C P^1=\b S^2) \right)$ to $\pm 1$. Therefore it is an isomorphism from $\pi_3(\b S^2)$ to $\b Z$. Since $Y$ is in the kernel of this isomorphism, it is homotopically trivial so that $Y$ is homotopic to a constant on $B$, relatively to the boundary of $B$, and $Y_0$ is homotopic to $X$ on $M$, relatively to the boundary of $M$. 
\end{proof}

\section[Variation of Pontrjagin numbers under LP$_\b Q$-surgeries]{Variation of Pontrjagin numbers \\ under LP$_\b Q$-surgeries}
\subsection{For pseudo-parallelizations}
In this subsection we recall the variation formula and the finite type property of Pontrjagin numbers of pseudo-parallelizations, which are contained in \cite[Section 11]{lescopcube}.
\def\taubarre{\mbox{$\overline{\tau}$}}
\begin{proposition} \label{prop_FTIppara}
\noindent
For $M$ a compact oriented 3-manifold and $(\sfrac{B}{A})$ an LP$_\b Q$-surgery in $M$, if ${\bar\tau}_{M}$ and ${\bar\tau}_{M(\sfrac{B}{A})}$ are pseudo-parallelizations of $M$ and $M(\sfrac{B}{A})$ which coincide on $M\setminus \mathring A$ and coincide with a genuine parallelization on $\partial A$, then
$$
p_1({\taubarre}_{M(\sfrac{B}{A})},{\taubarre}_{M}) = p_1({{\taubarre}_{M(\sfrac{B}{A})}}_{|B},{{\taubarre}_{M}}_{|A}).
$$
\end{proposition}
\begin{proof}
Let $W^-$ be a signature zero cobordism from $A$ to $B$. By definition, the obstruction $p_1({{\taubarre}_{M(\sfrac{B}{A})}}_{|B},{{\taubarre}_{M}}_{|A})$ is the Pontrjagin obstruction to extending the complex trivialization \linebreak $\tau({{\taubarre}_{M(\sfrac{B}{A})}}_{|B},{{\taubarre}_{M}}_{|A})$ of $TW^-_{\hspace{-1mm}|\partial W^-} \hspace{-1mm} \otimes \b C$ as a trivialization of $TW^- \hspace{-1mm}\otimes \b C$. Let $W^+ \hspace{-1.5mm}=\hspace{-1mm} [0,1]\times (M\setminus \mathring A)$ and let $V=-[0,1]\times\partial A$. As shown in \cite[Proof of Proposition 5.3 item 2]{lescopEFTI}, since $(\sfrac{B}{A})$ is an LP$_\b Q$-surgery in $M$, the manifold $W=W^+\cup_V W^-$ has signature zero. Furthermore, since $W$ has signature zero, $p_1({\taubarre}_{M(\sfrac{B}{A})},{\taubarre}_{M})$ is the Pontrjagin obstruction to extending the triviali\-zation $\tau({\taubarre}_{M(\sfrac{B}{A})},{\taubarre}_{M})$ of $TW_{|\partial W} \otimes \b C$ as a trivialization of $TW\otimes \b C$. Finally, it is clear that $\tau({\taubarre}_{M(\sfrac{B}{A})},{\taubarre}_{M})$ extends as a trivialization of $TW_{|W^+} \otimes \b C$ so that
$$
p_1({\taubarre}_{M(\sfrac{B}{A})},{\taubarre}_{M}) = p_1({{\taubarre}_{M(\sfrac{B}{A})}}_{|B},{{\taubarre}_{M}}_{|A}).
$$

\end{proof}

\begin{corollary} \label{cor_FTppara}
Let $M$ be a compact oriented 3-manifold and let $\lbrace \sfrac{B_i}{A_i} \rbrace_{i\in \lbrace 1, \ldots , k\rbrace}$ be a family of disjoint LP$_\b Q$-surgeries where $k\geqslant2$. For any family $\lbrace \bar \tau_I \rbrace_{I \subset \lbrace 1, \ldots , k \rbrace}$ of pseudo-parallelizations of the $\lbrace M(\lbrace \sfrac{B_i}{A_i}\rbrace_{i\in I}) \rbrace_{I \subset \lbrace 1, \ldots , k \rbrace}$ whose links sit in $M\setminus(\cup_{i=1}^k \partial A_i)$ and such that, for all subsets $I, J\subset \lbrace 1, \ldots , k \rbrace$, $\bar \tau^I$ and $\bar \tau^J$ coincide on $(M\setminus \cup_{I\cup J}A_i)\cup_{I\cap J} B_i$, the following identity holds :
$$
\sum_{I \subset \lbrace 2, \ldots , k \rbrace} (-1)^{\card(I)} p_1 (\bar \tau ^I,\bar \tau ^{I\cup\lbrace 1 \rbrace})=0.
$$
\end{corollary}

\subsection[Lemmas for the proof of Theorem~\ref{thm_D2nd}]{Lemmas for the proof of Theorem~\ref{thm_D2nd}}
\begin{lemma} \label{lem_Xdsection}
If $X$ is a combing of a compact oriented 3-manifold $M$ and if $\partial A$ is the connected boundary of a submanifold of $M$ of dimension 3, then the normal bundle $X^\perp_{|\partial A}$ admits a nonvanishing section.
\end{lemma}
\begin{proof}
Parallelize $M$ so that $X$ induces a map $X_{|\partial A} : \partial A \rightarrow \b S^2$. This map must have degree 0 since $X$ extends this map to $A$ (so that $(X_{|\partial A})_* : H_2(\partial A;\b Q)\rightarrow H_2(\b S^2;\b Q)$ factors through the inclusion $H_2(\partial A ; \b Q) \rightarrow H_2(A;\b Q)$, which is zero). It follows that $X_{|\partial A}$ is homotopic to the map $(m \in\partial A \mapsto e_1 \in \b S^2)$ whose normal bundle admits a nonvanishing section.
\end{proof}

\begin{lemma} \label{ind}
Let $(M,X)$ be a compact oriented 3-manifold equipped with a combing, let $(\sfrac{B}{A},X_B)$ be an LP$_\b Q$-surgery in $(M,X)$ and let $\sigma$ be a nonvanishing section of $X^\perp_{|\partial A}$. Let $P$ stand for Poincaré duality isomorphisms and recall the sequence of isomorphisms induced by the inclusions $i^{A_i}$ and $i^{B_i}$
$$
H_1(A;\b Q) \stackrel{i^{A}_*}{\longleftarrow} \frac{H_1(\partial A;\b Q)}{\go L_{A}} = \frac{H_1(\partial B;\b Q)}{\go L_{B}}  \stackrel{i^{B}_*}{\longrightarrow} H_1(B;\b Q).
$$
The class $ \left(i^A_* \circ {(i^B_*)}^{-1}\left(P(e_2^B(X_B^\perp, \sigma))\right) - P(e_2^A(X_{|A}^\perp, \sigma))\right)$ in $H_1( A ; \b Q)$ is independent of the choice of the section $\sigma$. 
\end{lemma}
\begin{proof}
Let us drop the inclusions $i^B_*$ and $i^A_*$ from the notation. According to Proposition~\ref{prop_euler}, the class $P(e_2^A(X_{|A}^\perp, \sigma))$ verifies
$$
[ X(A) - (-X)(A) +\l H_{X,\sigma}^{-X}(\partial A \times [0,1]) ] = [ P(e_2^A(X_{|A}^\perp, \sigma)) \times \b S^2 ] \mbox{ \ in \ } H_1(UA ;\b Q).
$$
It follows that, for another choice $\sigma'$ of section of $X^\perp_{|\partial A}$, 
$$
\begin{aligned}
[P(e_2^B(X_B^\perp, \sigma)) \times \b S^2 & - P(e_2^B(X_B^\perp, \sigma')) \times \b S^2 ] \\
&= [\l H_{X,\sigma}^{-X}(\partial A\times [0,1])-\l H_{X,\sigma'}^{-X}(\partial A\times [0,1])] \\
&=[P(e_2^A(X_{|A}^\perp, \sigma)) \times \b S^2 - P(e_2^A(X_{|A}^\perp, \sigma')) \times \b S^2 ].
\end{aligned}
$$
\end{proof}

\begin{lemma} \label{zero}
Let $(M,X)$ be a compact oriented 3-manifold equipped with a combing and let $(\sfrac{B}{A},X_B)$ be an LP$_\b Q$-surgery in $(M,X)$. If $(X,\sigma)$ is a torsion combing then $(X(\sfrac{B}{A}),\sigma)$ is a torsion combing if and only if 
$$
i_*^{A} \circ (i_*^B)^{-1} \big( P(e_2^B({X_{B}}^\perp, \zeta))\big) - P(e_2^A(X_{|A}^\perp, \zeta))  = 0 \mbox{ \ in $H_1(M; \b Q)$}
$$
for some nonvanishing section $\zeta$ of $X^\perp_{|\partial A}$.
\end{lemma}
\begin{proof}
By definition, we have
$$
\begin{aligned}
 P(e_2(X^\perp,\sigma)) &= P(e_2^A(X_{|A}^\perp, \zeta)) + P(e_2^{M \setminus \mathring A}(X^\perp,\sigma \cup \zeta)) \\
 P(e_2({X(\sfrac{B}{A})}^\perp,\sigma)) &= P(e_2^B(X_{B}^\perp, \zeta)) + P(e_2^{M \setminus \mathring A}(X'^\perp,\sigma \cup \zeta)) \\
\end{aligned}
$$
where $\zeta$ is any nonvanishing section of $X^\perp_{|\partial A}$. So, it follows that, using appropriate identifications,
$$
P(e_2(X(\sfrac{B}{A})^\perp,\sigma)) - P(e_2({X}^\perp,\sigma)) = P(e_2^B(X_{B}^\perp, \zeta)) - P(e_2^A(X_{|A}^\perp, \zeta)).
$$
If $X$ is a torsion combing, then $P(e_2({X}^\perp,\sigma))$ is rationally null-homologous in $M$. Hence, $X(\sfrac{B}{A})$ is a torsion combing if and only if 
$$
i_*^{A} \circ (i_*^B)^{-1} \big( P(e_2^B({X_{B}}^\perp, \zeta))\big) - P(e_2^A(X_{|A}^\perp, \zeta))  = 0 \mbox{ \ in $H_1(M; \b Q)$}.
$$
\end{proof}

\begin{lemma} \label{norm}
Let $(M,X)$ be a compact oriented 3-manifold equipped with a combing. Let $\lbrace (\sfrac{B_i}{A_i} , X_{B_i}) \rbrace_{i \in \lbrace 1,\ldots,k\rbrace}$ be a family of disjoint LP$_\b Q$-surgeries in $(M,X)$, where $k \in \b N\setminus \lbrace 0 \rbrace$. For all $I\subset \lbrace 1, \ldots, k \rbrace$, let $M_I=M(\lbrace \sfrac{B_i}{A_i} \rbrace_{i \in I})$ and $X^I=X(\lbrace \sfrac{B_i}{A_i}\rbrace_{i \in I})$. There exists a family of pseudo-parallelizations $\lbrace\bar \tau^I\rbrace_{ I \subset \lbrace 1,\ldots,k\rbrace}$ of the $\lbrace M_I \rbrace_{I \subset \lbrace 1, \ldots , k \rbrace}$ such that :
\begin{enumerate}[(i)]
\item the third vector of $\bar\tau=\bar\tau^\emptyset$ coincides with $X$ on $\cup_{i=1}^k  \partial A_i$,
\item for all $I \subset \lbrace 1 , \ldots , k \rbrace$, $\bar\tau^I$ is compatible with $X^I$,
\item for all $I \subset \lbrace 1,\ldots,k\rbrace$, if $\gamma^I$ denotes the link of $\bar \tau_I$, then $N(\gamma^I) \cap \left(\cup_{i=1}^k \partial A_i \right) = \emptyset$,
\item for all $I, J \subset \lbrace 1,\ldots,k\rbrace$, $\bar\tau^I$ and $\bar\tau^J$ coincide on $(M\setminus \cup_{i\in I\cup J} A_i)\cup_{i\in I\cap J}B_i$,
\item there exist links $L^\pm_{A_i}$ in $A_i$, $L^\pm_{B_i}$ in $B_i$ and $L^\pm_{ext}$ in $M \setminus \cup_{i=1}^k \mathring A_i$ such that, for all subset $I \subset \lbrace 1, \ldots , k \rbrace$~:
$$
\begin{aligned}
& 2 \cdot L_{\bar\tau^I=X^I} = L_{{E^d}^I=X^I} + L_{{E^g}^I=X^I} = L^+_{ext} + \sum_{i \in I} L^+_{B_i} + \sum_{i \in \lbrace 1,\ldots , k \rbrace \setminus I} L^+_{A_i} , \\
& 2 \cdot L_{\bar\tau^I=-X^I} = L_{{E^d}^I=-X^I} + L_{{E^g}^I=-X^I} = L^-_{ext} + \sum_{i \in I} L^-_{B_i} + \sum_{i \in \lbrace 1,\ldots , k  \rbrace \setminus I} L^-_{A_i},
\end{aligned}
$$
where ${E^d}^I$ and ${E^g}^I$ are the Siamese sections of $\bar\tau^I$.
\end{enumerate}
\end{lemma}
\begin{proof}
Let $\l C$ denote a collar of $\cup_{i=1}^k  \partial A_i$. Using Lemma~\ref{lem_Xdsection}, construct a trivialization $\tau_e$ of $\cup_{i=1}^k TM_{|\l C}$ so that its third vector coincides with $X$ on $\l C$. Then use Lemma~\ref{lem_extendpparallelization} to extend $\tau_e$ as pseudo-parallelizations of the $\lbrace A_i \rbrace_{i\in \lbrace 1,\ldots , k \rbrace }$, of the $\lbrace B_i \rbrace_{i\in \lbrace 1,\ldots , k \rbrace}$ and of $M\setminus(\cup_{i=1}^k \mathring A_i)$. Finally, use these pseudo-parallelizations to construct the pseudo-parallelizations of the 3-manifolds $\lbrace M_I \rbrace_{I \subset \lbrace 1,\ldots , k \rbrace}$ as in the statement.
\end{proof}

\begin{lemma} \label{inter}
In the context of Lemma~\ref{norm}, using the sequence of isomorphisms induced by the inclusions $i^{A_i}$ and $i^{B_i}$
$$
H_1(A_i;\b Q) \stackrel{i^{A_i}_*}{\longleftarrow} \frac{H_1(\partial A_i;\b Q)}{\go L_{A_i}} = \frac{H_1(\partial B_i;\b Q)}{\go L_{B_i}}  \stackrel{i^{B_i}_*}{\longrightarrow} H_1(B_i;\b Q),
$$
for all $i \in \lbrace 1 , \ldots , k \rbrace$, we have that, in $H_1(A_i ; \b Q)$,
$$
\left[i_{\ast}^{A_i} \circ \left(i_{\ast}^{B_i}\right)^{-1}\left( L_{B_i}^{\pm}\right)- L_{A_i}^\pm\right] = \pm \left( i^{A_i}_* \circ(i^{B_i}_*)^{-1}\big(P(e_2^{B_i}(X_{B_i}^\perp, \sigma_i))\big) - P(e_2^{A_i}(X_{|A_i}^\perp, \sigma_i)) \right)
$$
where $\sigma_i$ is any nonvanishing section of $X^\perp_{|\partial A_i}$.
\end{lemma}
\begin{proof}
Let us drop the inclusions $i^B_*$ and $i^A_*$ from the notation. Let $i \in \lbrace 1,\ldots,k \rbrace$. According to Lemma~\ref{zero}, it is enough to prove the statement for a particular non vanishing section $\sigma_i$ of $X^\perp_{|\partial A_i}$. Recall that $A_i$ is equipped with a combing $X_{|A_i}$ and a pseudo-parallelization \linebreak $\bar\tau_{|A_i}=(N(\gamma\cap A_i); {\tau_e}_{|A_i}, {\tau_d}_{|A_i}, {\tau_g}_{|A_i})$ such that $X_{|\partial A_i}$ coincides with ${E_3^e}_{|\partial A_i}$ where $\tau_e=(E_1^e,E_2^e,E_3^e)$. Furthermore, 
$$
L^+_{A_i} = 2\cdot L_{\bar\tau_{|A_i}= X_{|A_i}} \mbox{ \ and  \ } L^-_{A_i} = 2\cdot L_{\bar\tau_{|A_i}= -X_{|A_i}}.
$$
Construct a pseudo-parallelization $\check\tau = (N(\check\gamma) ; \check \tau_e ,\check \tau_d ,\check \tau_g)$ of $A_i$ by modifying $\bar\tau$ as follows so that $\check\tau$ and $X$ coincide on $\partial A_i$. Consider a collar $\l C=[0,1]\times \partial A_i$ of $\partial A_i$ such that $\lbrace 1 \rbrace \times \partial A_i = \partial A_i$ and $\l C \cap \gamma = \emptyset$. Without loss of generality, assume that $X_{|A_i}$ coincides with $E_3^e$ on the collar $\l C$. Let $\check\tau$ coincide with $\bar\tau_{|A_i}$ on $\overline{A_i \setminus \l C}$. End the construction of $\check\tau$ by requiring 
$$
\forall (s,b) \in \l C = [0,1] \times \partial A_i, \ \forall v \in \b S^2 \ : \ \check\tau_e( (s,b),v) = \tau_e \left((s,b), R_{e_2, \frac{-\pi s}{2}}(v)\right).
$$
Note that $\check\tau$ and $X_{|A_i}$ are compatible and that
$$
L^+_{A_i} = 2\cdot L_{\check\tau= X_{|A_i}} \mbox{ \ and  \ } L^-_{A_i} = 2\cdot L_{\check\tau= -X_{|A_i}}.
$$
Using $\check\tau$ and Proposition~\ref{prop_linksinhomologyI}, we get 
$$
\begin{aligned}
[L^+_{A_i}] &= P(e_2^{A_i}(X^\perp_{|A_i}, {\check E^e}_{2|\partial A_i})) + \frac{1}{2}  P(e_2^{A_i}({\check{E}}^{d \perp}_{1}, {\check E^e}_{2|\partial A_i})) + \frac{1}{2}  P(e_2^{A_i}({\check{E}}^{g \perp}_{1}, {\check E^e}_{2|\partial A_i})) \\
[L^-_{A_i}] &= - P(e_2^{A_i}(X^\perp_{|A_i}, {\check E^e}_{2|\partial A_i})) + \frac{1}{2}  P(e_2^{A_i}({\check{E}}^{d \perp}_{1}, {\check E^e}_{2|\partial A_i})) + \frac{1}{2}  P(e_2^{A_i}({\check{E}}^{g \perp}_{1}, {\check E^e}_{2|\partial A_i}))
\end{aligned}
$$
where $\check\tau_e=(\check E_1^e,\check E_2^e,\check E_3^e)$ and where ${\check E_1}^d$ and ${\check E_1}^g$ are the Siamese sections of $\check\tau$. By construction, it follows that
$$
\begin{aligned}
[L^+_{A_i}] &= P(e_2^{A_i}(X^\perp_{|A_i}, {E_2^e}_{|\partial A_i})) + \frac{1}{2}  P(e_2^{A_i}({E_1^d}^\perp_{|A_i}, {E_2^e}_{|\partial A_i})) + \frac{1}{2}  P(e_2^{A_i}({E_1^g}^\perp_{|A_i}, {E_2^e}_{|\partial A_i})) \\
[L^-_{A_i}] &= -P(e_2^{A_i}(X^\perp_{|A_i}, {E_2^e}_{|\partial A_i})) + \frac{1}{2}  P(e_2^{A_i}({E_1^d}^\perp_{|A_i}, {E_2^e}_{|\partial A_i})) + \frac{1}{2}  P(e_2^{A_i}({E_1^g}^\perp_{|A_i}, {E_2^e}_{|\partial A_i}))
\end{aligned}
$$
where $E_1^d$ and $E_1^g$ are the Siamese sections of $\bar\tau$. Using the same method, we also get that
$$
\begin{aligned}
[L^+_{B_i}] &= P(e_2^{B_i}(X^\perp_{B_i}, {E_2^e}_{|\partial A_i})) + \frac{1}{2}  P(e_2^{B_i}({E_1^d}^\perp_{|B_i}, {E_2^e}_{|\partial A_i})) + \frac{1}{2}  P(e_2^{B_i}({E_1^g}^\perp_{|B_i}, {E_2^e}_{|\partial A_i})) \\
[L^-_{B_i}] &= -P(e_2^{B_i}(X^\perp_{B_i}, {E_2^e}_{|\partial A_i})) + \frac{1}{2}  P(e_2^{B_i}({E_1^d}^\perp_{|B_i}, {E_2^e}_{|\partial A_i})) + \frac{1}{2}  P(e_2^{B_i}({E_1^g}^\perp_{|B_i}, {E_2^e}_{|\partial A_i})).
\end{aligned}
$$
Conclude with Lemma~\ref{simppara}.
\end{proof}

\subsection[Variation formula for torsion combings]{Variation formula for torsion combings}
\begin{proof}[Proof of Theorem~\ref{thm_D2nd}]
Let $(M,X)$ be a compact oriented 3-manifold equipped with a combing. Let $\lbrace (\sfrac{B_i}{A_i},X_{B_i}) \rbrace_{i\in \lbrace 1, 2\rbrace}$ be two disjoint LP$_\b Q$-surgeries in $(M,X)$ and assume that, for all subset $I\subset\lbrace 1,2 \rbrace$, $X^I=X(\lbrace\sfrac{B_i}{A_i}\rbrace_{i \in I})$ is a torsion combing of the 3-manifold $M_I = M(\lbrace \sfrac{B_i}{A_i} \rbrace_{i \in I})$. Note that, for all $I,J\subset \lbrace 1,2 \rbrace$, $X^I$ and $X^J$ coincide on $(M\setminus \cup_{i\in I\cup J} A_i)\cup_{i\in I\cap J}B_i$. Finally, let $\lbrace \bar\tau^I \rbrace_{I\subset \lbrace 1,2\rbrace}$ be a family of pseudo-parallelizations as in Lemma~\ref{norm}, let ${E_1^d}^I$ and ${E_1^g}^I$ denote the Siamese sections of $\bar\tau^I$ for all $I\subset \lbrace 1,2 \rbrace$ and let $L_I$ stand for $L_{{E_1^d}^I=-{E_1^g}^I}$. Using Corollary~\ref{cor_FTppara}, we have
$$
\begin{aligned}
p_1 \left([X^{\lbrace 2 \rbrace}],[X^{\lbrace 1,2 \rbrace}]\right) - p_1\left([X],[X^{\lbrace 1 \rbrace}]\right)
&=  p_1 \left([X^{\lbrace 2 \rbrace}],[X^{\lbrace 1,2 \rbrace}]\right)- p_1 \left(\bar\tau^{\lbrace 2 \rbrace},\bar\tau^{\lbrace 1,2 \rbrace}\right) \\
&  - p_1\left([X],[X^{\lbrace 1 \rbrace}]\right) + p_1\left(\bar\tau,\bar\tau^{\lbrace 1 \rbrace}\right) \\
\end{aligned}
$$
which, using Lemma~\ref{lem1} and Theorem~\ref{thm_defp1Xb}, reads
$$
\begin{aligned}
& 4 \cdot lk_{M_{\lbrace1,2\rbrace}}(L_{\bar\tau^{\lbrace 1,2 \rbrace}=X^{\lbrace 1,2 \rbrace}} \ , \ L_{\bar\tau^{\lbrace 1,2 \rbrace}=-X^{\lbrace 1,2 \rbrace}}) - 4 \cdot lk_{M_{\lbrace 2 \rbrace}}(L_{\bar\tau^{\lbrace 2 \rbrace}=X^{\lbrace 2 \rbrace}} \ , \ L_{\bar\tau^{\lbrace 2 \rbrace}=-X^{\lbrace 2 \rbrace}}) \\
& - lk_{\b S^2} \left(e_1-(-e_1) \ , \ P_{\b S^2} \circ (\tau^{\lbrace 1,2 \rbrace}_d)^{-1} \circ X^{\lbrace 1,2 \rbrace}(L_{\lbrace 1,2 \rbrace}) - P_{\b S^2} \circ (\tau^{\lbrace 2 \rbrace}_d)^{-1} \circ X^{\lbrace 2 \rbrace}(L_{\lbrace 2 \rbrace})\right)\\
&- 4 \cdot lk_{M_{\lbrace 1 \rbrace}}(L_{\bar\tau^{\lbrace 1 \rbrace}=X^{\lbrace 1 \rbrace}} \ , \ L_{\bar\tau^{\lbrace 1 \rbrace}=-X^{\lbrace 1 \rbrace}}) + 4 \cdot lk_{M}(L_{\bar\tau=X} \ , \ L_{\bar\tau=-X}) \\
& + lk_{\b S^2} \left(e_1-(-e_1) \ , \ P_{\b S^2} \circ (\tau^{\lbrace 1 \rbrace}_d)^{-1} \circ X^{\lbrace 1 \rbrace}(L_{\lbrace 1 \rbrace}) - P_{\b S^2} \circ (\tau_d)^{-1} \circ X(L) \right).
 \end{aligned}
$$
This can further be reduced to the following by using Lemma~\ref{norm},
$$
\begin{aligned}
& lk_M \left( L^+_{ext} + L^+_{A_1} + L^+_{A_2}, \ L^-_{ext} +  L^-_{A_1} +  L^-_{A_2} \right) \\
&- lk_{M_{\lbrace 1 \rbrace}} \left( L^+_{ext} + L^+_{B_1} +  L^+_{A_2}, \ L^-_{ext} + L^-_{B_1} +  L^-_{A_2} \right) \\
&- lk_{M_{\lbrace 2 \rbrace}} \left( L^+_{ext} +  L^+_{A_1} + L^+_{B_2}, \ L^-_{ext} +  L^-_{A_1} +  L^-_{B_2} \right) \\
&+ lk_{M_{\lbrace 1,2 \rbrace}} \left( L^+_{ext} +  L^+_{B_1} +  L^+_{B_2}, \ L^-_{ext} +  L^-_{B_1} + L^-_{B_2} \right).
 \end{aligned}
$$

In order to compute these linking numbers, let us construct specific 2-chains. Let us introduce a more convenient set of notations. For all $i \neq j \in \lbrace 1,2 \rbrace$, let
$$
 L^\pm_{i}=L^\pm_{A_i}, \  L^\pm_{ij}= L^\pm_{A_i} + L^\pm_{A_j}, \  L^\pm_{eij}= L^\pm_{ext} + L^\pm_{A_i} + L^\pm_{A_j}.
$$
Set also similar notations with primed indices where a primed index $i'$, $i \in \lbrace 1,2 \rbrace$, indicates that $L^\pm_{A_i}$ should be replaced by $L^\pm_{B_i}$. For instance, $L^\pm_{i'}=L^\pm_{B_i}$, $L^\pm_{i,j'}= L^\pm_{A_i} + L^\pm_{B_j}$, \textit{etc}. Using these notations, $p_1 \left([X^{\lbrace 2 \rbrace}],[X^{\lbrace 1,2 \rbrace}]\right) - p_1\left([X],[X^{\lbrace 1 \rbrace}]\right)$ reads :
$$
\begin{aligned}
 lk_M(L_{e12}^+,L_{e12}^-)- lk_{M_{\lbrace 1 \rbrace}}(L_{e1'2}^+,L_{e1'2}^-) - lk_{M_{\lbrace 2 \rbrace}}(L_{e12'}^+,L_{e12'}^-) + lk_{M_{\lbrace 1,2 \rbrace}} (L_{e1'2'}^+,L_{e1'2'}^-) .
\end{aligned}
$$
Recall from Lemma~\ref{nulexcep} that there exist rational 2-chains $\Sigma_{e12}^\pm$ of $M$ which are bounded by the links $L^\pm_{e12}$. Similarly, there exist rational two chains $\Sigma^\pm_{e1'2'}$ bounded by $L^\pm_{e1'2'}$ in $M_{\lbrace 1,2\rbrace}$. Note that, for all $i \in \lbrace 1,2 \rbrace$, the 2-chains $\Sigma_{e12}^\pm\cap A_i$ are cobordisms between the $L^\pm_i$ and 1-chains $\ell^\pm_i$ in $\partial A_i$. Similarly, for all $i \in \lbrace 1,2 \rbrace$, the 2-chains $\Sigma_{e1'2'}^\pm\cap B_i$ are cobordisms between the $L^\pm_{i'}$ and 1-chains $\ell^\pm_{i'}$ in $\partial B_i$. Furthermore, according to Lemma~\ref{inter}, for all $i\in \lbrace 1,2 \rbrace$ and for any nonvanishing section $\sigma_i$ of $X^\perp_{|\partial A_i}$, in $H_1(\partial A_i ; \b Q)/\go L_{A_i}$ :
$$
[\ell^\pm_{i'}-\ell^\pm_{i}]= \pm \big((i_*^{B_i})^{-1}(P(e_2^{B_i}(X_{B_i}^\perp, \sigma_i))) - (i_*^{A_i})^{-1}(P(e_2^{A_i}(X_{|A_i}^\perp, \sigma_i)))\big).
$$
So, according to Lemma~\ref{zero}, for all $I \subset \lbrace 1,2 \rbrace$, there exists a 2-chain $S_{\sfrac{B_i}{A_i}}^I$ in $M_I$ which is bounded by $\ell^+_{i'}-\ell^+_i$. Finally, since $(\sfrac{B_1}{A_1})$ and $(\sfrac{B_2}{A_2})$ are LP$_\b Q$-surgeries, we can construct these chains so that
$$
\begin{aligned}
S_{\sfrac{B_1}{A_1}}^{\lbrace 1\rbrace} \cap (M_{\lbrace 1 \rbrace}\minusens \mathring A_2) &= S_{\sfrac{B_1}{A_1}}^{\lbrace 1,2 \rbrace} \cap (M_{\lbrace 1,2 \rbrace}\minusens \mathring B_2), \\
S_{\sfrac{B_2}{A_2}}^{\lbrace 2\rbrace} \cap (M_{\lbrace 2 \rbrace}\minusens \mathring A_1) &= S_{\sfrac{B_2}{A_2}}^{\lbrace 1,2 \rbrace} \cap (M_{\lbrace 1,2 \rbrace})\minusens \mathring B_1). 
\end{aligned}
$$

Let us now return to the computation of $p_1 \left([X^{\lbrace 2 \rbrace}],[X^{\lbrace 1,2 \rbrace}]\right) - p_1\left([X],[X^{\lbrace 1 \rbrace}]\right)$. Using the 2-chains we constructed, we have :
$$
\begin{aligned}
  lk_M(L_{e12}^+,L_{e12}^-) &=  \langle \Sigma^+_{e12}, L_{e12}^- \rangle \\
  lk_{M_{\lbrace 1 \rbrace}}(L_{e1'2}^+,L_{e1'2}^-) &= \langle \Sigma^+_{e12} \cap (M \setminus \mathring A_1) + S_{\sfrac{B_1}{A_1}}^{\lbrace 1 \rbrace} + \Sigma^+_{e1'2'}\cap B_1, L_{e1'2}^- \rangle \\
  lk_{M_{\lbrace 2 \rbrace}}(L_{e12'}^+,L_{e12'}^-) &= \langle \Sigma^+_{e12}\cap (M \setminus \mathring A_2) + S_{\sfrac{B_2}{A_2}}^{\lbrace 2 \rbrace} + \Sigma^+_{e1'2'}\cap B_2 , L_{e12'}^- \rangle \\
  lk_{M_{\lbrace 1,2 \rbrace}} (L_{e1'2'}^+,L_{e1'2'}^-)&=\langle \Sigma^+_{e12}\cap(M\setminus(\mathring A_1\hspace{-1mm}\cup\hspace{-1mm}\mathring A_2)) +S_{\sfrac{B_1}{A_1}}^{\lbrace 1 , 2 \rbrace} \\
  &\hspace{1cm}+S_{\sfrac{B_2}{A_2}}^{\lbrace 1 , 2 \rbrace} + \Sigma^+_{e1'2'}\cap(B_1\cup B_2) , L_{e1'2'}^- \rangle. \\
\end{aligned}
$$
So, the contribution of the intersections in $M \minusens (\mathring A_1\cup \mathring A_2)$ is zero since it reads :
$$
\begin{aligned}
& \langle \Sigma^+_{e12} \cap (M \minusens (\mathring A_1 \cup \mathring A_2)) , L_{e}^- \rangle_{M\setminus (\mathring A_1\cup \mathring A_2)} \\
& -\langle \Sigma^+_{e12} \cap (M \minusens (\mathring A_1 \cup \mathring A_2)) + S_{\sfrac{B_1}{A_1}}^{\lbrace 1 \rbrace}\cap (M \minusens (\mathring A_1 \cup \mathring A_2)) , L_{e}^- \rangle_{M \setminus (\mathring A_1\cup \mathring A_2)} \\
& -\langle \Sigma^+_{e12} \cap (M \minusens (\mathring A_1 \cup \mathring A_2)) + S_{\sfrac{B_2}{A_2}}^{\lbrace 2 \rbrace}\cap (M \minusens (\mathring A_1 \cup \mathring A_2)), L_{e}^- \rangle_{M\setminus (\mathring A_1\cup \mathring A_2)} \\
& +\langle \Sigma^+_{e12} \cap (M \minusens (\mathring A_1 \cup \mathring A_2)) + S_{\sfrac{B_1}{A_1}}^{\lbrace 1 , 2 \rbrace}\cap (M \minusens (\mathring A_1 \cup \mathring A_2)) \\
&\hspace{5cm}+ S_{\sfrac{B_2}{A_2}}^{\lbrace 1 , 2 \rbrace}\cap (M \minusens (\mathring A_1 \cup \mathring A_2)) , L_{e}^- \rangle_{M\setminus (\mathring A_1\cup \mathring A_2)} .
\end{aligned}
$$
The contribution in $A_1$ is
$$
\begin{aligned}
  \langle \Sigma^+_{e12} \cap A_1 , L_{1}^- \rangle_{A_1}  - \langle\Sigma^+_{e12} \cap A_1 + S_{\sfrac{B_2}{A_2}}^{\lbrace 2 \rbrace}\cap A_1 , L_{1}^- \rangle_{A_1}
  = - \langle S_{\sfrac{B_2}{A_2}}^{\lbrace 2 \rbrace}\cap A_1, L_{1}^-  \rangle_{A_1}.
\end{aligned}
$$
In $A_2$, we similarly get $- \langle S_{\sfrac{B_1}{A_1}}^{\lbrace 1 \rbrace}\cap A_2, L_{2}^-  \rangle_{A_2}$. The contribution in $B_1$ is
$$
\begin{aligned}
 &-\hspace{-1mm}\langle S_{\sfrac{B_1}{A_1}}^{\lbrace 1 \rbrace}\cap B_1  \hspace{-1mm}+\hspace{-1mm} \Sigma^+_{e1'2'}\cap B_1 , L_{1'}^- \rangle_{B_1}  \hspace{-1mm}+\hspace{-1mm} \langle S_{\sfrac{B_1}{A_1}}^{\lbrace 1 , 2 \rbrace}\cap B_1  \hspace{-1mm} +\hspace{-1mm}  S_{\sfrac{B_2}{A_2}}^{\lbrace 1 , 2 \rbrace}\cap B_1  \hspace{-1mm} +\hspace{-1mm}  \Sigma^+_{e1'2'}\cap B_1 , L_{1'}^- \rangle_{B_1} \\
 &= \langle S_{\sfrac{B_2}{A_2}}^{\lbrace 1 , 2 \rbrace}\cap B_1  , L_{1'}^- \rangle_{B_1}
\end{aligned}
$$
and, in $B_2$, we get $\langle S_{\sfrac{B_1}{A_1}}^{\lbrace 1 , 2 \rbrace}\cap B_2 , L_{2'}^- \rangle_{B_2}$. Eventually, $L_{\lbrace X^I \rbrace}(\sfrac{B_i}{A_i}) = i_*^{A_i}([\ell_{i'}^+ - \ell_{i}^+])$ for $i \in \lbrace 1,2\rbrace$. Moreover, recall that $[\ell^+_{i'}-\ell^+_{i}]=-[\ell^-_{i'}-\ell^-_{i}]$ in $H_1(\partial A_i ; \b Q)/\go L_{A_i}$, and complete the computations~:
$$
\begin{aligned}
p_1 & \left([X^{\lbrace 2 \rbrace}],[X^{\lbrace 1,2 \rbrace}]\right) - p_1\left([X],[X^{\lbrace 1 \rbrace}]\right) \\
&= \langle S_{\sfrac{B_2}{A_2}}^{\lbrace 1 , 2 \rbrace}\cap B_1 , L_{1'}^- \rangle_{B_1} + \langle S_{\sfrac{B_1}{A_1}}^{\lbrace 1 , 2 \rbrace}\cap B_2 , L_{2'}^- \rangle_{B_2} \\
&- \langle S_{\sfrac{B_2}{A_2}}^{\lbrace 2 \rbrace}\cap A_1, L_{1}^-  \rangle_{A_1} - \langle S_{\sfrac{B_1}{A_1}}^{\lbrace 1 \rbrace} \cap A_2, L_{2}^-  \rangle_{A_2} \\ 
&= \langle S_{\sfrac{B_2}{A_2}}^{\lbrace 1 , 2 \rbrace} , \ell_{1'}^- \rangle_{M_{\lbrace 1,2 \rbrace}} + \langle S_{\sfrac{B_1}{A_1}}^{\lbrace 1 , 2 \rbrace} , \ell_{2'}^- \rangle_{M_{\lbrace 1,2 \rbrace}} \\
&- \langle S_{\sfrac{B_2}{A_2}}^{\lbrace 2 \rbrace}, \ell_{1}^-  \rangle_{M_{\lbrace 2 \rbrace}} - \langle S_{\sfrac{B_1}{A_1}}^{\lbrace 1 \rbrace} , \ell_{2}^-  \rangle_{M_{\lbrace 1 \rbrace}} \\ 
&= \langle S_{\sfrac{B_2}{A_2}}^{\lbrace 1 , 2 \rbrace} , \ell_{1'}^- - \ell_{1}^-  \rangle_{M_{\lbrace 1,2 \rbrace}} + \langle S_{\sfrac{B_1}{A_1}}^{\lbrace 1 , 2 \rbrace} , \ell_{2'}^- - \ell_{2}^- \rangle_{M_{\lbrace 1,2 \rbrace}} \\
&=- 2 \cdot lk_M \left(L_{\lbrace X^I \rbrace}(\sfrac{B_1}{A_1}) \ , \ L_{\lbrace X^I \rbrace}(\sfrac{B_2}{A_2}) \right).
\end{aligned}
$$
\iffalse
&=   lk_{M_{\lbrace 1,2 \rbrace}} \left(L_{\lbrace X^I \rbrace}(\sfrac{B_2}{A_2}) \ , \ - L_{\lbrace X^I \rbrace}(\sfrac{B_1}{A_1})\right) \\
&+  \ lk_{M_{\lbrace 1,2 \rbrace}}\left(L_{\lbrace X^I \rbrace}(\sfrac{B_1}{A_1}) \ , \ - L_{\lbrace X^I \rbrace}(\sfrac{B_2}{A_2})\right) \\
&=- 2 \cdot lk_{M_{\lbrace 1,2 \rbrace}} \left(L_{\lbrace X^I \rbrace}(\sfrac{B_1}{A_1}) \ , \ L_{\lbrace X^I \rbrace}(\sfrac{B_2}{A_2}) \right) \\
\fi
\end{proof}

\newpage

\nocite{hirzebruch,KM,pontrjagin,turaev,lickorish,rolfsen,MR1030042,MR1189008,MR1712769}

\bibliographystyle{alpha}
\bibliography{bib}

\end{document}